\documentclass[final,onefignum,onetabnum]{siamart171218}

\newcommand{\BD}[1]{B_{#1}}

\newcommand{\tdi}{\mathbf{\td{i}}}
\newcommand{\lb}{\left(}
\newcommand{\rb}{\right)}
\newcommand{\eps}{\epsilon}

\newcommand{\td}{\tilde}
\newcommand{\E}{\mathbb{E}}
\newcommand{\R}{\mathbb{R}}
\newcommand{\I}{\mathcal{I}}
\renewcommand{\L}{\mathcal{L}}
\newcommand{\D}{D}
\newcommand{\X}{\mathcal{X}}
\newcommand{\Dx}{D_{x}}
\renewcommand{\DH}{D_{H}}
\newcommand{\DF}{D_{F}}

\newcommand{\tto}{\td{\td{O}}}
\renewcommand{\succeq}{\ge}

\newcommand{\breg}[3]{#1(#2) - #1(#3) - \langle\nabla #1(#3), #2 - #3\rangle}
\newcommand{\ik}[3]{\langle\nabla #1(#3), #2 - #3 \rangle}

\newcommand{\la}{\left\langle}
\newcommand{\ra}{\right\rangle}
\renewcommand{\H}{\mathcal{H}}
\newcommand{\comp}{C_{\mathrm{comp}}}

\newcommand{\xj}{x^{(j)}}
\newcommand{\yj}{y^{(j)}}
\newcommand{\nuj}{\nu^{(j)}}
\newcommand{\zetaj}{\zeta^{(j)}}
\newcommand{\Qj}{Q^{(j)}}
\newcommand{\Mj}{M^{(j)}}
\newcommand{\tdMj}{\td{M}^{(j)}}
\newcommand{\Wj}{W_{j}}

\newcommand{\sI}{\td{\mathcal{I}}^{(j)}}
\newcommand{\etaj}{\eta_{j}}

\newcommand{\Bj}{B_{j}}
\newcommand{\bj}{b_{j}}
\newcommand{\mj}{m_{j}}
\newcommand{\mjj}{m_{j-1}}
\newcommand{\ej}{e_{j}}
\newcommand{\Ej}{\E_{j}}
\newcommand{\Ij}{\I_{j}}
\newcommand{\Nj}{N_{j}}
\newcommand{\Lj}{\mathcal{L}_{j}}
\newcommand{\Ljj}{\mathcal{L}_{j-1}}

\newcommand{\citep}{\cite}
\newcommand{\citet}{\cite}
\DeclareMathOperator*{\argmin}{arg\,min}
\DeclareMathOperator*{\argmax}{arg\,max}

\usepackage{lipsum}
\usepackage{amsfonts}
\usepackage{graphicx}
\usepackage{epstopdf}
\usepackage{enumerate}
\usepackage{algorithm}
\usepackage{amsmath,amssymb}
\usepackage{graphics}
\usepackage{algpseudocode}
\ifpdf
  \DeclareGraphicsExtensions{.eps,.pdf,.png,.jpg}
\else
  \DeclareGraphicsExtensions{.eps}
\fi

\newsiamremark{remark}{Remark}
\newsiamremark{hypothesis}{Hypothesis}
\crefname{hypothesis}{Hypothesis}{Hypotheses}
\newsiamthm{claim}{Claim}

\headers{Adaptivity of SCSG}{L. Lei and M. I. Jordan}

\title{On the Adaptivity of Stochastic Gradient-Based Optimization}

\author{Lihua Lei\thanks{Department of Statistics, Stanford University
  (\email{lihualei@stanford.edu}). This work was completed while he was a graduate student in the Department of Statistics at UC Berkeley.}
\and Michael I. Jordan\thanks{Computer Science Division \& Department of Statistics, University of California, Berkeley
  (\email{jordan@cs.berkeley.edu}).}}

\usepackage{amsopn}

\ifpdf
\hypersetup{
  pdftitle={On the Adaptivity of Stochastic Gradient-Based Optimization},
  pdfauthor={L. Lei and M. I. Jordan}
}
\fi

\begin{document}

\maketitle

% REQUIRED
\begin{abstract}
Stochastic-gradient-based optimization has been a core enabling methodology in applications to large-scale problems in machine learning and related areas. Despite the progress, the gap between theory and practice remains significant, with theoreticians pursuing mathematical optimality at a cost of obtaining specialized procedures in different regimes (e.g., modulus of strong convexity, magnitude of target accuracy, signal-to-noise ratio), and with practitioners not readily able to know which regime is appropriate to their problem, and seeking broadly applicable algorithms that are reasonably close to optimality. To bridge these perspectives it is necessary to study algorithms that are \emph{adaptive} to different regimes. We present the stochastically controlled stochastic gradient (SCSG) method for composite convex finite-sum optimization problems and show that SCSG is adaptive to both strong convexity and target accuracy. The adaptivity is achieved by batch variance reduction with adaptive batch sizes and a novel technique, which we referred to as {\it geometrization}, which sets the length of each epoch as a geometric random variable. The algorithm achieves strictly better theoretical complexity than other existing adaptive algorithms, while the tuning parameters of the algorithm only depend on the smoothness parameter of the objective. 
\end{abstract}

% REQUIRED
\begin{keywords}
  adaptivity, stochastic gradient method, finite-sum optimization, geometrization, variance reduction
\end{keywords}

% REQUIRED
\begin{AMS}
  90C15, 90C25, 90C06
\end{AMS}

\section{Introduction}
 
 The application of gradient-based optimization methodology to statistical machine learning has been a major success story, in practice and in theory.  Indeed, there is an increasingly detailed theory available for gradient-based algorithms that helps to explain their practical success.  There remains, however, a significant gap between theory and practice, in that the designer of machine learning algorithms is required to make numerous choices that depend on parameters that are unlikely to be known in a real-world machine-learning setting.  For example, existing theory asserts that different algorithms are preferred if a problem is strongly convex or merely convex, if the target accuracy is high or low, if the signal-to-noise is high or low and if data are independent or correlated.  This poses a serious challenge to builders of machine-learning software, and to users of that software.  Indeed, a distinctive aspect of machine-learning problems, especially large-scale problems, is that the user of an algorithm can be expected to know little or nothing about quantitative structural properties of the functions being optimized.  It is hoped that the data and the data analysis will inform such properties, not the other way around.

To take a classical example, the stochastic gradient descent (SGD) algorithm takes different forms for strongly convex objectives and non-strongly convex objectives.  In the former case, letting $\mu$ denote the strong-convexity parameter, if the stepsize is set as $O(1 / \mu t)$ then SGD exhibits a convergence rate of $O(1 / \mu \eps)$, where $\eps$ is the target accuracy~\citep{nesterov04}.  In the latter case setting the stepsize to $O(1/ \sqrt{t})$ yields a rate of $O(1 / \eps^2)$~\citep{Nemirovsky09}.  Using the former scheme for non-strongly convex objectives can significantly deteriorate the convergence~\citep{Nemirovsky09}.  It is sometimes suggested that one can insure strong convexity by simply adding a quadratic regularizer to the objective, using the coefficient of the regularizer as a conservative estimate of the strong-convexity parameter.  But this produces a significantly faster rate only if $\mu \gg \eps$, a regime that is unrealistic in many machine-learning applications, where $\eps$ is relatively large.  Setting $\mu$ to such a large value would have a major effect on the statistical properties of the optimizer.

Similar comments apply to presumptions of knowledge of Lipschitz parameters, mini-batch sizes, variance-reduction tuning parameters, etc.  Current practice often involves heuristics in setting these tuning parameters, but the use of these heuristics can change the algorithm and the optimality guarantees may disappear.

Our goal, therefore, should be that our algorithms are \emph{adaptive}, in the sense that they perform as well as an algorithm that is assumed to know the ``correct'' choice of tuning parameters, even if they do not know those parameters.  In particular, in the convex setting, we wish to derive an algorithm that  does not involve $\mu$ in its implementation but whose convergence rate would be better for larger $\mu$ while still reasonable for smaller $\mu$, including the non-strongly convex case where $\mu = 0$. 

Such adaptivity has been studied implicitly in the classic literature. \cite{ruppert1988efficient} and \cite{polyak1990new} and \cite{polyak1992acceleration} showed that the average iterate of SGD with stepsize $O(t^{-\alpha})$ for $\alpha \in (1/2, 1)$ satisfies a central limit theorem with information-theoretically optimal asymptotic variance.  This implies adaptivity because the performance adapts to the underlying parameters of the problem, including the modulus of strong convexity, even though the algorithm does not require knowledge of them.  The analysis by \cite{polyak1992acceleration} is, however, asymptotic and relies on the smoothness of Hessian.  Under similar assumptions on the Hessian, \cite{moulines2011non} provided a non-asymptotic analysis establishing adaptivity of SGD with Polyak-Ruppert averaging.  Further contributions to this line of work include \cite{bach13, flammarion15, dieuleveut17}, who prove the adaptivity of certain versions of SGD with refined rates for self-concordant objectives, including least-square regression and logistic regression. 

This line of work relies on conditions on higher-order derivatives which are not required in the modern literature on stochastic gradient methods.  In fact, under fairly standard assumptions for first-order methods, \cite{moulines2011non} provided a non-asymptotic analysis for SGD with stepsize $O(t^{-\alpha})$ without averaging and showed that this algorithm exhibits adaptivity to strong convexity while having reasonable guarantee for non-strongly convex objectives.  Specifically, if $\alpha = 2/3$, their results show that the rate to achieve an $\eps$-accurate solution for the expected function value is $\td{O}\big(\min\big\{1 / \mu^{3} + 1 / \mu \eps^{2},$ $1 / \eps^{3}\big\}\big)$, where $\td{O}$ hides logarithmic factors. This result was taken further by studying alternative stepsize schemes; in particular, \cite{xu2019accelerate} proposed a variant of projected SGD with stagewise diminishing stepsizes and diameters. Unlike the aforementioned work, the adaptivity in this case is weaker because it requires knowledge of the strong convexity parameter, as well as the initial suboptimality, in the complexity analysis (in particular, they require  a sufficiently large initial diameter).  A weaker form of adaptivity, to smoothness but not to strong convexity, was established by \cite{levy2018online} when the polynomial decaying stepsize is replaced by an AdaGrad-type stepsize \citep{Adagrad}, assuming a bounded domain and bounded stochastic gradients. Finally, \cite{chen2018sadagrad} presented a restarting variant of AdaGrad with provable adaptivity to strong convexity given an initial overestimate of the strong convexity parameter. 

Further progress has been made by focusing on a setting that is particularly relevant to machine learning---that of \emph{finite-sum optimization}% ~\citep[see, e.g.,][]{SAG, Mixedgrad, SAGA, SVRG, SDCA, Catalyst, Finito, Katyusha}
.  The objective function in this setting takes the following form:
\begin{equation}
  \label{eq:obj_finite_sum}
  \min_{x\in \X}F(x) = f(x) + \psi(x), \quad \text{where } f(x)
   = \frac{1}{n}\sum_{i=1}^{n}f_{i}(x),
\end{equation}
where $\X$ is the parameter space, $n$ is the number of data points, the functions $f_{i}(x)$ are data-point-specific \emph{loss} functions and $\psi(x)$ is the \emph{regularization} term.  We assume that each $f_{i}(x)$ is differentiably convex and $\psi(x)$ is convex but can be non-differentiable.  The introduction of the parameter $n$ into the optimization problem has two important implications.  First, it implies that the number of operations to obtain a full gradient is $O(n)$, which is generally impractical in modern machine-learning applications, where the value of $n$ can be in the tens to hundreds of millions.  This fact motivates us to make use of stochastic estimates of gradients.  Such randomness introduces additional variance that interacts with the variability of the data, and tuning parameters are often introduced to control this variance.

Second, the finite-sum formulation highlights the need for adaptivity to the target accuracy $\eps$, where that accuracy is related to the number of data points $n$ for statistical reasons.  Unfortunately, different algorithms perform better in high-accuracy versus low-accuracy regimes, and the choice of regime is generally not clear to a user of machine-learning algorithms, given that target accuracy varies not only as a function of $n$, but also as a function of other parameters, such as the signal-to-noise ratio, that the user is not likely to know.  Ideally, therefore, optimization algorithms should be adaptive to target accuracy, performing well in either regime. 

Deterministic gradient-descent-based methods can be made adaptive to strong convexity and smoothness simultaneously by exploiting the Polyak stepsize \citep{hazan2019revisiting}. However, computation of a full gradient is expensive, rendering the method undesirable for finite-sum optimization. A recent line of research has shown that algorithms with lower complexity can be designed in the finite-sum setting with some adaptivity, generally via careful control of the variance. The stochastic average gradient (SAG) method opened this line of research, establishing a complexity of $\td{O}\lb \min\{n / \eps, n + L / \mu\}\rb$~\citep{SAG}.  Importantly, this result shows that SAG is adaptive to strong convexity.  To achieve such adaptivity, however, SAG requires two sequences of iterates, the average iterate and the last iterate.  \cite{SAGA} propose a single-sequence variant of SAG that is also adaptive to strong convexity, yet under stronger assumption that each $f_{i}$ is strongly convex.  Both methods suffer, however, from a prohibitive storage cost of $O(nd)$, where $d$ is the dimension of $\X$.  Further developments in this vein include the stochastic variance reduced gradient (SVRG) method \citep{SVRG} and the stochastic dual coordinate ascent (SDCA) method \citep{SVRG}; they achieve the same computational complexity as SAG while reducing the storage cost to $O(d)$.  They are not, however, adaptive to strong convexity. 

\cite{SCSG} presented a randomized variant of SVRG that achieves the same convergence rate and adaptivity as SAG but with the same storage cost as SVRG.  However, as is the case with SAG, the complexity of $O(n / \eps)$ for the non-strongly convex case is much larger than the oracle lower bound of $O(n + \sqrt{n / \eps})$~\citep{woodworth16}.  \cite{nguyen2019finite} propose another variance-reduction method that is provably adaptive to strong convexity, though the result is proved for the expected gradient norm and is thus weaker than \cite{SCSG}. \cite{xu17} develop another variant of SVRG which adapts to a more general condition, called a ``H\"{o}lderian error bound,'' with strong convexity being a special case.  In contrast to \cite{SCSG}, they required an initial conservative estimate of the strong convexity parameter.  Under an extra strong assumption of gradient interpolation---that all individual loss functions have vanishing gradients at the optimum---\cite{vaswani2019painless} developed an algorithm that achieves adaptivity to the smoothness parameter and to the modulus of strong convexity simultaneously. On the other hand, recent work of \citet{lan2019unified} that came after our work presents an algorithm that achieves certain adaptivity to the target accuracy. However, they need to know the modulus of strong convexity and the adaptivity is only obtained in the high-accuracy regime because it requires full gradient computations periodically. Finally, while our focus is convex optimization, we note that adaptivity has also been studied for nonconvex finite-sum optimization (e.g., \cite{lei17, paquette2017catalyst}). 

In this article we present an algorithm, the stochastically controlled stochastic gradient (SCSG) algorithm, that exhibits adaptivity both to strong convexity and to target accuracy.  SCSG is a nested procedure that is similar to the SVRG algorithm.  Crucially, it does not require the computation of a full gradient in the outer loop as performed by SVRG, but makes use of stochastic estimates of gradients in both the outer loop and the inner loop.  Moreover, it makes essential use of a randomization technique (``geometrization'') that allows terms to telescope across the outer loop and the inner loops; such telescoping does not happen in SVRG, a fact which leads to the loss of adaptivity for SVRG.

The rest of the article is organized as follows: Section \ref{sec:notation} introduces notation, assumptions and definitions. In Section \ref{sec:SCSG} and Section \ref{sec:ana}, we focus on the relatively simple setting of unregularized problems and Euclidean geometry, introducing the key ideas of geometrization and adaptive batching. We provide key proofs  in Section \ref{sec:ana} and leave non-essential proofs into Appendix \ref{app:ana} in the Supplementary Material. We extend these results to regularized problems and to non-Euclidean geometry in Section \ref{sec:ana_general}.  The extension relaxes standard assumptions for analyzing mirror descent methods % is based on a novel class of functions that satisfy a generalized Nemirovsky inequality
and may be of independent interest. All proofs for the general case are relegated into Appendix \ref{app:ana_general} and some miscellaneous results are presented in Appendix \ref{app:miscellaneous} in Supplementary Material. Finally, the desirable empirical performance of SCSG is demonstrated in Appendix \ref{app:empirical}.

\section{Notation, Assumptions and Definitions}\label{sec:notation}

We write $a\wedge b$ (resp. $a \vee b$) for $\min\{a, b\}$ (resp. $\max\{a, b\}$), and $(a)_{*}^{\xi}$ (or $[a]_{*}^{\xi}$) for $\max\{a, 1\}^{\xi}$ throughout the paper. The symbol $\E$ denotes the expectation of a random element and $\E_{X}$ denotes an expectation over the randomness of $X$ while conditioning on all other random elements. We adopt Landau's notation ($O(\cdot), o(\cdot)$), and we occasionally use $\td{O}(\cdot)$ to hide logarithmic factors.  We define computational cost by making use of the IFO framework of \cite{Agarwal14, reddi16svrg}, where we assume that sampling an index $i$ and computing the pair $(f_{i}(x), \nabla f_{i}(x))$ incurs a unit of cost. For notational convenience, given a subset $\I\subset \{1, \ldots, n\}$, we denote by $\nabla f_{\I}(x)$ the batch gradient: 
\[\nabla f_{\I}(x) = \frac{1}{|\I|}\sum_{i\in \I}\nabla f_{i}(x).\]
By definition, computing $\nabla f_{\I}(x)$ incurs $|\I|$ units of cost.

In this section and the following two sections we focus on unregularized problems and Euclidean geometry, turning to regularized problems and non-Euclidean geometry in Section~\ref{sec:ana_general}.  Specifically, we consider the case $\X = \R^{d}$, $\psi \equiv 0$ and make the following assumptions that target the finite-sum optimization problem:
\begin{enumerate}[\textbf{A}1]
\item $f_{i}$ is convex with $L$-Lipschitz gradient
\[
f_{i}(x) - f_{i}(y) - \ik{f_{i}}{x}{y} \le \frac{L}{2}\|x - y\|_{2}^{2},
\quad \forall i = 1,\ldots, n,
\]
for some $L < \infty$;
\item $F = f$ is strongly convex at $x^{*}$ with
\[ f(x) - f(x^{*})\ge \frac{\mu}{2}\|x - x^{*}\|_{2}^{2},
\]
for some $\mu \ge 0$.
\end{enumerate}

Note that assumption \textbf{A}2 always holds with $\mu = 0$, corresponding to the non-strongly convex case.  Note also that with the exception of \cite{SAG}, this assumption is weaker than most of the the literature on smooth finite-sum optimization, where strong convexity of $f$ is required at every point.

Our analysis will make use of the following key quantity~\citep{SCSG}:
\begin{equation*}
  \H(f) = \frac{1}{n}\sum_{i=1}^{n}\|\nabla f_{i}(x^{*})\|_{2}^{2},
\end{equation*}
where $x^{*}$ denotes the optimum of $f$.  If multiple optima exist we take one that minimizes $\H(f)$.  We use $\H(f)$, an average squared norm at the optimum, in place of the uniform upper bound on the gradient that is often assumed in other work.  The latter is not realistic for many practical problems in machine learning, including least squares, where the gradient is unbounded.  On the other hand,  it is noteworthy that $\H(f)$ only depends on the optimum. For instance, $\H(f) = 0$ if all gradients vanish at the optimum, as studied for over-parametrized models (e.g., \citep{ma2017power, vaswani2018fast, vaswani2019painless}). As will be shown later, the complexity of our algorithm only depends on $\H(f)$ so it can be applied to study the case with data interpolation. We will write $\H(f)$ as $\H$ when no confusion can arise.

We let $\td{x}_{0}$ denote the initial value (possibly random) and define the following measures of complexity:
\begin{equation}
  \label{eq:DxDH}
  \Dx = L\E \|\td{x}_{0} - x^{*}\|_{2}^{2}, \quad \DH = \H / L, \quad \D = \max\{\Dx, \DH\}.
\end{equation}

Recall that a geometric random variable, $N\sim \mathrm{Geom}(\gamma)$, is a discrete random variable with probability mass function $P(N = k) =  (1- \gamma)\gamma^{k},
\mbox{for}\ k = 0, 1, \ldots$, and expectation:
\begin{equation*}
\E N = \frac{\gamma}{1 - \gamma}.
\end{equation*}
Geometric random variables will play a key role in the design and analysis of our algorithm.

Finally, we introduce two fundamental definitions that serve to clarify desirable properties of optimization algorithms.  We refer to the first property as \emph{$\eps$-independence}.
\begin{definition*}
  An algorithm is \textbf{$\mathbf{\eps}$-independent} if it guarantees convergence
  at all target accuracies $\eps$.
\end{definition*}

$\eps$-independence is a crucial property in practice because a target accuracy is usually not exactly known apriori. An $\eps$-independent algorithm satisfies the ``one-pass-for-all'' property where the theoretical complexity analysis applies to the whole path of the iterates. In contrast, an $\eps$-dependent algorithm only has a theoretical guarantee for a particular $\eps$, whose value is often unknown in practice.  To illustrate we consider SGD, where the iterate is updated by $x_{k+1} = x_{k} - \eta_{k}\nabla f_{i_{k}}(x_{k})$ and where $i_{k}$ is a uniform index from $\{1, \ldots, n\}$.  There are two popular schemes for theoretical analysis: (1) $\eta_{k} = O(1 / \sqrt{k})$ or (2) $\eta_{k}\equiv 1 / \sqrt{T}$ and the iterates are updated for $O(T)$ steps where $T = O(1 / \eps^{2})$.  Although both versions have theoretical complexity $\td{O}\lb\eps^{-2}\rb$, only the former is $\eps$-independent.

The second important property is referred to as \emph{almost universality}.
\begin{definition*}
  An algorithm is \textbf{almost universal} if it only requires the knowledge of
  the smoothness parameters $L$.
\end{definition*}

The term \emph{almost universality} is motivated by the notion of \emph{universality} introduced by \citet{nesterov14} which does not require the knowledge of $L$ or other parameters such as the variance of the stochastic gradients.  Returning to the previous example, both versions of SGD are universal. It is noteworthy that universal gradient methods are usually either $\eps$-dependent (e.g. \citep{nesterov14}) or require imposing other assumptions such as uniformly bounded $\nabla f_{i}$ (e.g. \citep{Nemirovsky09}).  The SCSG algorithm developed in this paper is both $\eps$-independent and almost universal. This category also includes SGD for general convex functions \citep{Nemirovsky09}, SAG \citep{SAG}, SAGA \citep{SAGA}, SVRG++ \citep{Zhu15}, Katyusha for non-strongly convex functions \citep{Katyusha}, and AMSVRG \citep{nitanda16}. In contrast, algorithms such as SGD for strongly convex functions \citep{Nemirovsky09}, SVRG \citep{SVRG}, SDCA \citep{SDCA}, APCG \citep{APCG}, Katyusha for strongly convex functions \citep{Katyusha} and adaptive SVRG \citep{xu17} are $\eps$-independent but not almost universal because they need full or partial knowledge of $\mu$. Furthermore, algorithms such as Catalyst \citep{Catalyst} and AdaptReg \citep{blackbox} even depend on unknown quantities such as $F(x_{0}) - F(x^{*})$ or the variance of the $\nabla f_{i}$.  In comparing algorithms we believe that clarity on these distinctions is critical, in addition to comparison of convergence rates.

\section{Stochastically Controlled Stochastic Gradient (SCSG)}\label{sec:SCSG}

In this section we present SCSG, a computationally efficient framework for variance reduction in stochastic gradient descent algorithms.  SCSG builds on the SVRG algorithm of \cite{SVRG}, incorporating several essential modifications that yield not only computational efficiency but also adaptivity.  Recall that SVRG is a nested procedure that computes a full gradient in each outer loop and uses that gradient as a baseline to reduce the variance of the stochastic gradients that are computed in an inner loop.  The need to compute a full gradient, at a cost of $n$ operations, unfortunately makes the SVRG procedure impractical for large-scale applications. SCSG seeks to remove this bottleneck by replacing the full gradient with an approximate, stochastic gradient, one that is based on a batch size that is significantly smaller than $n$ but larger than the size used for the stochastic gradients in the inner loop.  By carefully weighing the contributions to the bias and variance of these sampling-based estimates, SCSG achieves a small iteration complexity while also keeping the per-iteration complexity feasibly small.  % In particular, in the non-strongly-convex setting, SCSG achieves a complexity of $O\lb\frac{1}{\eps^{2}}\wedge \frac{n}{\eps}\rb$, a rate which is a minimum over the complexities of SGD (the first term) and SVRG (the second term).

Further support for the SCSG framework comes from the comparison with SVRG in the setting of strongly convex objectives.  In this setting, SVRG relies heavily on a presumption of knowledge of the strong convexity parameter $\mu$.  In particular, to achieve a complexity of $O((n + \kappa)\log (1/\eps))$, the number of stochastic gradients queried in the inner loop of SVRG needs to scale as $\kappa$.  By contrast, the SCSG framework achieves the same complexity without knowledge of $\mu$.  This is achieved by setting the number of inner-loop stochastic gradients to be a geometric random variable.  As we discuss below, the usage of a geometric random variable---a technique that we refer to as ``geometrization''---is crucial in the design and analysis of SCSG.  We believe that it is a key theoretical tool for achieving adaptivity to strong convexity.

The original version of SCSG was $\eps$-dependent and not almost universal, because it required knowledge of the parameter $\H$~\citep{SCSG}.  Moreover the algorithm had a sub-optimal rate in the high-accuracy regime.  In further development of the SCSG framework, in the context of nonconvex optimization~\citep{lei17}, we found that $\eps$-independence and almost universality could be achieved by employing an increasing sequence of batch sizes.

In the remainder of this section, we bring these ideas together and present the general form of the SCSG algorithm, incorporating adaptive batching, geometrization and mini-batches in the inner loop. The resulting algorithm is adaptive, $\eps$-independent and almost universal. Roughly speaking, the adaptive batching enables the adaptivity to target accuracy and the geometrization
enables the adaptivity to strong convexity.  The pseudocode for SCSG is shown in Algorithm~\ref{algo:SCSGplus}. Guidelines for practical choice for the parameters is provided in Remark \ref{rem:practical} in Section \ref{subsec:multi-epoch}. As can be seen, the algorithm is superficially complex, but, as in the case of line-search and trust-region methods that augment simple gradient-based methods in deterministic optimization, the relative lack of dependence on hyperparameters makes the algorithm robust and relatively easy to deploy.

Note that in Algorithm~\ref{algo:SCSGplus}, and throughout the paper, we use $\td{x}_{j}$ to denote the iterate in the $j$th outer loop and $x_{k}^{(j)}$ to denote the iterate in the $k$th step of the $j$th inner loop. 

\begin{algorithm}
\caption{SCSG for unconstrained finite-sum optimization}
\label{algo:SCSGplus}
\textbf{Inputs: } Number of stages $T$, initial iterate $\td{x}_{0}$, stepsizes $(\etaj)_{j=1}^{T}$, block sizes $(\Bj)_{j=1}^{T}$, inner loop sizes $(m_{j})_{j=1}^{T}$, mini-batch sizes $(\bj)_{j=1}^{T}$.

\textbf{Procedure}
\begin{algorithmic}[1]
  \For{$j = 1, 2, \cdots, T$}
  \State Uniformly sample a batch $\I_{j}\subset \{1,\ldots,n\}$ with $|\I_{j}| = \Bj$;
  \State $\mu_{j}\gets \nabla f_{\I_{j}}(\td{x}_{j-1})$;
  \State $\xj_{0} \gets \td{x}_{j-1}$;
  \State Generate $N_{j}\sim \mathrm{Geom}\lb\frac{\mj}{\mj + \bj}\rb$;
  \For{$k = 1, 2, \cdots, N_{j}$}
  \State Uniformly sample a batch $\sI_{k-1}\subset \{1,\ldots, n\}$ with $|\sI_{k-1}| = \bj$;
  \State $\nuj_{k - 1} \gets \nabla f_{\sI_{k-1}}(\xj_{k-1}) - \nabla f_{\sI_{k-1}}(\xj_{0}) + \mu_{j}$;
  \State $\xj_{k}\gets \xj_{k-1} - \etaj\nuj_{k - 1}$;
  \EndFor
  \State $\td{x}_{j}\gets \xj_{N_{j}}$;
  \EndFor
\end{algorithmic}
\textbf{Output: } $\td{x}_{T}$.
\end{algorithm}

To measure the computational complexity of SCSG, let $T(\eps)$ denote the first time step at which $\td{x}_{T}$ is an $\eps$-approximate solution, as well as all following iterates $\td{x}_{T+1}, \td{x}_{T+2}, \ldots$:
\begin{equation}
  \label{eq:Teps}
  T(\eps) = \min\{T': \E(f(\td{x}_{T}) - f(x^{*}))\le \eps, \,\, \forall T \ge T'\}.
\end{equation}
This criterion is more stringent than those considered in some other work which neglects the performance of $\td{x}_{T}$ for $T > T(\eps)$. The computational cost incurred for computing $\td{x}_{T}$ is
\begin{equation}\label{eq:comp_eps}
\comp(\eps) = \sum_{j=1}^{T(\eps)}(\bj\Nj + \Bj).
\end{equation}
Noting that $\comp(\eps)$ is random, we consider the average complexity obtained
by taking the expectation of $\comp(\eps)$. Since $\Nj\sim \mathrm{Geom}(\frac{\mj}
{\mj + \bj})$, we have:
\begin{equation}
\label{eq:generic_complexity}
  \E\comp(\eps) = \sum_{j=1}^{T(\eps)}\lb\bj\frac{\mj}{\bj} + \Bj\rb
  = \sum_{j=1}^{T(\eps)}(\mj + \Bj).
\end{equation}

\subsection{Two key ideas: adaptive batching and geometrization} 

The adaptivity of SCSG is achieved via two techniques: adaptive batching and geometrization.  We provide intuitive motivation for these two ideas in this section.

The motivation for adaptive batching is straightforward. Heuristically, at the early stages of the optimization process, the iterate is far from the optimum and a small subset of data is sufficient to reduce the variance.  On the other hand, at later stages, finer variance reduction is required to prevent the iterate from moving in the wrong direction.  By allowing the batch sizes to increase, SCSG behaves like SGD for the purposes of low-accuracy computation while it behaves like SVRG for high-accuracy computation.

The motivation for geometrization is more subtle.  To isolate its effect, let us consider a special case of SCSG in which the parameters are set as follows:
\[
\Bj = \mj \equiv n, \bj = 1, \etaj \equiv \eta = O\lb\frac{1}{L}\rb.
\]
Note that the above setting is only used to illustrate the effect of geometrization and the setting that leads to adaptivity to both strong convexity and target accuracy is more involved and given in Section \ref{sec:ana}. In this simplified setting, SCSG reduces to SVRG if we replace line 5 by
$\Nj\sim \mathrm{Unif}\lb \{0, \ldots, \mj - 1\}\rb$, with $\mj \equiv m$
for some positive integer $m$. (Although SVRG is usually implemented in
practice by setting $\Nj$ to be a fixed $m$, a uniform random $\Nj$ is crucial
for the analysis of SVRG \citep{SVRG}.)  SVRG achieves a rate of
$O\lb (n + \kappa)\log (1 / \eps)\rb$ rate only if
\begin{equation}\label{eq:SVRG_cond}
\frac{1}{\mu\eta (1 - 2\eta L)m} + \frac{2\eta L}{1 - 2\eta L} < 1.
\end{equation}
This requires $m > \frac{1}{\mu\eta}$; hence, SVRG requires knowledge
of $\mu$ to achieve the theoretical rate.  We briefly sketch the step
in the proof of the convergence of SVRG where this limitation arises, and we show how geometrization circumvents the need to know $\mu$.  To simplify our arguments we follow \cite{SVRG} and make the assumption that strong convexity holds everywhere for $f$; note that this is stronger than our assumption \textbf{A}2.

In Theorem 1 of \cite{SVRG}, the following argument appears:
\begin{align}
\lefteqn{2\eta \E\la \nabla f(\xj_{k}), \xj_{k} - x^{*}\ra - 4\eta^{2}L
\E (f(\xj_{k}) - f(x^{*}))} \nonumber \\
& \le 4\eta^{2}L \E (f(\xj_{0}) - f(x^{*})) + \E \|\xj_{k} - x^{*}\|_{2}^{2}
- \E \|\xj_{k+1} - x^{*}\|_{2}^{2}.  \label{eq:SVRG1}
\end{align}
Strong convexity implies that
\begin{align}
\lefteqn{2\eta(1 - 2\eta L)\E (f(\xj_{k}) - f(x^{*})) + \eta \mu
\E \|\xj_{k} - x^{*}\|_{2}^{2}} \nonumber\\
&\le 4\eta^{2}L \E (f(\xj_{0}) - f(x^{*})) + \E \|\xj_{k} - x^{*}\|_{2}^{2} - \E \|\xj_{k+1} - x^{*}\|_{2}^{2}.  \nonumber
\end{align}
Note that this conclusion is independent of the choice of $\Nj$ and hence holds for both SVRG and SCSG.  To assess the overall effect of the $j$th inner loop on the left-hand side, we let $k = \Nj$, thereby focusing on the last step of the inner loop, and we substitute $\td{x}_{j}$ for $\xj_{\Nj}$ and
$\td{x}_{j - 1}$ for $\xj_{0}$.  We have:
\begin{align}
\lefteqn{2\eta(1 - 2\eta L)\E (f(\td{x}_{j}) - f(x^{*})) + \eta \mu
\E \|\td{x}_{j} - x^{*}\|_{2}^{2}} \nonumber\\
&\le 4\eta^{2}L \E (f(\td{x}_{j-1}) - f(x^{*})) + \E \|\xj_{\Nj} - x^{*}\|_{2}^{2} - \E \|\xj_{\Nj+1} - x^{*}\|_{2}^{2}.  \label{eq:SVRG3}
\end{align}
For SVRG, $\Nj \sim \mathrm{Unif}\{0, \ldots, m-1\}$, and thus \eqref{eq:SVRG3} reduces to
\begin{align}
\lefteqn{2\eta (1 - 2\eta L)\E (f(\td{x}_{j}) - f(x^{*})) + \eta \mu \E \|\td{x}_{j} - x^{*}\|_{2}^{2}} \nonumber\\
& \le 4\eta^{2}L \E (f(\td{x}_{j-1}) - f(x^{*})) + \frac{1}{m}
\E \|\xj_{0} - x^{*}\|_{2}^{2} - \frac{1}{m}\E \|\xj_{m} - x^{*}\|_{2}^{2}\nonumber\\
& = 4\eta^{2}L \E (f(\td{x}_{j-1}) - f(x^{*})) + \frac{1}{m}
\E \|\td{x}_{j-1} - x^{*}\|_{2}^{2} - \frac{1}{m}\E \|\xj_{m} - x^{*}\|_{2}^{2}. \nonumber
\end{align}
Unfortunately, given that $\xj_{m}\not = \td{x}_{j}$, the last two terms do not telescope, and one has to drop the final term, leading to the following conservative bound:
\begin{align}
  \label{eq:SVRG5}
  \lefteqn{2\eta (1 - 2\eta L)\E (f(\td{x}_{j}) - f(x^{*})) + \eta \mu\E \|\td{x}_{j} - x^{*}\|_{2}^{2}} \\
&\le 4\eta^{2}L \E (f(\td{x}_{j-1}) - f(x^{*})) + \frac{1}{m}\E \|\td{x}_{j-1} - x^{*}\|_{2}^{2}.\nonumber
\end{align}
Without strong convexity (i.e., when $\mu = 0$), $\E \|\td{x}_{j-1} - x^{*}\|_{2}^{2}$ can be arbitrarily larger than $\E (f(\td{x}_{j-1}) - f(x^{*}))$ and hence \eqref{eq:SVRG5} is not helpful. Thus \cite{SVRG} exploit strong convexity at this point, using
$\E \|\td{x}_{j-1} - x^{*}\|_{2}^{2}\le \frac{2}{\mu}\E (f(\td{x}_{j-1}) - f(x^{*}))$.
Then \eqref{eq:SVRG5} implies that
\begin{equation*}
  2\eta (1 - 2\eta L)\E (f(\td{x}_{j}) - f(x^{*})) \le \lb 4\eta^{2}L + \frac{2}{m\mu}\rb \E (f(\td{x}_{j-1}) - f(x^{*})).
\end{equation*}
This requires the coefficient on the left-hand side to be larger than that on the right-hand side, leading to the condition \eqref{eq:SVRG_cond}.

Summarizing, the reason that \cite{SVRG} rely on the knowledge of $\mu$ is that it permits the removal of the last term in \eqref{eq:SVRG3}. By contrast, if $\Nj$ is a geometric random variable instead of a uniform random variable, the problem is
completely circumvented, by making use of the following elementary lemma.

\begin{lemma}\label{lem:geom}
Let $N\sim \mathrm{Geom}(\gamma)$ for $\gamma > 0$. Then for any sequence $D_{0}, D_{1}, \ldots$ with $\E |D_{N}| < \infty$,
\[
\E (D_{N} - D_{N + 1}) = \lb\frac{1}{\gamma} - 1\rb\lb D_{0} - \E D_{N}\rb.
\]
\end{lemma}
\begin{remark}
The requirement $\E |D_{N}| < \infty$ is essential. A useful sufficient condition if $|D_{k}| = O(\mathrm{poly}(k))$ because a geometric random variable has finite moments of any order.
\end{remark}
\begin{proof}
By definition,
  \begin{align*}
\lefteqn{\E (D_{N} - D_{N + 1})  = \sum_{n\ge 0}(D_{n} - D_{n+1}) \gamma^{n}(1 - \gamma)} \\
&= (1 - \gamma)\lb D_{0} - \sum_{n\ge 1}D_{n}(\gamma^{n-1} - \gamma^{n})\rb
= (1 - \gamma)\lb \frac{1}{\gamma}D_{0} - \sum_{n\ge 0}D_{n}(\gamma^{n-1} - \gamma^{n})\rb\\
&=  (1 - \gamma)\lb \frac{1}{\gamma}D_{0} - \frac{1}{\gamma}
\sum_{n\ge 0}D_{n}\gamma^{n}(1 - \gamma)\rb = \lb\frac{1}{\gamma} - 1\rb (D_{0} - \E D_{N}),
  \end{align*}
where the last equality is followed by the condition that $\E |D_{N}| < \infty$.
\end{proof}

Returning to \eqref{eq:SVRG3} for SCSG with Lemma \ref{lem:geom} in hand, where $\Nj \sim \mathrm{Geom}(\frac{n}{n + 1})$ and $D_{k} = \E\|\xj_{k} - x^{*}\|_{2}^{2}$, and assuming that $\E D_{\Nj} < \infty$, we obtain
\begin{align}
\lefteqn{2\eta (1 - 2\eta L)\E (f(\td{x}_{j}) - f(x^{*}))
+ \eta \mu \E \|\td{x}_{j} - x^{*}\|_{2}^{2}} \nonumber\\
& \le 4\eta^{2}L \E (f(\td{x}_{j-1}) - f(x^{*})) +
\frac{1}{n}\E \|\xj_{0} - x^{*}\|_{2}^{2} - \frac{1}{n}\E \|\xj_{\Nj} - x^{*}\|_{2}^{2}\nonumber\\
& = 4\eta^{2}L \E (f(\td{x}_{j-1}) - f(x^{*}))
+ \frac{1}{n}\E \|\td{x}_{j-1} - x^{*}\|_{2}^{2} - \frac{1}{n}\E \|\td{x}_{j} - x^{*}\|_{2}^{2}.
\label{eq:SVRG7}
\end{align}
The assumption that $\E D_{\Nj} < \infty$ will be justified in our general theory and is taken for granted here to avoid distraction. \eqref{eq:SVRG7} can be rearranged to yield a function that provides a better assessment of progress than the function in \eqref{eq:SVRG5}:
\begin{align}\label{eq:SVRG8}
\lefteqn{2\eta (1 - 2\eta L)\E (f(\td{x}_{j}) - f(x^{*})) + \lb\frac{1}{n}
+ \eta \mu\rb\E \|\td{x}_{j} - x^{*}\|_{2}^{2}} \nonumber\\
& \le 4\eta^{2}L \E (f(\td{x}_{j-1}) - f(x^{*}))
+ \frac{1}{n}\E \|\td{x}_{j-1} - x^{*}\|_{2}^{2}.
\end{align}
We accordingly view the left-hand side of \eqref{eq:SVRG8} as a Lyapunov function and define:
\[\mathcal{L}_{j} = 2\eta (1 - 2\eta L)\E (f(\td{x}_{j}) - f(x^{*}))
+ \lb\frac{1}{n} + \eta \mu\rb\E \|\td{x}_{j} - x^{*}\|_{2}^{2}.\]
We then have:
\[
\mathcal{L}_{j}\le \max\left\{\frac{2\eta L}{1 - 2\eta L},
\frac{1}{1 + n\eta \mu}\right\}\mathcal{L}_{j-1} \triangleq \lambda^{-1} \mathcal{L}_{j-1}.
\]
As a result,
\[
\mathcal{L}_{T}\le \eps, \quad \forall T \ge \frac{\log
\frac{\mathcal{L}_{0}}{\eps}}{\log \lambda}\Longrightarrow T(\eps)
\le \frac{\log \frac{\mathcal{L}_{0}}{\eps}}{\log \lambda},
\]
and, by \eqref{eq:comp_eps},
\[
\E \comp(\eps) \le 2nT(\eps) = O\lb n \frac{\log \frac{\mathcal{L}_{0}}{\eps}}
{\log \lambda}\rb.
\]
Suppose $\eta L < \frac{1}{6}$,
\[\lambda \ge 1 + (n\eta \mu \wedge 1) \Longrightarrow \frac{1}{\log \lambda}
= O\lb\frac{1}{n\eta \mu \wedge 1}\rb = O\lb \frac{\kappa}{n} + 1\rb.\]
Therefore the complexity of SCSG is
\[\E \comp(\eps) = O\lb (n + \kappa) \log\lb\frac{\mathcal{L}_{0}}{\eps}\rb\rb.\]
In summary, the better control provided by geometrization enables SCSG to achieve the fast rate of SVRG without knowledge of $\mu$.

\section{Convergence Analysis of SCSG for Unregularized Smooth Problems}\label{sec:ana}

\subsection{One-epoch analysis}
We start with the analysis for a single epoch. The key difficulty lies in controlling the bias of $\nuj_{k}$, conditional on $\I_{j}$ drawn at the beginning of the $j$th epoch. We have:
\begin{equation}
  \label{eq:def_ej}
  \E_{\sI_{k}}\nuj_{k} - \nabla f(\xj_{k}) = \nabla f_{\Ij}(\td{x}_{j-1}) - \nabla f(\td{x}_{j - 1})\triangleq \ej.
\end{equation}
We deal with this extra bias by exploiting Lemma \ref{lem:geom} and obtaining the following theorem which connects the iterates produced in consecutive epochs. The proof of the theorem is relegated to Section \ref{subsec:technical_proofs}.

\begin{theorem}\label{thm:one_epoch_simplify_L2}
Fix any $\Gamma \le 1 / 4$. Assume that
\begin{equation}\label{eq:etaj_mj_bj_Bj_cond}
  \etaj L \le \min\left\{\frac{1 - \Gamma}{2}, \Gamma \bj, \frac{\Gamma^{2}\bj\Bj}{2\mj}\right\}, \quad \mj \ge \bj.
\end{equation}
Then under assumptions \textbf{A}1 and \textbf{A}2,
\begin{align}
  \lefteqn{\E (f(\td{x}_{j}) - f(x^{*})) + \lb\frac{2\bj (1 - \Gamma)}{3\etaj \mj} + \frac{\mu}{6}\rb\E \|\td{x}_{j} - x^{*}\|_{2}^{2}} \nonumber\\
& \le 4\Gamma \E (f(\td{x}_{j-1}) - f(x^{*})) + \frac{2\bj (1 - \Gamma)}{3\etaj \mj}\E \|\td{x}_{j-1} - x^{*}\|_{2}^{2} + \frac{2\etaj L}{\Gamma \bj}\frac{\mj}{\Bj} \DH I(\Bj < n)\nonumber.
\end{align}
\end{theorem}

\subsection{Multi-epoch analysis}\label{subsec:multi-epoch}
We now turn to the multi-epoch analysis, focusing on using the one-epoch analysis to determine the setting of the hyperparameters. Interestingly, we require that the batch size $\Bj$ scales as the square of the number of inner-loop iterations $\mj$. The proof is relegated to Section \ref{subsec:technical_proofs}.
\begin{theorem}\label{thm:mainL2}
   Fix any constant $\alpha > 1$, $m_{0} > 0$ and  $\xi \in (0, 1)$. Let 
\begin{equation*}
\etaj \equiv \eta, \quad \bj \equiv b, \quad \mj = m_{0}\alpha^{j}, \quad \Bj = \lceil B_{0}\alpha^{2j} \wedge n\rceil.
\end{equation*}
Take $\Gamma = 1 / 4\alpha^{1/\xi}$ and assume that 
\begin{equation*}
m_{0}\ge b / \Gamma \quad 2\eta L\le \min\left\{1 - \Gamma, 2\Gamma b, \Gamma^{2}bB_{0} / m_{0}\right\}.
\end{equation*}
Then 
\[\E (f(\td{x}_{T}) - f(x^{*}))\le \Lambda_{T}^{-1}\frac{\Dx}{2\eta L} + \td{\Lambda}_{T}^{-1}\frac{2\eta L m_{0}}{\Gamma b B_{0}}\DH (T\wedge T_{n}^{*}),\]
where $\Lambda_{T} = \lambda_{T}\lambda_{T-1}\cdots \lambda_{1}$, $\td{\Lambda}_{T} = \td{\lambda}_{T}\td{\lambda}_{T-1}\cdots \td{\lambda}_{1}$, 
\[\lambda_{j} =
  \left\{\begin{array}{cc}
    \alpha & (j\le T_{\kappa}^{*})\\
    \alpha^{1/\xi} & (j > T_{\kappa}^{*})
  \end{array}\right., \quad 
  \td{\lambda}_{j} =
  \left\{\begin{array}{cc}
    \alpha & (j\le T_{n}^{*} \vee T_{\kappa}^{*})\\
    \alpha^{1/\xi} & (j > T_{n}^{*} \vee T_{\kappa}^{*}),
  \end{array}\right.
\]
and $T_{\kappa}^{*}, T_{n}^{*}$ be positive numbers such that
\[\alpha^{T_{\kappa}^{*}} = \frac{1}{\eta \mu}, \quad B_{0}\alpha^{2T_{n}^{*}} = n.\]
\end{theorem}

\begin{remark}\label{rem:practical}
  In practice, we recommend the following setting as a default: 
  \[\alpha = 1.25, \quad m_{0} = 5B_{0} = 50b.\]
  This setting works well as demonstrated in Appendix \ref{app:empirical} in the Supplementary Material. For those examples, $b$ was chosen as $10^{-4}n$ for fair comparison. Here we choose $B_{0}$ smaller than $m_{0}$ because the batch size $B_{t}$ grows faster than the inner loop size $m_{t}$ and the variance reduction is more effective in later stages. Under this setting, SCSG takes $\lceil \log (n / B_{0}) / 2\log \alpha\rceil = 16$ passes for $B_{t}$ to reach $n$. This creates a reasonably long transition from little to full variance reduction. 
\end{remark}

\subsection{Complexity analysis}
Under the specification of Theorem \ref{thm:mainL2} and recalling the definition of $T(\eps)$ in \eqref{eq:Teps}, we have
\begin{equation*}
  \sum_{j=1}^{T(\eps)}\mj = m_{0}\sum_{j=1}^{T(\eps)}\alpha^{j} = O\lb \alpha^{T(\eps)}\rb.
\end{equation*}
On the other hand, 
\[\sum_{j=1}^{T(\eps)}\Bj = O\lb \sum_{j=1}^{T(\eps)}(\alpha^{2j}\wedge n)\rb = O\lb \min\left\{\sum_{j=1}^{T(\eps)}\alpha^{2j}, \sum_{j=1}^{T(\eps)}n\right\}\rb = O\lb\alpha^{2T(\eps)}\wedge nT(\eps)\rb.\]
By \eqref{eq:generic_complexity}, we conclude that
\begin{equation}\label{eq:Ecompeps}
  \E\comp(\eps) = O\lb \alpha^{T(\eps)} + \alpha^{2T(\eps)} \wedge nT(\eps)\rb = O\lb \alpha^{2T(\eps)} \wedge (\alpha^{T(\eps)} + nT(\eps))\rb.
\end{equation}
The following theorem gives the size of $T(\eps)$ and thus provides the theoretical complexity of SCSG. The proof is relegated to Section \ref{subsec:technical_proofs}.
\begin{theorem}\label{thm:complexityL2}
Under the specification of Theorem \ref{thm:mainL2}, we have 
\[\E\comp(\eps) = O\lb A(\eps)^{2}\wedge \lb A(\eps) + n\log A(\eps)\rb\rb,\]
where
\[A(\eps) = \td{O}\lb\min\left\{\frac{\D}{\eps}, \kappa\lb\frac{\Dx}{\eps \kappa}\rb^{\xi}_{*} + \frac{\DH}{\eps}, \td{\kappa}\lb\frac{\D}{\eps \td{\kappa}}\rb^{\xi}_{*}\right\}\rb, \quad \td{\kappa} = \sqrt{n} + \kappa.\]
In particular,
\begin{equation}\label{eq:mainL2_complexity}
\E\comp(\eps) = \td{O}\lb \min\left\{\frac{D^{2}}{\eps^{2}}, \frac{\DH^{2}}{\eps^{2}} + \kappa^{2}\lb \frac{\Dx}{\eps\kappa}\rb_{*}^{2\xi}, n + \frac{\D}{\eps}, n + \td{\kappa}\lb\frac{\D}{\eps \td{\kappa}}\rb^{\xi}_{*}\right\}\rb.
\end{equation}
\end{theorem}
\begin{remark}
The version of SCSG considered in Theorem \ref{thm:complexityL2} that achieves the complexity \eqref{eq:mainL2_complexity} is $\eps$-independent and almost universal.% , because the specification \eqref{eq:specification} and the condition \eqref{eq:mainL2_cond} only depend on the the unknown smoothness parameter $L$ while all other parameters, $m_{0}, B_{0}, b, \alpha, \xi$, are user-specified. 
\end{remark}

\subsection{Discussion}
Our complexity result involves the unusual terms $\lb\frac{\Dx}{\eps\kappa}\rb^{2\xi}_{*}$ and $\lb\frac{\D}{\eps\kappa}\rb^{2\xi}_{*}$. However, they are relatively insignificant as the exponent $\xi$ can be made arbitrarily small and $\frac{1}{\eps \kappa}$ is small in practice unless the target accuracy is unusually high. For instance, under the setting given in Remark \ref{rem:practical}, $\xi$ can be as small as $\log\alpha / \log (m_{0} / 4b) \approx 0.088$ to guarantee the condition $m_{0} \ge b / \Gamma$ as long as the stepsize is sufficiently small. Thus, the term $\lb\frac{\Dx}{\eps \kappa}\rb^{2\xi}_{*}$ is generally negligible. If we use $\tto$ to denote a bound that hides these terms and the logarithmic terms, we have 
\begin{equation}\label{eq:tto_complexity}
\E \comp(\eps) = \tto\lb \frac{\D^{2}}{\eps^{2}}\wedge \left\{\kappa^{2} + \frac{\DH^{2}}{\eps^{2}}\right\}\wedge \left\{n + \frac{\D}{\eps}\right\} \wedge \left\{n + \kappa\right\} \rb.
\end{equation}
We discuss some of the consequences of \eqref{eq:tto_complexity}. 

\subsubsection{Adaptivity to target accuracy}

For non-strongly convex objectives, \eqref{eq:tto_complexity} implies that
\begin{equation}\label{eq:adapt_accuracy1}
\E \comp(\eps) = \tto\lb \frac{\D^{2}}{\eps^{2}}\wedge \left\{n + \frac{\D}{\eps}\right\}\rb.
\end{equation}
whereas, for strongly convex objectives, \eqref{eq:tto_complexity} implies that
\begin{equation}\label{eq:adapt_accuracy2}
\E \comp(\eps) = \tto\lb  \left\{\kappa^{2} + \frac{\DH^{2}}{\eps^{2}}\right\}\wedge \left\{n + \kappa\right\}\rb.
\end{equation}
Both \eqref{eq:adapt_accuracy1} and \eqref{eq:adapt_accuracy2} exhibit the adaptivity of SCSG to the target accuracy: for low-accuracy computation (i.e., large $\eps$), SCSG achieves the same complexity as SGD for non-strongly convex objectives, which can be much more efficient than SVRG-type algorithms in the setting of large datasets (i.e., large $n$).  On the other hand, for high-accuracy computation (i.e., small $\eps$), SCSG avoids the high variance of SGD and achieves the same complexity as SVRG++ \citep{Zhu15} for non-strongly convex objectives and as SVRG for strongly convex objectives \citep{SVRG}.

% \cite{reddi16svrg} proposed combining SGD and SVRG to achieve adaptivity. Their approach, however requires predetermining which algorithm to use and requires knowledge of parameters such as the bound of the gradient norms. In other words, this algorithm is neither $\eps$-independent nor almost universal. In addition, SGD and SVRG involve different tuning procedures: SGD requires the setting of a time-varying stepsize that scales as $t^{-\frac{1}{2}}$, or a fixed stepsize that scales as $T^{-\frac{1}{2}}$ for some predetermined number of steps $T$, while SVRG requires the use of a constant stepsize. For this reason, this proposal is not practically useful despite its theoretical guarantee. 

\subsubsection{Adaptivity to strong convexity}

The first two terms of \eqref{eq:tto_complexity} are independent of $n$:
\begin{equation}\label{eq:adapt_accuracy3}
\E \comp(\eps) = \tto\lb \frac{\D^{2}}{\eps^{2}}\wedge \left\{\kappa^{2} + \frac{\DH^{2}}{\eps^{2}}\right\}\rb.
\end{equation}
The last two terms of \eqref{eq:tto_complexity} depend on $n$ but have better dependence on $\eps$:
\begin{equation}
  \label{eq:adapt_accuracy4}
\E\comp(\eps) = \tto\lb\left\{n + \frac{\D}{\eps}\right\} \wedge \left\{n + \kappa\right\}\rb.
\end{equation}
Both \eqref{eq:adapt_accuracy3} and \eqref{eq:adapt_accuracy4} show the adaptivity of SCSG to strong convexity. In both cases, if $\kappa <\!\!< \frac{1}{\eps}$, SCSG benefits from the strong convexity: for the former, \eqref{eq:adapt_accuracy3} yields
\[\E \comp(\eps) = \tto\lb \kappa^{2} + \frac{\DH^{2}}{\eps^{2}}\rb,\]
which can be much smaller than $O\lb\frac{\D^{2}}{\eps^{2}}\rb$, if $\Dx >\!\!> \DH$.  For the latter, \eqref{eq:adapt_accuracy4} yields
\[\E \comp(\eps) = \tto\lb n + \kappa\rb,\]
which is the same as SVRG, but without the knowledge of $\mu$. On the other hand, in ill-conditioned problems where $\kappa >\!\! > \frac{1}{\eps}$, SCSG still achieves the best of SGD and SVRG++ \citep{Zhu15} for non-strongly convex objectives. This is not achieved by Adaptive SVRG \citep{xu17} and only partially achieved by R-SVRG \citep{SCSG}, which requires two sequences of iterates. Finally, although SAG and SAGA provide guarantees in ill-conditioned problems, they have an inferior complexity of $\td{O}\lb \frac{n}{\eps}\wedge (n + \kappa)\rb$. 

\subsubsection{Weaker requirement on gradients}

For algorithms without access to full gradients, it is necessary to impose some conditions on $\nabla f_{i}(x)$. The strongest condition imposes a uniform bound (see, e.g., \citep{Nemirovsky09}):
\begin{equation}\label{eq:sigma2}
\sigma^{2}\triangleq \max_{i}\sup_{x}\|\nabla f_{i}(x)\|_{2}^{2} < \infty,
\end{equation}
while a slightly milder condition imposes the following bound (see, e.g., \citep{minibatchSGD})
\begin{equation}\label{eq:A2}
A^{2}\triangleq \sup_{x}\frac{1}{n}\sum_{i=1}^{n}\|\nabla f_{i}(x) - \nabla f(x)\|_{2}^{2} < \infty.
\end{equation}
These two types of conditions are typical in analyses of SGD when $f_{i}$ is not assumed to be convex. This is satisfied by many practical problems; e.g., generalized linear models. In our situation, the extra assumption on the convexity of each component allows us to relax assumptions such as \eqref{eq:sigma2} or \eqref{eq:A2} into 
\begin{equation*}
  \H\triangleq \frac{1}{n}\sum_{i=1}^{n}\|\nabla f_{i}(x^{*})\|_{2}^{2} < \infty.
\end{equation*}
First it is easy to show that
\[\H \le A^{2}\le \sigma^{2}.\]
More importantly, $\H$ can be much smaller than the other two measures, and there are common situations where $A^{2} = \sigma^{2} = \infty$ while $\H < \infty$. For instance, in least-squares problems where $f_{i}(x) = \frac{1}{2}(a_{i}^{T}x - b_{i})^{2}$, $A^{2} = \sigma^{2} = \infty$ unless the domain is bounded. Although assuming a bounded domain is a common assumption in the literature, there is generally no guarantee, at least for algorithms involving stochasticity, that the iterate will stay in the domain unless a projection step is performed. However, the projection step is never performed in practice and thus the bounded domain assumption is artificial. By contrast, \cite{SCSG} show that 
\[\H \le \frac{2}{n}\sum_{i=1}^{n}(2b_{i}^{2} - f_{i}(x^{*}))\le \frac{4\sum_{i=1}^{n}b_{i}^{2}}{n},\]
for least-squares problems, without a bounded domain. This implies that $\H = O(1)$ provided that $\frac{1}{n}\sum_{i=1}^{n}b_{i}^{2} = O(1)$. Similar bounds can be derived for generalized linear models \citep{SCSG}. It turns out that $\H = O(1)$ for various applications where there is no guarantee for $\sigma^{2}$ or $A^{2}$. We refer the readers to \cite{SCSG} for an extensive discussion.

\subsubsection{Optimality of the complexity bound}

To the best of our knowledge, SCSG is the first stochastic algorithm that achieves adaptivity to both target accuracy and strong convexity. However, it is still illuminating to compare each component of \eqref{eq:tto_complexity} separately with the best achievable rate in the literature. 

\begin{itemize}
\item The first component $\tto\lb\frac{\D^{2}}{\eps^{2}}\rb$ is optimal in terms of $\eps$-dependence for non-strongly convex objectives \citep{Agarwal10, woodworth16}. Under slightly stronger assumptions on the gradient bounds (but without the convexity of each $f_{i}$), mini-batched SGD achieves the $O\lb\frac{1}{\eps^{2}}\rb$ rate \citep{Nemirovsky09, minibatchSGD}. However, the dependence on $\D$ is suboptimal. Without knowing $\Dx$ and $\sigma^{2}$ or $A^{2}$, defined in \eqref{eq:sigma2} and \eqref{eq:A2} , the resulting complexity of (mini-batched) SGD can be no better than $O\lb\frac{(\Dx \vee A)^{2}}{\eps^{2}}\rb \succeq O\lb\frac{D^{2}}{\eps^{2}}\rb$. When they are known, \cite{Nemirovsky09} and \cite{minibatchSGD} are able to improve it to $O\lb\frac{\Dx\sigma}{\eps^{2}}\rb$ and $O\lb\frac{\Dx A}{\eps^{2}}\rb$. With averaging, \cite{juditsky11} improve it to
\[O\lb\frac{\Dx}{\eps} + \frac{\sigma^{2}}{\eps^{2}}\rb\succeq O\lb\frac{\Dx}{\eps} + \frac{\DH^{2}}{\eps^{2}}\rb.\]
With momentum acceleration, \cite{lan12} further improve the rate to
\begin{equation}\label{eq:SGD_optimal}
O\lb\sqrt{\frac{\Dx}{\eps}} + \frac{\sigma^{2}}{\eps^{2}}\rb\succeq O\lb\sqrt{\frac{\Dx}{\eps}} + \frac{\DH^{2}}{\eps^{2}}\rb.
\end{equation}
% We conjecture that the RHS of \eqref{eq:SGD_optimal} can be achieved by adding momentum terms to SCSG. 
\item The second component $\tto\lb \kappa^{2} + \frac{\DH^{2}}{\eps^{2}}\rb$ is new to the best of our knowledge. When $\mu$ is known and $\E\|\nabla  f_{i}(x)\|_{2}^{2}$ is uniformly bounded for all $i$ and $x$, a requirement that is more stringent than our setting, it is known that the optimal complexity in terms of $\eps$-dependence and $\mu$-dependence is $O\lb\frac{\kappa}{\eps}\rb$; see, e.g., \cite{hazan10, Rakhlin12} for the upper bound and \cite{woodworth16} for the lower bound. However, the lower bound is established under the condition that $\mu$ is known. It remains an interesting direction to derive a tight lower bound when $\mu$ is unknown.
\item The third component $\tto\lb n + \frac{\D}{\eps}\rb$ should be sub-optimal in terms of both $\eps$ and $\D$. SVRG++ \citep{Zhu15} achieves the $\td{O}\lb n + \frac{\Dx}{\eps}\rb$ rate. By adding momentum terms, Adaptive SVRG \citep{nitanda16} slightly improves the rate to $\td{O}\lb \left\{n + \frac{\Dx}{\eps}\right\}\wedge n\sqrt{\frac{\Dx}{\eps}}\rb$. On the other hand, \cite{woodworth16} prove a lower bound $\td{O}\lb n + \sqrt{\frac{n\Dx}{\eps}}\rb$. This can be achieved by Accelerated SDCA \citep{APSDCA} or Katyusha \citep{Katyusha}. However, the former has only been established for particular problems such as generalized linear models, and the latter involves black-box acceleration \citep{blackbox}, which requires setting the parameters based on unknown quantities such as $\Dx$. The Varag algorithm \cite{lan2019unified}, which came after our work, is the first algorithm that achieves the lower bound for generic finite-sum optimization problems. % It remains unclear whether there exists an almost universal algorithm that achieves the lower bound for generic finite-sum optimization problems.
\item The last component $\tto\lb n + \kappa\rb$ has been proved by \cite{arjevani17} to be optimal, up to small factors $\lb\frac{\DH}{\eps\kappa}\rb^{2\xi}_{*}$, for a large class of algorithms when $\mu$ is unknown. The story is different when $\mu$ is known. The optimal complexity can be improved to $\td{O}(n + \sqrt{n\kappa})$ and can be achieved by, for instance, by Katyusha \citep{Katyusha}.
\end{itemize}

In summary, a major remaining challenges is to derive oracle lower bounds involving all of $\eps, n, \Dx, \DH$, for $\eps$-independent and almost universal algorithms.

\subsection{Proofs}\label{subsec:technical_proofs}
\subsubsection{Lemmas}
We start by four lemmas with proofs presented in Section \ref{app:ana} in Supplementary Material. The first lemma gives an upper bound of the expected squared norm of $\nuj_{k}$, which is standard in the analyses of most first-order methods.
\begin{lemma}\label{lem:nuknorm}
Under assumption \textbf{A}1,
\begin{align*} 
\lefteqn{\E_{\sI_{k}} \|\nuj_{k}\|_{2}^{2}}\\
&\le \frac{2L}{\bj} (f(\xj_{0}) - f(\xj_{k})) + \frac{2L}{\bj} \la\nabla f(\xj_{k}), \xj_{k} - \xj_{0}\ra + 2\|\nabla f(\xj_{k})\|_{2}^{2} + 2\|\ej\|_{2}^{2},
\end{align*}
where $\ej$ is defined in \eqref{eq:def_ej}.
\end{lemma}

The second lemma gives an upper bound for $\E \|\ej\|_{2}^{2}$. 
\begin{lemma}\label{lem:ejL2}
Under assumption \textbf{A}1,
\[\E \|\ej\|_{2}^{2}\le \frac{2}{\Bj}\left\{ 2L\E (f(\td{x}_{j-1}) - f(x^{*})) + \H \, I(\Bj < n)\right\}.\]  
\end{lemma}

The third lemma connects the iterates $\td{x}_{j}$ and $\td{x}_{j-1}$ in adjacent epochs. The proof exploits the elegant property of geometrization. 
\begin{lemma}\label{lem:dual}
Let $u\in \R^{d}$ be any variable that is independent of $\Ij$ and subsequent random subsets within the $j$th epoch, $\sI_{0}, \sI_{1}, \ldots$, with $\E \|u - x^{*}\|_{2}^{2} < \infty$. Then under assumption \textbf{A}1,
\begin{align*}
\lefteqn{2\etaj\E\la \nabla f(\td{x}_{j}), \td{x}_{j} - u \ra}\\
 &\le \frac{\bj}{\mj}\lb\E\|\td{x}_{j-1} - u\|_{2}^{2} - \E\|\td{x}_{j} - u\|_{2}^{2}\rb + 2\etaj\E \la\ej, \td{x}_{j - 1} - \td{x}_{j}\ra + \E\Wj,
\end{align*}  
where 
\[\Wj = \frac{2\etaj^{2}L}{\bj} (f(\td{x}_{j-1}) - f(\td{x}_{j})) + \frac{2\etaj^{2}L}{\bj} \la\nabla f(\td{x}_{j}), \td{x}_{j} - \td{x}_{j-1}\ra + 2\etaj^{2}\|\nabla f(\td{x}_{j})\|_{2}^{2} + 2\etaj^{2}\|\ej\|_{2}^{2}.\]
\end{lemma}

The term $\E \la \ej, \td{x}_{j-1} - \td{x}_{j}\ra$ is non-standard. We derive an upper bound in the following lemma. Surprisingly, this lemma is a direct consequence of Lemma \ref{lem:dual}.
\begin{lemma}\label{lem:xjxj-1L2}
Fix any $\gamma_{j} > 0$. Under assumption \textbf{A}1,
\begin{align*}
\lefteqn{2\etaj\E \la\ej, \td{x}_{j-1} - \td{x}_{j}\ra}\\
& \le -2\gamma_{j}\etaj \E \la\nabla f(\td{x}_{j}), \td{x}_{j} - \td{x}_{j-1}\ra + \gamma_{j}\E\Wj + \frac{\etaj^{2}\mj}{\bj}\frac{(1+\gamma_{j})^{2}}{\gamma_{j}}\E\|\ej\|_{2}^{2}.
\end{align*}
\end{lemma}

\subsubsection{Proof of Theorem \ref{thm:one_epoch_simplify_L2}}
We start from a more general version of Theorem \ref{thm:one_epoch_simplify_L2}. 
\begin{theorem}\label{thm:one_epoch_L2}
Fix any $\Gamma_{j} \in (0, 1)$. Assume that
\begin{equation}\label{eq:etaj_cond}
  \etaj L \le \min\left\{\frac{1 - \Gamma_{j}}{2}, \Gamma_{j}\bj\right\}.
\end{equation}
Then under assumption \textbf{A}1,
\begin{align}
  \lefteqn{\E (f(\td{x}_{j}) - f(x^{*})) + \frac{\bj (1 - \Gamma_{j})}{2\etaj \mj}\E \|\td{x}_{j} - x^{*}\|_{2}^{2}} \nonumber\\
& \le \lb \Gamma_{j} + 2\lb\frac{1}{\Gamma_{j}} + \frac{2\bj}{\mj}\rb\frac{\etaj L \mj}{\bj\Bj}\rb \E (f(\td{x}_{j-1}) - f(x^{*})) + \frac{\bj (1 - \Gamma_{j})}{2\etaj \mj}\E \|\td{x}_{j-1} - x^{*}\|_{2}^{2}\nonumber\\
& + \lb \frac{1}{\Gamma_{j}} + \frac{2\bj}{\mj}\rb\frac{\etaj \mj}{\bj \Bj} \H I(\Bj < n).\nonumber
\end{align}
\end{theorem}
\begin{proof}
Letting $u = x^{*}$ in Lemma \ref{lem:dual} and applying Lemma \ref{lem:xjxj-1L2} with $\gamma_{j} = \Gamma_{j} / (1 - \Gamma_{j})$ (i.e. $\Gamma_{j} = \gamma_{j} / (1 + \gamma_{j})$), we obtain:
\begin{align}
\hspace{.05cm}\lefteqn{2\etaj\E\la \nabla f(\td{x}_{j}), \td{x}_{j} - x^{*} \ra  \le  \frac{\bj}{\mj}\lb\E\|\td{x}_{j-1} - x^{*}\|_{2}^{2} - \E\|\td{x}_{j} - x^{*}\|_{2}^{2} \rb}\nonumber \\
& \hspace{.5cm} -2\gamma_{j}\etaj \E \la\nabla f(\td{x}_{j}), \td{x}_{j} - \td{x}_{j-1}\ra + (1 + \gamma_{j})\E\Wj + \frac{\etaj^{2}\mj}{\bj}\frac{(1 + \gamma_{j})^{2}}{\gamma_{j}}\E\|\ej\|_{2}^{2}\nonumber\\
& \le \frac{\bj}{\mj}\lb\E\|\td{x}_{j-1} - x^{*}\|_{2}^{2} - \E\|\td{x}_{j} - x^{*}\|_{2}^{2}\rb + 2\etaj^{2}(1 + \gamma_{j})\E \|\nabla f(\td{x}_{j})\|_{2}^{2} \nonumber\\
& \hspace{.5cm}+ \frac{2(1 + \gamma_{j})\etaj^{2}L}{\bj}\E (f(\td{x}_{j-1}) - f(\td{x}_{j}))\nonumber\\
& \hspace{.5cm} - 2\etaj \lb\gamma_{j} - \frac{(1 + \gamma_{j})\etaj L}{\bj}\rb \E \la\nabla f(\td{x}_{j}), \td{x}_{j} - \td{x}_{j-1}\ra \nonumber\\
&  \hspace{.5cm} + \frac{(1 + \gamma_{j})\etaj^{2}\mj}{\bj}\lb \frac{1 + \gamma_{j}}{\gamma_{j}} + \frac{2\bj}{\mj}\rb\E \|\ej\|_{2}^{2}.\label{eq:one_epoch_L21}
\end{align}
First, by Lemma \ref{lem:cocoercive} with $x = x^{*}, y = \td{x}_{j}$ and the fact that $\nabla f(x^{*}) = 0$,
\begin{equation}
  \label{eq:one_epoch_L22}
  \|\nabla f(\td{x}_{j})\|_{2}^{2}\le 2L(f(x^{*}) - f(\td{x}_{j}) + \la\nabla f(\td{x}_{j}), \td{x}_{j} - x^{*}\ra).
\end{equation}
Second, since $\gamma_{j}\ge \frac{(1 + \gamma_{j})\etaj L}{\bj}$ by \eqref{eq:etaj_cond}, by convexity of $f$ we obtain that
\begin{align}
  \lefteqn{\frac{2(1 + \gamma_{j})\etaj^{2}L}{\bj}\E (f(\td{x}_{j-1}) - f(\td{x}_{j})) - 2\etaj \lb\gamma_{j} - \frac{(1 + \gamma_{j})\etaj L}{\bj}\rb \E \la\nabla f(\td{x}_{j}), \td{x}_{j} - \td{x}_{j-1}\ra} \label{eq:one_epoch_L23}\\
& \le \frac{2(1 + \gamma_{j})\etaj^{2}L}{\bj}\E (f(\td{x}_{j-1}) - f(\td{x}_{j})) - 2\etaj \lb\gamma_{j} - \frac{(1 + \gamma_{j})\etaj L}{\bj}\rb \E (f(\td{x}_{j}) - f(\td{x}_{j-1}))\nonumber\\
& = 2\etaj \gamma_{j} \E (f(\td{x}_{j-1}) - f(\td{x}_{j})).\nonumber
\end{align}
Combining \eqref{eq:one_epoch_L21}-\eqref{eq:one_epoch_L23} yields
\begin{align}
  \lefteqn{2\etaj (1 - 2(1 + \gamma_{j})\etaj L)\E \la\nabla f(\td{x}_{j}), \td{x}_{j} - x^{*}\ra + 4(1 + \gamma_{j})\etaj^{2}L \E (f(\td{x}_{j}) - f(x^{*}))} \label{eq:one_epoch_L24} \\
& \le \frac{\bj}{\mj}\lb\E\|\td{x}_{j-1} - x^{*}\|_{2}^{2} - \E\|\td{x}_{j} - x^{*}\|_{2}^{2}\rb + 2\etaj \gamma_{j}\E (f(\td{x}_{j-1}) - f(\td{x}_{j}))\nonumber\\
& + \frac{(1 + \gamma_{j})\etaj^{2}\mj}{\bj}\lb \frac{1 + \gamma_{j}}{\gamma_{j}} + \frac{2\bj}{\mj}\rb\E \|\ej\|_{2}^{2}.\nonumber
\end{align}
Since $1 \ge 2(1 + \gamma_{j})\etaj L$ by \eqref{eq:etaj_cond}, by convexity of $f$,
\[\la\nabla f(\td{x}_{j}), \td{x}_{j} - x^{*}\ra\ge f(\td{x}_{j}) - f(x^{*}).\]
By \eqref{eq:one_epoch_L24} and Lemma \ref{lem:ejL2}, 
\begin{align}
\lefteqn{ 2\etaj \E (f(\td{x}_{j}) - f(x^{*}))} \nonumber\\
& \le \frac{\bj}{\mj}\lb\E\|\td{x}_{j-1} - x^{*}\|_{2}^{2} - \E\|\td{x}_{j} - x^{*}\|_{2}^{2}\rb + 2\etaj \gamma_{j}\E (f(\td{x}_{j-1}) - f(\td{x}_{j}))\nonumber\\
&  + 2(1+\gamma_{j})\lb \frac{1 + \gamma_{j}}{\gamma_{j}} + \frac{2\bj}{\mj}\rb\frac{\etaj^{2}\mj}{\bj \Bj}\lb 2L \E (f(\td{x}_{j-1}) - f(x^{*})) + \H I(\Bj < n)\rb.\nonumber
\end{align}
Rearranging the terms, we have
\begin{align}
  \lefteqn{2\etaj(1 + \gamma_{j})\E (f(\td{x}_{j}) - f(x^{*})) + \frac{\bj}{\mj}\E \|\td{x}_{j} - x^{*}\|_{2}^{2}}\nonumber\\
&\le  2\etaj \lb \gamma_{j} + 2(1+\gamma_{j})\lb \frac{1 + \gamma_{j}}{\gamma_{j}} + \frac{2\bj}{\mj}\rb\frac{\etaj L \mj}{\bj\Bj}\rb \E (f(\td{x}_{j-1}) - f(x^{*})) \nonumber\\
& \hspace{.5cm} + \frac{\bj}{\mj}\E \|\td{x}_{j-1} - x^{*}\|_{2}^{2} + 2(1+\gamma_{j})\lb \frac{1 + \gamma_{j}}{\gamma_{j}} + \frac{2\bj}{\mj}\rb\frac{\etaj^{2}\mj}{\bj \Bj} \H I(\Bj < n).\nonumber
\end{align}
Dividing both sides by $2\etaj (1 + \gamma_{j})$ and recalling the definition that $\Gamma_{j} = \gamma_{j} / (1 + \gamma_{j})$, we complete the proof.
% Since $\etaj L \le \Gamma_{j}\bj$, we have
% \begin{align*}
%   \Gamma_{j} + 2\lb\frac{1}{\Gamma_{j}} + \frac{2\bj}{\mj}\rb\frac{\etaj L \mj}{\bj\Bj} \le \Gamma_{j} + \frac{2\mj}{\Bj} + \frac{4\bj}{\Bj}\Gamma_{j} = \Gamma_{j}\lb 1 + \frac{4\bj}{\Bj} + \frac{2\mj}{\Gamma_{j}\Bj}\rb.
% \end{align*}
% This completes the proof.
\end{proof}

\begin{proof}[\textbf{Proof of Theorem \ref{thm:one_epoch_simplify_L2}}]
  Let $\Gamma_{j}\equiv \Gamma$ in Theorem \ref{thm:one_epoch_L2}.  Under condition \eqref{eq:etaj_mj_bj_Bj_cond}, \eqref{eq:etaj_cond} is satisfied because $\Gamma \le 1 / 4$. Moreover,
  \begin{align*}
    \Gamma + \lb\frac{1}{\Gamma} + \frac{2\bj}{\mj}\rb\frac{2\etaj L \mj}{\bj\Bj}\le \Gamma + \lb\frac{1}{\Gamma} + 2\rb\Gamma^{2}\le 3\Gamma,
  \end{align*}
and 
\[\frac{1}{\Gamma} + \frac{2\bj}{\mj}\le \frac{1}{\Gamma} + 2\le \frac{3}{2\Gamma}.\]
By assumption \textbf{A}2,
\[\frac{1}{4}(f(\td{x}_{j}) - f(x^{*}))\ge \frac{\mu}{8}\|\td{x}_{j} - x^{*}\|_{2}^{2}.\]
By Theorem \ref{thm:one_epoch_L2} we have
\begin{align}
  \lefteqn{\frac{3}{4}\E (f(\td{x}_{j}) - f(x^{*})) + \lb\frac{\bj (1 - \Gamma)}{2\etaj \mj} + \frac{\mu}{8}\rb\E \|\td{x}_{j} - x^{*}\|_{2}^{2}} \nonumber\\
& \le 3\Gamma \E (f(\td{x}_{j-1}) - f(x^{*})) + \frac{\bj (1 - \Gamma)}{2\etaj \mj}\E \|\td{x}_{j-1} - x^{*}\|_{2}^{2} + \frac{3}{2\Gamma}\frac{\etaj \mj}{\bj \Bj} \H I(\Bj < n)\nonumber.
\end{align}
The proof is completed by multiplying by $4 / 3$ in both sides and that $\DH = \H / L$.
\end{proof}

\subsubsection{Other Proofs}

\begin{proof}[\textbf{Proof of Theorem \ref{thm:mainL2}}]
  By Theorem \ref{thm:one_epoch_simplify_L2} with $\Gamma = 1 / 4\alpha^{1/\xi}$, 
  \begin{align}
      \lefteqn{\E (f(\td{x}_{j}) - f(x^{*})) + \lb\frac{2 b (1 - \Gamma)}{3\eta \mj} + \frac{\mu}{6}\rb\E \|\td{x}_{j} - x^{*}\|_{2}^{2}} \nonumber\\
& \le \frac{1}{\alpha^{1/\xi}} \E (f(\td{x}_{j-1}) - f(x^{*})) + \frac{2b (1 - \Gamma)}{3\eta \mj}\E \|\td{x}_{j-1} - x^{*}\|_{2}^{2} + \frac{2\eta L}{\Gamma b}\frac{\mj}{\Bj} \DH I(\Bj < n)\nonumber\\
& = \frac{1}{\alpha^{1/\xi}} \E (f(\td{x}_{j-1}) - f(x^{*})) + \frac{2b(1 - \Gamma)}{3\eta \mj}\E \|\td{x}_{j-1} - x^{*}\|_{2}^{2}n + \frac{2\eta L m_{0}}{\Gamma b B_{0}}\frac{1}{\alpha^{j}} \DH I(j < T_{n}^{*}).\nonumber
  \end{align}
Let 
\[\Lj = \E (f(\td{x}_{j}) - f(x^{*})) + \lb\frac{2 b (1 - \Gamma)}{3\eta \mj} + \frac{\mu}{6}\rb\E \|\td{x}_{j} - x^{*}\|_{2}^{2}.\]
Then 
\begin{align}
  \Lj &\le \max\left\{\frac{1}{\alpha^{1/\xi}}, \frac{\frac{2 b (1 - \Gamma)}{3\eta \mj}}{\frac{2 b (1 - \Gamma)}{3\eta \mjj} + \frac{\mu}{6}}\right\}\Ljj + \frac{2\eta L m_{0}}{\Gamma b B_{0}}\frac{1}{\alpha^{j}} \DH I(j < T_{n}^{*})\nonumber\\
& = \max\left\{\frac{1}{\alpha^{1/\xi}}, \frac{1}{\frac{\mj}{\mjj} + \frac{\eta\mu \mj}{4b (1 - \Gamma)}}\right\}\Ljj + \frac{2\eta L m_{0}}{\Gamma b B_{0}}\frac{1}{\alpha^{j}} \DH I(j < T_{n}^{*})\nonumber\\
& = \max\left\{\frac{1}{\alpha^{1/\xi}}, \frac{1}{\alpha + \eta \mu \alpha^{j}\frac{m_{0}}{4b (1 - \Gamma)}}\right\}\Ljj + \frac{2\eta L m_{0}}{\Gamma b B_{0}}\frac{1}{\alpha^{j}} \DH I(j < T_{n}^{*})\nonumber\\
& \le \max\left\{\frac{1}{\alpha^{1/\xi}}, \frac{1}{\alpha + \eta \mu \alpha^{j + 1/\xi}}\right\}\Ljj + \frac{2\eta L m_{0}}{\Gamma b B_{0}}\frac{1}{\alpha^{j}} \DH I(j < T_{n}^{*}), \label{eq:mainL21}
\end{align}
where the last line uses the condition that
\[\frac{m_{0}}{4b (1 - \Gamma)}\ge \frac{m_{0}}{4b}\ge \frac{1}{4\Gamma} = \alpha^{1/\xi}.\]
For any $j\ge 0$,
\[\max\left\{\frac{1}{\alpha^{1/\xi}}, \frac{1}{\alpha + \eta\mu \alpha^{j + 1/\xi}}\right\}\le \max\left\{\frac{1}{\alpha^{1/\xi}}, \frac{1}{\alpha}\right\}\le \frac{1}{\alpha}.\]
When $j\ge T_{\kappa}^{*}$, we have $\alpha^{j}\ge \kappa / \eta L = 1 / \eta \mu$, and thus
\[\max\left\{\frac{1}{\alpha^{1/\xi}}, \frac{1}{\alpha + \eta\mu \alpha^{j + 1/\xi}}\right\}\le \frac{1}{\alpha^{1/\xi}}.\]
In summary, 
\begin{equation*}
  \max\left\{\frac{1}{\alpha^{1/\xi}}, \frac{1}{\alpha + \eta\mu \alpha^{j + 1/\xi}}\right\}\le\lambda_{j}^{-1}.
\end{equation*}
Plugging this into \eqref{eq:mainL21}, we conclude that
\begin{equation}
  \label{eq:mainL23}
  \Lj\le \lambda_{j}^{-1}\Ljj + \frac{2\eta L m_{0}}{\Gamma b B_{0}}\frac{1}{\alpha^{j}} \DH I(j < T_{n}^{*}).
\end{equation}

Finally we prove the following statement by induction.
\begin{equation}\label{eq:mainL24}
\L_{T} \le \Lambda_{T}^{-1} \L_{0} + C_{0}\td{\Lambda}_{T}^{-1}\DH (T\wedge T_{n}^{*}),\quad \mbox{where }C_{0} = \frac{2\eta L m_{0}}{\Gamma b B_{0}}.
\end{equation}
It is obvious that \eqref{eq:mainL24} holds for $T = 0$. Suppose it holds for $T - 1$, then by \eqref{eq:mainL23},
\begin{align*}
  \lefteqn{\L_{T}\le \lambda_{T}^{-1} \L_{T-1} + C_{0}\frac{\DH}{\alpha^{T}}I(T < T_{n}^{*})} \\
& \le \lambda_{T}^{-1}\lb \Lambda_{T-1}^{-1} \L_{0} + C_{0}\td{\Lambda}_{T-1}^{-1}\DH ((T - 1)\wedge T_{n}^{*})\rb + C_{0}\frac{\DH}{\alpha^{T}}I(T < T_{n}^{*})\\
  & = \Lambda_{T}^{-1}\L_{0} + C_{0}\DH\lb \lambda_{T}^{-1}\td{\Lambda}_{T-1}^{-1} ((T - 1)\wedge T_{n}^{*}) + \alpha^{-T}I(T < T_{n}^{*})\rb\\
 & \le \Lambda_{T}^{-1}\L_{0} + C_{0}\DH\lb \td{\Lambda}_{T}^{-1} ((T - 1)\wedge T_{n}^{*}) + \alpha^{-T}I(T < T_{n}^{*})\rb,
\end{align*}
where the last line uses the fact that $\td{\lambda}_{T}\le \lambda_{T}$ for all $T > 0$. If $T < T_{n}^{*}$, then $\td{\Lambda}_{T} = \alpha^{-T}$ and thus
\[\td{\Lambda}_{T}^{-1} ((T - 1)\wedge T_{n}^{*}) + \alpha^{-T}I(T < T_{n}^{*}) = \td{\Lambda}_{T}^{-1} (T\wedge T_{n}^{*}).\]
If $T > T_{n}^{*}$, 
\[\td{\Lambda}_{T}^{-1} ((T - 1)\wedge T_{n}^{*}) + \alpha^{-T}I(T < T_{n}^{*}) = \td{\Lambda}_{T}^{-1} ((T - 1)\wedge T_{n}^{*})\le \td{\Lambda}_{T}^{-1} (T\wedge T_{n}^{*}).\]
Therefore, \eqref{eq:mainL24} is proved. The proof is then completed by noting that
\[\L_{T}\ge \E (f(\td{x}_{T}) - f(x^{*}))\]
and 
\begin{align*}
  \L_{0} & = \E (f(\td{x}_{0}) - f(x^{*})) + \lb\frac{2b(1 - \Gamma)}{3\eta m_{0}} + \frac{\mu}{6}\rb\E \|\td{x}_{0} - x^{*}\|_{2}^{2}\\
& \stackrel{(i)}{\le} \lb \frac{1}{2} + \frac{2b (1 - \Gamma)}{3\eta L m_{0}} + \frac{\mu}{6L}\rb\Dx \stackrel{(ii)}{\le} \lb \frac{2}{3} + \frac{1}{6\eta L}\rb\Dx\\
& \stackrel{(iii)}{\le} \lb \frac{1}{4\eta L} + \frac{1}{6\eta L}\rb\Dx  \le \frac{\Dx}{2\eta L}.
\end{align*}
where \emph{(i)} uses assumption \textbf{A}1 and the definition of $\Dx$, \emph{(ii)} uses the fact that $\mu \le L$ and the condition $b / m_{0}\le \Gamma \le 1 / 4$, \emph{(iii)} uses the fact that $\Gamma \le 1 / 4$ and thus $\eta L \le (1 - \Gamma) / 2 \le 3 / 8$.
\end{proof}

\begin{proof}[\textbf{Proof of Theorem \ref{thm:complexityL2}}]
Let
\begin{equation}\label{eq:complexityL21}
  T^{(1)}(\eps) = \min\left\{T: \Lambda_{T}\ge \frac{\Dx}{\eps \eta L}\right\}, \quad T^{(2)}(\eps) = \min\left\{T: \td{\Lambda}_{T}\ge \frac{\DH T_{n}^{*}}{\eps}\right\}.
\end{equation}
Then for any $T \ge \max\{T^{(1)}(\eps), T^{(2)}(\eps)\}$,
\[\E (f(\td{x}_{T}) - f(x^{*}))\le \frac{\eps}{2} + \frac{2\eta L m_{0}}{\Gamma b B_{0}}\eps \le \lb\frac{1}{2} + \Gamma\rb \eps \le \eps.\]
This entails
\begin{equation}
  T(\eps)\le \max\{T^{(1)}(\eps), T^{(2)}(\eps)\}.
\end{equation}
By definition, when $T \le T_{\kappa}^{*}$,
\[\Lambda_{T}\ge \alpha^{T},\]
and when $T > T_{\kappa}^{*}$, since $\xi < 1$,
\[\Lambda_{T} = \alpha^{\lfloor T_{\kappa}^{*}\rfloor} \alpha^{(T - \lfloor T_{\kappa}^{*}\rfloor) / \xi}\ge \alpha^{T_{\kappa}^{*}} \alpha^{(T - T_{\kappa}^{*}) / \xi} = \frac{\kappa}{\eta L}\alpha^{(T - T_{\kappa}^{*}) / \xi}.\]
As a result,
\begin{equation}\label{eq:TT1T2}
T_{1}(\eps) \le \min\left\{\frac{\log\lb\frac{\Dx}{\eps\eta L}\rb}{\log \alpha}, T_{\kappa}^{*} + \xi \frac{\log\lb\frac{\Dx}{\eps \kappa}\rb}{\log \alpha}\right\}
\end{equation}
and 
\begin{equation*}
\alpha^{T_{1}(\eps)} \le \min\left\{\frac{\Dx}{\eps\eta L}, \frac{\kappa}{\eta L}\lb\frac{\Dx}{\eps \kappa}\rb^{\xi}_{*}\right\}.
\end{equation*}
Similarly, when $T \le T_{n}^{*}\vee T_{\kappa}^{*}$
\[\td{\Lambda}_{T} = \alpha^{T},\]
and when $T > T_{n}^{*}\vee T_{\kappa}^{*}$,
\[\td{\Lambda}_{T}\ge \alpha^{T_{n}^{*}\vee T_{\kappa}^{*}}\alpha^{(T - T_{n}^{*} \vee T_{\kappa}^{*}) / \xi}\ge \alpha^{T_{n}^{*}+ T_{\kappa}^{*}}\alpha^{(T - T_{n}^{*} - T_{\kappa}^{*}) / \xi}.\]
Thus,
\[\td{\Lambda}_{T}\ge \lb\sqrt{\frac{n}{B_{0}}} + \frac{1}{\eta \mu}\rb\alpha^{(T - T_{n}^{*} - T_{\kappa}^{*}) / \xi} \ge \frac{\td{\kappa}}{\sqrt{B_{0}} + \eta L}\alpha^{(T - T_{n}^{*} - T_{\kappa}^{*}) / \xi}.\]
As a result,
\begin{equation*}
T_{2}(\eps) \le \min\left\{\frac{\log\lb\frac{\DH T_{n}^{*}}{\eps}\rb}{\log \alpha}, T_{n}^{*} + T_{\kappa}^{*} + \xi \frac{\log\lb\frac{\DH T_{n}^{*}(\sqrt{B_{0}} + \eta L)}{\eps\td{\kappa}}\rb}{\log \alpha}\right\},
\end{equation*}
and 
\begin{equation*}
\alpha^{T_{2}(\eps)} \le \min\left\{\frac{\DH T_{n}^{*}}{\eps}, \td{\kappa} \lb\frac{\DH T_{n}^{*}(\sqrt{B_{0}} + \eta L)}{\eps \td{\kappa}}\rb^{\xi}_{*}\right\}.
\end{equation*}
In summary, by \eqref{eq:TT1T2}
\begin{align}
  \lefteqn{\alpha^{T(\eps)}\le \alpha^{T_{1}(\eps)} + \alpha^{T_{2}(\eps)}} \nonumber\\
& \le  \min\left\{\frac{\Dx}{\eps\eta L}, \frac{\kappa}{\eta L}\lb\frac{\Dx}{\eps \kappa}\rb^{\xi}_{*}\right\} + \min\left\{\frac{\DH T_{n}^{*}}{\eps}, \td{\kappa}\lb\frac{\DH T_{n}^{*}(\sqrt{B_{0}} + \eta L)}{\eps \td{\kappa}}\rb^{\xi}_{*}\right\}\nonumber\\
% & = O\lb\min\left\{\frac{\D}{\eps}, \kappa\lb\frac{\Dx}{\eps \kappa}\rb^{\xi}_{*} + \frac{\DH}{\eps}, \kappa\lb\frac{\Dx}{\eps \kappa}\rb^{\xi}_{*} + \sqrt{n}\lb\frac{\DH}{\eps \sqrt{n}}\rb^{\xi}_{*}\right\}\rb\nonumber\\
& = O\lb\min\left\{\frac{\Dx + \DH T_{n}^{*}}{\eps}, \kappa\lb\frac{\Dx}{\eps \kappa}\rb^{\xi}_{*} + \frac{\DH T_{n}^{*}}{\eps}, \td{\kappa}\lb\frac{\Dx + \DH T_{n}^{*}}{\eps \td{\kappa}}\rb^{\xi}_{*}\right\}\rb,\nonumber
\end{align}
where the last line uses the monotonicity of the mapping $x\mapsto x^{1 - \xi}$. By definition, $T_{n}^{*} = O(\log n) = \td{O}(1)$. As a result, $\Dx + \DH T_{n}^{*} = \td{O}(\max\{\Dx, \DH\}) = \td{O}(\D)$ and thus
\[\alpha^{T(\eps)} = \td{O}\lb\min\left\{\frac{\D}{\eps}, \kappa\lb\frac{\Dx}{\eps \kappa}\rb^{\xi}_{*} + \frac{\DH}{\eps}, \td{\kappa}\lb\frac{\D}{\eps \td{\kappa}}\rb^{\xi}_{*}\right\}\rb.\]
The proof is then completed by replacing $\alpha^{T(\eps)}$ by $A(\eps)$ and equation \ref{eq:Ecompeps}.
\end{proof}

\section{Mirror-Proximal SCSG for Composite Problems}\label{sec:ana_general}
In this section we extend SCSG to composite problems in non-Euclidean spaces. Throughout this section we deal with problem \eqref{eq:obj_finite_sum} with $\mathcal{X}$ assumed to be a \emph{subset} of a Hilbert space $\X_{0}$, equipped with an inner product $\la\cdot, \cdot\ra$. Let $\|\cdot\|_{2}$ denote the norm induced by the inner product; i.e. $\|x\|_{2} = \sqrt{\la x, x \ra}$. For any convex function $g$, let $g^{*}$ denote the convex conjugate of $g$:
\[g^{*}(x) = \sup_{y\in \X_{0}}\la x, y\ra - g(x).\]
For any differential convex function $w$, let $\BD{w}(\cdot, \cdot)$ denote the Bregman divergence:
\[\BD{w}(x, y) = w(x) - w(y) - \la \nabla w(y), x - y\ra.\]
We denote by $\R^{+}$ the set of nonnegative reals.

We define \emph{Mirror-Proximal SCSG} as a variant of Algorithm \ref{algo:SCSGplus} designed for composite problems (with $\psi \not = 0$). The algorithm is detailed below. The only difference lies in line 9 where the gradient step is replaced by a mirror-proximal step. This is the standard extension to composite problems in general Hilbert spaces (see, e.g., \citep{duchi10, lan15, Katyusha}). Note that when $\psi(x) \equiv 0$ and $w(x) = \|x\|_{2}^{2} / 2$, Algorithm \ref{algo:mirrorproximalSCSGplus} reduces to Algorithm \ref{algo:SCSGplus}. Whenever $w(x) = \|x\|_{2}^{2} / 2$, line 9 reduces to the proximal gradient step.

Unlike most of the literature on composite mirror descent algorithms, our analysis requires a weaker condition on the distance-generating function $w(x)$. To state the condition, we define a class of functions which we refer to as \emph{Convex sup-Homogeneous Envelope Functions (CHEF)}.% ; See \citet[][e.g.]{burai05} for related notions.

\begin{algorithm}[htp]
\caption{Mirror-proximal SCSG for regularized finite-sum optimization}
\label{algo:mirrorproximalSCSGplus}
\textbf{Inputs: } Number of stages $T$, initial iterate $\td{x}_{0}$, stepsizes $(\etaj)_{j=1}^{T}$, block sizes $(\Bj)_{j=1}^{T}$, inner loop sizes $(m_{j})_{j=1}^{T}$, mini-batch sizes $(\bj)_{j=1}^{T}$.

\textbf{Procedure}
\begin{algorithmic}[1]
  \For{$j = 1, 2, \cdots, T$}
  \State Uniformly sample a batch $\I_{j}\subset \{1,\ldots,n\}$ with $|\I_{j}| = \Bj$;
  \State $\mu_{j}\gets \nabla f_{\I_{j}}(\td{x}_{j-1})$;
  \State $x_{0}\gets \td{x}_{j-1}$;
  \State Generate $N_{j}\sim \mathrm{Geom}\lb\frac{\mj}{\mj + \bj}\rb$;
  \For{$k = 1, 2, \cdots, N_{j}$}
  \State Uniformly sample a batch $\sI_{k-1}\subset \{1,\ldots, n\}$ with $|\sI_{k-1}| = \bj$;
  \State $\nuj_{k - 1} \gets \nabla f_{\sI_{k-1}}(\xj_{k-1}) - \nabla f_{\sI_{k-1}}(\xj_{0}) + \mu_{j}$;
  \State $\xj_{k}\gets \argmin_{x\in \mathcal{X}} \lb \la \nuj_{k-1}, y\ra + \psi(y) + \frac{1}{\etaj}\BD{w}(y, \xj_{k-1})\rb$;
  \EndFor
  \State $\td{x}_{j}\gets x_{N_{j}}$;
  \EndFor
\end{algorithmic}
\textbf{Output: } $\td{x}_{T}$.
\end{algorithm}

\begin{definition}
 Given any increasing function $g: \R^{+}\mapsto \R^{+}$ such that
\[\lim_{\lambda\rightarrow 0}g^{-1}(\lambda) = 0,\]
 a function $G: \X_{0}\mapsto \R^{+}$ is a CHEF with parameters $(g(\cdot), C_{G})$ if 
  \begin{enumerate}[\textbf{C}1]
  \item $G$ is non-negative with $G(0) = 0$, convex and symmetric in the sense that $G(w) = G(-w)$ for any $w\in \X_{0}$;
  \item for any $w\in \X_{0}$ and $\lambda > 0$,
\[G(\lambda w)\ge \lambda g(\lambda) G(w);\]
  \item $G^{*}$, the convex conjugate of $G$, satisfies a \textbf{generalized Nemirovsky inequality} in the sense that for any set of independent mean-zero $\X_{0}$-valued random vectors $Z_{1}, \ldots, Z_{m}$, 
\[\E G^{*}\lb\sum_{j=1}^{m}Z_{j}\rb\le C_{G}\sum_{j=1}^{m}\E G^{*}(Z_{j}).\]
  \end{enumerate}

\end{definition}

% Note that we do not require a CHEF to be symmetric in the sense that $H(x) = H(-x)$. 
Note that $G^{*}$ is non-negative, since for any $w \in \X_{0}$, 
\[G^{*}(w) = \sup_{x\in \X_{0}} \la x, w\ra - G(x) \ge \la 0, w\ra - G(0) = 0.\]
Our first condition is imposed on the Bregman divergence induced by $w(x)$.
\begin{enumerate}[\textbf{B}1]
\item There exists a CHEF 
such that for any $x, y\in \X$,
\[\BD{w}(x, y)\ge G(x - y).\]
\end{enumerate}
In the literature (see, e.g., \citep{beck03, duchi10, Katyusha}), it is common to consider a special case where 
\[G(w) = G_{2}(w; \|\cdot\|) = \frac{1}{2}\|w\|^{2},\]
where $\|\cdot\|$ can be any norm, not necessarily $\|\cdot\|_{2}$, on $\X_{0}$. \cite{srebro11} considered a more general class of $G$'s in the form of 
\[G(w) = G_{q}(w; \|\cdot\|) = \frac{1}{q}\|w\|^{q}.\]
It is clear that $G_{q}$ satisfies \textbf{C}1 and \textbf{C}2 for any $q > 1$ with $g(\lambda) = \lambda^{q - 1}$.

To see that $G_{q}(w; \|\cdot\|)$ satisfies \textbf{C}3, we first consider the case where $q = 2$, where $\X_{0} = \R^{d}$ is the Euclidean space and where $\|\cdot\| = \|\cdot\|_{r}$ for some $r\ge 1$, with 
\[\|x\|_{r} = \left\{
    \begin{array}{ll}
      \lb \sum_{i=1}^{n}|x_{i}|^{r}\rb^{1/r} & (1\le r < \infty)\\
      \max_{i=1}^{n}|x_{i}| & (r = \infty)
    \end{array}
\right..\]
Then $\|\cdot\|_{*} = \|\cdot\|_{r'}$ where $r' = r / (r - 1)$. By Lemma \ref{lem:dual_Gq}, 
\[G_{2}^{*}(x; \|\cdot\|_{r}) = \frac{1}{2}\|x\|_{r'}^{2}.\]
By Nemirovsky's inequality (Theorem 2.2 of \cite{dumbgen10}), for any independent mean-zero random vectors $Z_{1}, \ldots, Z_{m}\in \R^{d}$, 
\[\E \left\|\sum_{j=1}^{n} Z_{j}\right\|_{r'}^{2}\le K_{\mathrm{Nem}}(d, r')\sum_{j=1}^{n}\E \|Z_{j}\|_{r'}^{2},\]
where 
\[K_{\mathrm{Nem}}(d, r')\le \min\{r' - 1, 2e\log d\}.\]
Thus, whenever $r' = O(1)$, $K_{\mathrm{Nem}}(d, r') = O(1)$. Even when $r' = \infty$, in which case $r = 1$, $K_{\mathrm{Nem}}(d, r')$ scales as $\log d$. 

Generally, given $G(x) = \|x\|^{q} / q$ for $q \in (1, 2)$ and a norm $\|\cdot\|$ on a general Hilbert space $\X$, Lemma \ref{lem:dual_Gq} implies that $G^{*}(x) = \|x\|_{*}^{p} / p$ where $ p = q / (q - 1)$. Then the property \textbf{C}3 is equivalent to the condition that $\X_{0}$ has \emph{Martingale Type} $p$ (see, e.g., \citep{pisier75}). In particular, when $\X_{0} = \R^{d}$ and $\|\cdot\| = \|\cdot\|_{r}$ with $r \le q$, we prove in Proposition \ref{prop:Hanner} that $G^{*}(x)$ satisfies the property \textbf{C}3, using Hanner's inequality \citep{hanner56}. In summary, the property \textbf{C}3 is satisfied in almost all cases that have been commonly studied in the literature on mirror descent methods.

% By \cite[][Theorem 1]{wenzel02}, this is true if $X^{*}$ is super-reflexive; see \citep[e.g.][]{pisier75, wenzel02} for the definition of super-reflexive Banach spaces. By \cite[][Theorem 1]{james72}, $X^{*}$ is super-reflexive iff $X$ is super-reflexive

Besides assumption \textbf{B}1 we need the analogous assumptions of \textbf{A}1 and \textbf{A}2 for the smoothness and strong convexity of the objectives. 
\begin{enumerate}[\textbf{B}1]
\setcounter{enumi}{1}
\item $\displaystyle 0\le \BD{f_{i}}(x, y)\le L G(x - y)$
for all $i$ and $x, y\in \X$;
\item $\displaystyle F(x) - F(x^{*})\ge \mu \BD{w}(x^{*}, x)$
for some $\mu\ge 0$.
\end{enumerate}
It is easy to see that assumptions \textbf{B}2 and \textbf{B}3 reduce to \textbf{A}1 and \textbf{A}2 when $\X = \R^{d}, G(x) = w(x) = \|x\|_{2}^{2} / 2$. Note that \textbf{B}3 only requires strong convexity at $x^{*}$; it does not requires global strong convexity. % It is implied under \textbf{B}1 by a slightly weaker assumption that
% \[\BD{f}(x, x^{*})\ge \mu G(x - y).\]
% In the case $G(x) = \|x\|_{2}^{q} / q$, the above condition is equivalent to uniform convexity \cite[e.g.][]{}. In summary, Assumptions \textbf{B}2 and \textbf{B}3 are weaker than common assumptions in the literature.

Finally, we modify the definitions of $\Dx, \DH$ as
\[\Dx = L \E \BD{w}(x^{*}, \td{x}_{0}), \quad \DH = \frac{L}{n}\sum_{i=1}^{n}G^{*}\lb \frac{\nabla f_{i}(x^{*}) - \nabla f(x^{*})}{L}\rb,\]
where $x^{*}$ is the optimum of $F$. It is straightforward to show that $\Dx$ and $\DH$ coincides with \eqref{eq:DxDH} up to a constant two when $\X_{0} = \X = \R^{d}$ and $w(x) = \|x\|_{2}^{2} / 2$. We also define an extra quantity $\DF$ as
\begin{equation*}
  \DF = \E [F(\td{x}_{0}) - F(x^{*})].
\end{equation*}
In the unregularized case, assumption \textbf{B}1 and \textbf{B}2 imply that $\DF \le \Dx$. However, this comparison may not hold in the regularized case. Finally we re-define $\D$ as the maximum of $\Dx$, $\DH$ and $\DF$.

\subsection{Main results}
Similar to the unregularized case, we present results on the one-epoch analysis, the multiple-epoch analysis and the complexity analysis. The results are almost the same as those in Section \ref{sec:ana} though the proofs are much more involved. All proofs are relegated to Appendix \ref{app:ana_general}.
\begin{theorem}\label{thm:one_epoch_simplify}
Fix any $\xi > 0$ and $\Gamma = 1 / 4\alpha^{1/\xi}$. Assume that
\[6\etaj L\le \min \left\{1, \bj g\lb \frac{\Gamma}{3C_{G}}\rb\right\}, \quad \mj \ge \max\{\bj, 4C_{G}\} / \Gamma,\]
and
\[ \Bj \ge \frac{5}{8\Gamma}g\lb\frac{\Gamma}{3C_{G}}\rb \frac{\mj}{g(1 / \mj)}.\]
Then under assumptions \textbf{B}1 - \textbf{B}3,
  \begin{align*}
    &\E (F(\td{x}_{j}) - F(x^{*})) + \lb\frac{4\bj (1 - \Gamma)}{3\eta}\frac{1}{\mj} + \frac{\mu}{3}\rb\E \BD{w}(x^{*}, \td{x}_{j})\\
\le & 4\Gamma \E (F(\td{x}_{j-1}) - F(x^{*})) + \frac{4\bj (1 - \Gamma)}{3\eta}\frac{1}{\mj}\E \BD{w}(x^{*}, \td{x}_{j-1}) + \frac{\DH}{6\alpha^{j}} I(\Bj < n).
  \end{align*}
\end{theorem}

Theorem \ref{thm:one_epoch_simplify} and Theorem \ref{thm:one_epoch_simplify_L2} give almost the same result, up to constants, except that Theorem \ref{thm:one_epoch_simplify_L2} has an additional term $\frac{\eta L}{b}$ in the coefficient of  $\DH$. In the cases where $\eta$ is small and $b$ is large, Theorem \ref{thm:one_epoch_simplify_L2} gives a better guarantee. However, in our settings for SCSG, $\eta L$ and $b$ are both taken as $O(1)$ and thus the theorems yield the same results up to the constant. 

To set the parameters for SCSG in this general case, we still take a constant stepsize, a constant mini-batch size and a geometrically increasing sequence for $\mj$. In contrast, $\Bj$ should scale as $\mj / g(1 / \mj)$. This coincides with Theorem \ref{thm:mainL2} since $g(x) = x$ in the unregularized case with the usual strong convexity condition (assumption \textbf{A}2). 

\begin{theorem}\label{thm:main}
 Fix and constant $\alpha > 1$, $m_{0} > 0$ and  $\xi \in (0, 1)$. Let 
 \[\etaj \equiv \eta, \quad \bj \equiv b, \quad \mj = m_{0}\alpha^{j}, \quad \Bj = \left\lceil \frac{5\alpha^{1/\xi}}{2}g\lb \frac{1}{12 C_{G}\alpha^{1/\xi}}\rb\frac{\mj}{g(1 / \mj)}\right\rceil.\]
Assume that
\[m_{0} \ge 4\alpha^{1/\xi}\max\{b, 4C_{G}\}, \quad 6\eta L\le \min \left\{1,  g\lb \frac{1}{12C_{G}\alpha^{1/\xi}}\rb b\right\}.\]
Then under assumptions \textbf{B}1 - \textbf{B}3, 
\[\E (F(\td{x}_{j}) - F(x^{*}))\le \Lambda_{T}^{-1}\lb\DF + \frac{\Dx}{3\eta L}\rb + \td{\Lambda}_{T}^{-1}\frac{\DH}{6}(T\wedge T_{n}^{*}),\]
where $\Lambda_{T}$ and $\td{\Lambda}_{T}$ are defined as in Theorem \ref{thm:mainL2}.
\end{theorem}

Theorem \ref{thm:main} gives almost the same result as Theorem \ref{thm:mainL2}, except that the second term is loose up to a term $\frac{\eta L}{b}$. As mentioned before, this is a constant in our setting and thus the gap is negligible in terms of the theoretical complexity. 

Applying the same argument as in Theorem \ref{thm:complexityL2}, we can derive the theoretical complexity. Again it coincides with Theorem \ref{thm:complexityL2} in the unregularized case with the usual strong convexity condition (i.e., assumption \textbf{A}2). 
\begin{theorem}\label{thm:complexity}
Under the specification of Theorem \ref{thm:main}, we have 
\[\E\comp(\eps) = O\lb \frac{A(\eps)}{g(1 / A(\eps))}\wedge \lb A(\eps) + n\log A(\eps)\rb\rb,\]
where $\td{\kappa} = \alpha^{T_{n}^{*}} + \kappa$ and
\[A(\eps) = \td{O}\lb\min\left\{\frac{\D}{\eps}, \kappa\lb\frac{\Dx + \DF}{\eps \kappa}\rb^{\xi}_{*} + \frac{\DH}{\eps}, \td{\kappa}\lb \frac{\D}{\eps\td{\kappa}}\rb_{*}^{\xi}\right\}\rb.\]
\end{theorem}

Interestingly, in the uniformly convex case \citep{juditsky2014deterministic}, where 
\[\BD{w}(x, y) \ge G(x - y) = \frac{1}{q}\|x - y\|^{q},\]
we can set $g(\lambda) = \lambda^{q - 1}$. Then Theorem \ref{thm:complexity} implies that
\begin{equation}\label{eq:uniform_convex1}
\E\comp(\eps) = O\lb A(\eps)^{q} \wedge \lb A(\eps) + n\log A(\eps)\rb\rb.
\end{equation}
Recalling that $\tto$ hides the negligible terms $\lb \frac{D}{\eps \kappa}\rb^{2\xi}$ and $\td{\kappa} = \alpha^{T_{n}^{*}} + \kappa = O(n + \kappa)$, we can rewrite \eqref{eq:uniform_convex1} as
\begin{equation*}
  \E\comp(\eps) = \tto\lb \lb\frac{D}{\eps}\rb^{q} \wedge \lb\kappa^{q} + \lb\frac{\DH}{\eps}\rb^{q}\rb\wedge \lb n + \frac{D}{\eps}\rb \wedge \lb n + \kappa\rb\rb.
\end{equation*}
The first term matches the bound in \cite{srebro11}. However, the other terms have not been investigated in the literature to the best of our knowledge.

\section{Conclusions}
We have presented SCSG, a gradient-based algorithm for the convex finite-sum optimization problem, which is $\eps$-independent and almost-universal.  These properties arise from two ideas:  \emph{geometrization} and \emph{batching variance reduction}. SCSG achieves strong adaptivity to both the target accuracy and to strong convexity with complexity $\tto\lb\frac{\D^{2}}{\eps^{2}}\wedge \lb\kappa^{2} + \frac{\DH^{2}}{\eps^{2}}\rb\wedge \lb n + \frac{\D}{\eps}\rb \wedge \lb n + \kappa\rb\rb$ up to negligible terms.  This is strictly better than other existing adaptive algorithms. We also present a mirror-proximal version of SCSG for problems involving non-Euclidean geometry. Our analysis requires the Bregman divergence to be lower bounded by a CHEF, a construct which unifies and generalizes existing work on mirror-descent methods. We derive a set of technical tools to deal with CHEFs which may be of interest in other problems. 

A major direction for further research is to delineate optimal rates for algorithms that exhibit adaptivity.  Our conjecture is that the optimal complexity for a reasonably large class of algorithms that do not require knowledge of $\mu$ is
\[
\td{O}\lb \lb\sqrt{\frac{\Dx}{\eps}} + \frac{\DH^{2}}{\eps^{2}}\rb\wedge \frac{\kappa\DH}{\eps}\wedge \lb n + \sqrt{\frac{n\Dx}{\eps}}\rb \wedge \lb n + \kappa\rb\rb.
\]
We believe that momentum terms are required to achieve such a rate.

\vspace{-0.2cm}
\bibliographystyle{siamplain}
\bibliography{SCSG}

\newpage

\begin{center}
  \begin{Large}
      \textbf{Supplementary Material}
  \end{Large}
\end{center}
~\\

\appendix
\section{Technical Proofs in Section \ref{sec:ana}}\label{app:ana}

\begin{proof}[\textbf{Proof of Lemma \ref{lem:nuknorm}}]
  Using the fact that $\E \|Z\|_{2}^{2} = \E \|Z - \E Z\|_{2}^{2} + \|\E Z\|_{2}^{2}$ (for any random vector $Z$), we have
  \begin{align*}
\lefteqn{\E_{\sI_{k}}\|\nuj_{k}\|_{2}^{2}  = \E_{\sI_{k}}\|\nuj_{k} - \E_{\sI_{k}}\nuj_{k}\|_{2}^{2} + \|\E_{\sI_{k}}\nuj_{k}\|_{2}^{2}} \\
 & = \E_{\sI_{k}}\|\nabla f_{\sI_{k}}(\xj_{k}) - \nabla f_{\sI_{k}}(\xj_{0}) - (\nabla f(\xj_{k}) - \nabla f(\xj_{0}))\|_{2}^{2} + \|\nabla f(\xj_{k}) + \ej\|_{2}^{2}\\
& \le  \frac{1}{\bj n}\lb\sum_{i = 1}^{n} \|\nabla f_{i}(\xj_{k}) - \nabla f_{i}(\xj_{0})\|_{2}^{2}\rb + \|\nabla f(\xj_{k}) + \ej\|_{2}^{2},
\end{align*}
where the last inequality is obtained from Lemma \ref{lem:var_sampling}. By Lemma \ref{lem:cocoercive} with $g = f_{i}, x = \xj_{0}, y = \xj_{k}$,
\begin{align*}
\|\nabla f_{i}(\xj_{k}) - \nabla f_{i}(\xj_{0})\|_{2}^{2}&\le 2L\lb f_{i}(\xj_{0}) - f_{i}(\xj_{k}) + \la\nabla f_{i}(\xj_{k}), \xj_{k} - \xj_{0}\ra\rb.
\end{align*}
The proof is completed by the fact that $\|a + b\|_{2}^{2} = 2\|a\|_{2}^{2} + 2\|b\|_{2}^{2} - \|a - b\|_{2}^{2} \le 2\|a\|_{2}^{2} + 2\|b\|_{2}^{2}$.
\end{proof}

\begin{proof}[\textbf{Proof of Lemma \ref{lem:ejL2}}]
Using the fact that $\|a + b\|_{2}^{2} \le 2\|a\|_{2}^{2} + 2\|b\|_{2}^{2}$ and $\nabla f(x^{*}) = 0$,  we have
\begin{align*}
&  \E \|\ej\|_{2}^{2}  \le 2\E \|\ej - \nabla f_{\I_{j}}(x^{*})\|_{2}^{2} + 2\E \|\nabla f_{\I_{j}}(x^{*})\|_{2}^{2}\\
& = 2\E \left\|\frac{1}{\Bj}\sum_{i\in \Ij}(\nabla f_{i}(\td{x}_{j-1}) - \nabla f_{i}(x^{*})) - (\nabla f(\td{x}_{j-1}) - \nabla f(x^{*}))\right\|_{2}^{2}\\
& \hspace{.5cm} + 2\E \left\|\frac{1}{\Bj}\sum_{i\in \Ij}\nabla f_{i}(x^{*})\right\|_{2}^{2}.
\end{align*}
Then
\begin{align*}
&\E \left\|\frac{1}{\Bj}\sum_{i\in \Ij}(\nabla f_{i}(\td{x}_{j-1}) - \nabla f_{i}(x^{*})) - (\nabla f(\td{x}_{j-1}) - \nabla f(x^{*}))\right\|_{2}^{2}\\
& \le \frac{I(\Bj < n)}{\Bj}\, \frac{1}{n}\sum_{i=1}^{n}\|\nabla f_{i}(\td{x}_{j-1}) - \nabla f_{i}(x^{*})\|_{2}^{2}\quad (\mbox{Lemma \ref{lem:var_sampling}})\\
& \le \frac{I(\Bj < n)}{\Bj}\, \frac{2L}{n}\sum_{i=1}^{n}\lb f_{i}(\td{x}_{j-1}) - f_{i}(x^{*}) - \la\nabla f_{i}(x^{*}), \td{x}_{j-1} - x^{*}\ra\rb \quad (\mbox{Lemma \ref{lem:cocoercive}})\\
& = \frac{I(\Bj < n)}{\Bj}\, 2L(f(\td{x}_{j - 1}) - f(x^{*}) - \la \nabla f(x^{*}), \td{x}_{j-1} - x^{*}\ra)\\
& = \frac{I(\Bj < n)}{\Bj}\, 2L(f(\td{x}_{j - 1}) - f(x^{*})).
\end{align*}
On the other hand, by Lemma \ref{lem:var_sampling} again, we obtain that 
\[\E \left\|\frac{1}{\Bj}\sum_{i\in \Ij}\nabla f_{i}(x^{*})\right\|_{2}^{2}\le \frac{\H\, I(\Bj < n)}{\Bj}.\]
Putting the pieces together we prove the result.
\end{proof}

To prove the remaining two lemmas, we need an extra lemma that justifies the condition $\E |D_{N}| < \infty$ when applying geometrization (Lemma \ref{lem:geom}) in our proofs for different choices of $\{D_{k}\}$. The proof is distracting and relegated to the end of this section. 
\begin{lemma}\label{lem:geom_finite_L2}
 Assume that $\etaj L \le 1 / 2$. Then for any $j$, 
 \[\E \|\td{x}_{j} - x^{*}\|_{2}^{2} < \infty, \quad \E \|\nuj_{\Nj}\|_{2}^{2} < \infty, \quad \E |\la \ej, \td{x}_{j} - \td{x}_{j-1}\ra| < \infty.\]
\end{lemma}

\begin{proof}[\textbf{Proof of Lemma \ref{lem:dual}}]
By definition,
  \begin{align}
&\E_{\sI_{k}}\|\xj_{k+1} - u\|_{2}^{2}  = \|\xj_{k} - u\|_{2}^{2} - 2\etaj\E_{\sI_{k}}\la  \nuj_{k}, \xj_{k} - u\ra + \etaj^{2} \E_{\sI_{k}}\|\nuj_{k}\|_{2}^{2}\nonumber\\
& = \|\xj_{k} - u\|_{2}^{2} - 2\etaj\la \E_{\sI_{k}}\nuj_{k}, \xj_{k} - u\ra + \etaj^{2} \E_{\sI_{k}}\|\nuj_{k}\|_{2}^{2}\nonumber\\
& = \|\xj_{k} - u\|_{2}^{2} - 2\etaj\la \nabla f(\xj_{k}), \xj_{k} - u \ra -2\etaj \la \ej, \xj_{k} - u\ra+ \etaj^{2} \E_{\sI_{k}}\|\nuj_{k}\|_{2}^{2}.\label{eq:dual1}
  \end{align}
Let $\Ej$ denote the expectation with respect to $\I_{j}$ and $\sI_{0}, \sI_{1}, \ldots$. Then 
\[\Ej\|\xj_{k+1} - u\|_{2}^{2} = \Ej\|\xj_{k} - u\|_{2}^{2} - 2\etaj\Ej\la \nabla f(\xj_{k}), \xj_{k} - u \ra -2\etaj \Ej\la \ej, \xj_{k} - u\ra+ \etaj^{2} \E_{j}\|\nuj_{k}\|_{2}^{2}.\]
Since $u$ is independent of $\Ij$ and $\E_{\Ij}\ej = 0$, we have
\[\Ej\la \ej, u\ra = \Ej\E_{\I_{j}}\la \ej, u\ra = \Ej \la \E_{\I_{j}}\ej, u\ra = 0.\]
Similarly, since $\xj_{0}$ is also independent of $\Ij$, 
\[\Ej \la\ej, \xj_{0}\ra = 0.\]
Therefore,
\[\Ej\la \ej, \xj_{k} - u\ra  = \Ej\la\ej, \xj_{k} - \xj_{0}\ra.\]
Let $D_{k} = \E \|\xj_{k} - u\|_{2}^{2}$. By Lemma \ref{lem:geom_finite_L2}, 
\[\E D_{\Nj}\le 2\E\|\td{x}_{j} - x^{*}\|_{2}^{2} + 2\E \|u - x^{*}\|_{2}^{2} < \infty.\]
Let $k = \Nj$ in \eqref{eq:dual1} and by Lemma \ref{lem:geom} and Lemma \ref{lem:geom_finite_L2},
\begin{align*}
&2\etaj\E_{\Nj}\Ej\la \nabla f(\xj_{\Nj}), \xj_{\Nj} - u \ra\\
&   = \E_{\Nj}\lb\Ej\|\xj_{\Nj} - u\|_{2}^{2} - \Ej\|\xj_{\Nj+1} - u\|_{2}^{2}\rb  - 2\etaj \E_{\Nj}\Ej\la\ej, \xj_{\Nj} - \xj_{0}\ra + \etaj^{2} \E_{\Nj}\E_{j}\|\nuj_{\Nj}\|_{2}^{2}\\
& = \frac{\bj}{\mj}\lb \Ej\|\xj_{0} - u\|_{2}^{2} - \E_{\Nj}\Ej\|\xj_{\Nj} - u\|_{2}^{2}\rb - 2\etaj \E_{\Nj}\Ej\la\ej, \xj_{\Nj} - \xj_{0}\ra + \etaj^{2} \E_{\Nj}\E_{j}\|\nuj_{\Nj}\|_{2}^{2}.
\end{align*}

By definition, $\td{x}_{j} = \xj_{\Nj}, \td{x}_{j-1} = \xj_{0}$ and thus
\begin{align*}
&2\etaj\E \la \nabla f(\td{x}_{j}), \td{x}_{j} - u \ra = \frac{\bj}{\mj}\lb \E\|\td{x}_{j-1} - u\|_{2}^{2} - \E\|\td{x}_{j} - u\|_{2}^{2}\rb - 2\etaj \E\la\ej, \td{x}_{j} - \td{x}_{j-1}\ra + \etaj^{2} \E\|\nuj_{\Nj}\|_{2}^{2}.
\end{align*}
The proof is completed by Lemma \ref{lem:nuknorm} which yields $\etaj^{2} \E\|\nuj_{\Nj}\|_{2}^{2}\le \E\Wj$.
\end{proof}

\begin{proof}[\textbf{Proof of Lemma \ref{lem:xjxj-1L2}}]
  Let $u = \td{x}_{j-1}$, then it is independent of $\Ij$ and $\sI_{0}, \sI_{1}, \ldots$. By Lemma \ref{lem:dual},
  \begin{align}
    &\frac{\bj}{\mj}\E \|\td{x}_{j} - \td{x}_{j-1}\|_{2}^{2} - 2\etaj \E \la\ej, \td{x}_{j-1} - \td{x}_{j}\ra \le -2\etaj \E \la\nabla f(\td{x}_{j}), \td{x}_{j} - \td{x}_{j-1}\ra + \E\Wj.\label{eq:xjxj-1L21}
  \end{align}
Using the fact that $2\la a, b\ra\le \frac{1}{\beta}\|a\|_{2}^{2} + \beta\|b\|_{2}^{2}$ for any $\beta > 0$, we have
\[2\la\ej, \td{x}_{j-1} - \td{x}_{j}\ra\le \frac{\bj}{\etaj\mj}\frac{\gamma_{j}}{1 + \gamma_{j}}\|\td{x}_{j} - \td{x}_{j - 1}\|_{2}^{2} + \frac{\etaj \mj}{\bj}\frac{1 + \gamma_{j}}{\gamma_{j}}\|\ej\|_{2}^{2}.\]
This implies that
\[\frac{\bj}{\mj}\|\td{x}_{j} - \td{x}_{j-1}\|_{2}^{2} \ge \frac{1+\gamma_{j}}{\gamma_{j}} 2\etaj \E \la\ej, \td{x}_{j-1} - \td{x}_{j}\ra - \frac{\etaj^{2}\mj}{\bj}\lb\frac{1 + \gamma_{j}}{\gamma_{j}}\rb^{2}\|\ej\|_{2}^{2}.\]
By \eqref{eq:xjxj-1L21}, we obtain that
\begin{align*}
  \lefteqn{\frac{2\etaj}{\gamma_{j}}\E \la\ej, \td{x}_{j-1} - \td{x}_{j}\ra}\\
&\le -2\etaj \E \la\nabla f(\td{x}_{j}), \td{x}_{j} - \td{x}_{j-1}\ra + \E\Wj + \frac{\etaj^{2}\mj}{\bj}\lb\frac{1 + \gamma_{j}}{\gamma_{j}}\rb^{2}\E\|\ej\|_{2}^{2}.
\end{align*}
The proof is then completed by multiplying both sides by $\gamma_{j}$.
\end{proof}

\begin{proof}[\textbf{Proof of Lemma \ref{lem:geom_finite_L2}}]
We prove the first claim by induction. When $j = 0$, the claim is obvious. Suppose we prove the claim for $j - 1$, i.e. 
\[\E\|\xj_{0} - x^{*}\|_{2}^{2} = \E \|\td{x}_{j-1} - x^{*}\|_{2}^{2} < \infty.\]
Let $\yj_{k}$ be another sequence constructed as follows:
\[\yj_{0} = \xj_{0}, \quad \yj_{k} = \yj_{k-1} - \etaj\zetaj_{k}, \quad \mbox{where }\zetaj = \nabla f_{\sI_{k-1}}(\yj_{k-1}) - \nabla f_{\sI_{k-1}}(\yj_{0}) + \nabla f(\yj_{0}).\]
In other words, $\yj_{k}$ is a hypothetical sequence of iterates produced by SVRG initialized at $\xj_{0}$ and updated using the same sequence of random subsets. Let $\mathrm{Id}$ denote the identity mapping. Then 
\[\xj_{k} - \yj_{k} = \lb\mathrm{Id} - \etaj \nabla f_{\sI_{k-1}}\rb(\xj_{k-1}) - \lb\mathrm{Id} - \etaj \nabla f_{\sI_{k-1}}\rb(\yj_{k-1}) - \etaj \ej.\]
where we use the fact that $\mu_{j} = \nabla f(\xj_{0}) + \ej = \nabla f(\yj_{0}) + \ej$. Since $f_{\sI_{k-1}}$ is $L$-smooth and convex and $\etaj \le 1 / 2L$, it is well known that $\mathrm{Id} - \etaj \nabla f_{\sI_{k-1}}$ is a non-expansive operator. Thus,
\begin{align*}
  \|\xj_{k} - \yj_{k}\|_{2}&\le \|\lb\mathrm{Id} - \etaj \nabla f_{\sI_{k-1}}\rb(\xj_{k-1}) - \lb\mathrm{Id} - \etaj \nabla f_{\sI_{k-1}}\rb(\yj_{k-1})\|_{2} + \etaj \|\ej\|_{2}\\
& \le \|\xj_{k-1} - \yj_{k-1}\|_{2} + \etaj\|\ej\|_{2}.
\end{align*}
As a result,
\begin{equation}
  \label{eq:xjkyjk}
  \|\xj_{k} - \yj_{k}\|_{2}\le \|\xj_{0} - \yj_{0}\|_{2} + \etaj k \|\ej\|_{2} = \etaj k \|\ej\|_{2}.
\end{equation}
On the other hand, By \eqref{eq:SVRG1} (Theorem 1 of \cite{SVRG}) and the convexity of $f$, 
\begin{align}
\lefteqn{2\etaj(1 - 2\etaj L) \E (f(\yj_{k}) - f(x^{*})) + \E \|\yj_{k+1} - x^{*}\|_{2}^{2}}\nonumber \\
& \le 4\eta^{2}L \E (f(\yj_{0}) - f(x^{*})) + \E \|\yj_{k} - x^{*}\|_{2}^{2}.\nonumber
\end{align}
Since $f(\yj_{k}) - f(x^{*})\ge 0$ and $1 - 2\etaj L\ge 0$, we have
\[\E \|\yj_{k+1} - x^{*}\|_{2}^{2} \le 4\eta^{2}L \E (f(\yj_{0}) - f(x^{*})) + \E \|\yj_{k} - x^{*}\|_{2}^{2}.\]
As a result, 
\begin{align}
  &\E \|\yj_{k} - x^{*}\|_{2}^{2}\le 4k\eta^{2}L \E (f(\yj_{0}) - f(x^{*})) + \E\|\yj_{0} - x^{*}\|_{2}^{2}\nonumber\\
&\le (4k\etaj^{2} L^{2} + 1)\E\|\yj_{0} - x^{*}\|_{2}^{2} \le (k + 1)\E\|\xj_{0} - x^{*}\|_{2}^{2}.  \label{eq:yjkxstar}
\end{align}
Putting \eqref{eq:xjkyjk} and \eqref{eq:yjkxstar} together, and using the fact that $\|a + b\|_{2}^{2}\le 2\|a\|_{2}^{2} + 2\|b\|_{2}^{2}$, we obtain that
\begin{align}
  \E\|\xj_{k} - x^{*}\|_{2}^{2} &\le 2 \E \|\xj_{k} - \yj_{k}\|_{2}^{2} + 2 \E \|\yj_{k} - x^{*}\|_{2}^{2}\nonumber\\
& \le k^{2}\lb \etaj^{2} \E \|\ej\|_{2}^{2} + \E\|\xj_{0} - x^{*}\|_{2}^{2}\rb.\label{eq:xjkxstar}
\end{align}
By Lemma \ref{lem:ejL2}, 
\begin{equation*}
  \E \|\ej\|_{2}^{2}\le 4L\E (f(\xj_{0}) - f(x^{*})) + 2\H\le 2L^{2} \E\|\xj_{0} - x^{*}\|_{2}^{2} + 2\H.
\end{equation*}

~\\
\noindent By \eqref{eq:xjkxstar},
\[\E\|\xj_{k} - x^{*}\|_{2}^{2}\le 2k^{2}\lb \E\|\xj_{0} - x^{*}\|_{2}^{2} + \etaj^{2}\H\rb.\]
By the induction hypothesis,
\[\E\|\td{x}_{j} - x^{*}\|_{2}^{2} \le 2\E \Nj^{2} \lb \E\|\xj_{0} - x^{*}\|_{2}^{2} + \etaj^{2}\H\rb < \infty,\]
and
\[\E \|\ej\|_{2}^{2} < \infty.\]
By Lemma \ref{lem:nuknorm}, 
\begin{align*}
  \E\|\nuj_{k}\|_{2}^{2}&\le 2L\E\lb(f(\xj_{0}) - f(\xj_{k})) - \la\nabla f(\xj_{k}), \xj_{0} - \xj_{k}\ra\rb + 2\E \|\nabla f(\xj_{k})\|_{2}^{2} + 2\E\|\ej\|_{2}^{2}\\
& \le L^{2} \|\xj_{0} - \xj_{k}\|_{2}^{2} + 2L^{2}\|\xj_{k} - x^{*}\|_{2}^{2} + 2\E\|\ej\|_{2}^{2}\\
& \le 2L^{2}\|\xj_{0} - x^{*}\|_{2}^{2} + 4L^{2}\|\xj_{k} - x^{*}\|_{2}^{2} + 2\E\|\ej\|_{2}^{2}.
\end{align*}
Then the first claim yields
\[\E \|\nuj_{\Nj}\|_{2}^{2}\le 2L^{2} \E \|\td{x}_{j-1} - x^{*}\|_{2}^{2} + 4L^{2} \E \|\td{x}_{j} - x^{*}\|_{2}^{2} + 2\E\|\ej\|_{2}^{2} < \infty.\]
Finally, 
\begin{align*}
  &\E |\la \ej, \td{x}_{j} - \td{x}_{j-1}\ra|\le \frac{1}{2}\E \|\ej\|_{2}^{2} + \frac{1}{2}\E \|\td{x}_{j} - \td{x}_{j-1}\|_{2}^{2} \\
& \le \frac{1}{2}\E \|\ej\|_{2}^{2} + \E \|\td{x}_{j} - x^{*}\|_{2}^{2} + \E\|\td{x}_{j-1} - x^{*}\|_{2}^{2} < \infty.
\end{align*}
\end{proof}

\section{Technical Proofs in Section \ref{sec:ana_general}}\label{app:ana_general}
\subsection{Technical tools to handle CHEFs}
In this section we establish several technical tools to tackle CHEFs. These results can be of independent interest because they unify and generalize various fundamental results which are widely used in the analysis of first-order methods. 

The first lemma shows the sub-homogeneity of $G^{*}$.
\begin{lemma}\label{lem:CHEF}[A key property of CHEF]
For any $w \in \X_{0}$ and $\lambda \ge 0$,
\[G^{*}(\lambda w)\le \lambda g^{-1}(\lambda)G^{*}(w).\]  
\end{lemma}
\begin{proof}
When $\lambda = 0$, since $G$ is nonnegative,
\[G^{*}(0) = \sup_{x\in \X_{0}} \la 0, x\ra - G(x)\le 0.\]
We assume that $\lambda > 0$ throughout the rest of the proof. Note that $\X_{0}$ is a Hilbert space and hence a cone; i.e., $x\in \X_{0}\Longrightarrow \lambda x\in \X_{0}$ for any $\lambda > 0$. By definition,
\[G^{*}(\lambda w) = \sup_{x\in \X_{0}}\la \lambda w, x\ra - G(x) = \lambda g^{-1}(\lambda) \left[\sup_{x\in \X_{0}}\la w, \frac{x}{g^{-1}(\lambda)}\ra - \frac{G(x)}{\lambda g^{-1}(\lambda)}\right].\]
By property \textbf{C}2 of CHEF, 
\[G(x)\ge  g^{-1}(\lambda)g(g^{-1}(\lambda))G\lb \frac{x}{g^{-1}(\lambda)}\rb = \lambda g^{-1}(\lambda)G\lb \frac{x}{g^{-1}(\lambda)}\rb.\]
Thus, 
\begin{align*}
  G^{*}(\lambda w)&\le \lambda g^{-1}(\lambda) \left[\sup_{x\in \X_{0}}\la w, \frac{x}{g^{-1}(\lambda)}\ra - G\lb \frac{x}{g^{-1}(\lambda)}\rb\right]\\
& = \lambda g^{-1}(\lambda) \left[\sup_{x\in \X_{0}}\la w, x\ra - G\lb x\rb\right] = \lambda g^{-1}(\lambda)G^{*}(w).
\end{align*}
\end{proof}

The second lemma gives the Fenchel-Young inequality which involves the Bregman divergence and the corresponding dual Bregman divergence. In the special case where $w(x) = \|x\|_{2}^{2} / 2$, it reduces to the basic inequality that $2\la u, z\ra\le \alpha \|z\|_{2}^{2} + \alpha^{-1}\|u\|_{2}^{2}$ for any $\alpha > 0$.
\begin{lemma}\label{lem:bdfy}
  For any $u, y, z\in \X$ and $\alpha > 0$,
\[\la u, z\ra\le \alpha\lb \BD{w}(y + z, y) + \BD{w^{*}}\lb\frac{u}{\alpha} + \nabla w(y), \nabla w(y)\rb\rb.\]
\end{lemma}

\begin{proof}
  Let $F_{y}(z) = \alpha\BD{w}(y + z, y)$. By the Fenchel-Young inequality,
\[\la u, z\ra\le F_{y}(z) + F_{y}^{*}(u).\]
By definition,
\begin{align*}
  F_{y}^{*}(u) & = \sup_{x\in \X_{0}}\la u, x\ra - F_{y}(x) = \sup_{x\in \X_{0}}\la u, x\ra - \alpha (w(y + x) - w(y) - \la \nabla w(y), x\ra)\\
& = \alpha \lb w(y) + \sup_{x\in \X_{0}}\la \frac{u}{\alpha} + \nabla w(y), x\ra - w(y + x)\rb\\
& = \alpha \lb w(y) - \la \frac{u}{\alpha} + \nabla w(y), y\ra+ \sup_{x\in \X_{0}}\la \frac{u}{\alpha} + \nabla w(y), x + y\ra - w(y + x)\rb\\
& = \alpha \lb w(y) - \la \frac{u}{\alpha} + \nabla w(y), y\ra+ \sup_{x\in \X_{0}}\la \frac{u}{\alpha} + \nabla w(y), x\ra - w(x)\rb\\
& = \alpha \lb w(y) - \la \frac{u}{\alpha} + \nabla w(y), y\ra+ w
^{*}\lb \frac{u}{\alpha} + \nabla w(y)\rb\rb\\
& \stackrel{(i)}{=} \alpha \bigg( w(y) + w^{*}(\nabla w(y)) - \la \nabla w(y), y\ra\\
& \qquad + w^{*}\lb \frac{u}{\alpha} + \nabla w(y)\rb - w^{*}\lb \nabla w(y)\rb - \la\frac{u}{\alpha}, \nabla w^{*}(\nabla w(y))\ra\bigg)\\
& = \alpha \BD{w^{*}}\lb \frac{u}{\alpha} + \nabla w(y), \nabla w(y)\rb + \alpha ( w(y) + w^{*}(\nabla w(y)) - \la \nabla w(y), y\ra),
\end{align*}
where \emph{(i)} uses the property that
\[\nabla w^{*}(\nabla w(y)) = y.\] 
It is left to prove that 
\begin{equation}\label{eq:bdfy1}
w^{*}(\nabla w(y)) + w(y) = \la y, \nabla w(y)\ra.
\end{equation}
By the Fenchel-Young inequality,
\[\la y, \nabla w(y)\ra\le w^{*}(\nabla w(y)) + w(y).\]
On the other hand, for any $y'\in \X_{0}$, by convexity of $w$,
\[w(y')\ge w(y) + \la \nabla w(y), y' - y\ra\Longrightarrow \la y', \nabla w(y)\ra - w(y') + w(y)\le \la y, \nabla w(y)\ra.\]
Taking the supremum over $y'$ we obtain that
\[w^{*}(\nabla w(y)) + w(y)\le \la y, \nabla w(y)\ra.\]
Putting two pieces together, we prove \eqref{eq:bdfy1}. The proof is then completed.
\end{proof}

The third lemma generalizes the co-coercive property of smooth convex functions. 
\begin{lemma}\label{lem:general_cocoercive}[Generalized co-coercive property]
  Let $h$ and $H$ be arbitrary convex functions on $\X_{0}$ such that
\[\BD{h}(x, y)\le L H(x - y)\] 
for any $x, y\in \X_{0}$. Then 
\[H^{*}\lb\frac{\nabla h(y) - \nabla h(x)}{L}\rb\le \frac{1}{L}\BD{h}(x, y).\]
\end{lemma}
\begin{proof}
  Let $h_{y}(x) = \BD{h}(x, y)$. Fix any $x'\in \X_{0}$. Then
  \begin{align*}
    \lefteqn{\BD{h_{y}}(x', x) = h_{y}(x') - h_{y}(x) - \la \nabla h_{y}(x), x' - x\ra} \\
& = (h(x') - h(y) - \la \nabla h(y), x' - y\ra) - (h(x) - h(y) - \la \nabla h(y), x - y\ra) - \la\nabla h(x) - \nabla h(y), x' - x\ra\\
& = h(x') - h(x) - \la\nabla h(x), x' - x\ra = \BD{h}(x', x).
  \end{align*}
Since $h$ is convex,
\[h_{y}(x')\ge 0.\]
The above two inequalities imply that
\begin{align*}
  \BD{h}(x', x) &= h_{y}(x') - h_{y}(x) - \la \nabla h_{y}(x), x' - x\ra
 \ge -h_{y}(x) - \la \nabla h_{y}(x), x' - x\ra.
\end{align*}
Thus,
\begin{align*}
  \BD{h}(x, y) &= h_{y}(x) \ge -\la \nabla h_{y}(x), x' - x\ra - \BD{h}(x', x)\\
& = \la \nabla h(y) - \nabla h(x), x' - x\ra - \BD{h}(x', x)\\
& \ge \la\nabla h(y) - \nabla h(x), x' - x\ra - L H(x' - x)\\
& = L\lb \la\frac{\nabla h(y) - \nabla h(x)}{L}, x' - x\ra - H(x' - x)\rb
\end{align*}
Taking a supremum over $x'\in \X_{0}$, we obtain that
\[\BD{h}(x, y)\ge LH^{*}\lb\frac{\nabla h(y) - \nabla h(x)}{L}\rb\]
\end{proof}

The last lemma gives the relationship between the Bregman divergence and its dual divergence.
\begin{lemma}\label{lem:dual_BD}
  Let $h$ and $H$ be arbitrary convex functions on $\X$ and $\X_{0}$ such that
\[\BD{h}(x, y)\ge H(x - y).\] 
Assume further that $H$ is symmetric in the sense that $H(x) = H(-x)$. Then 
\[\BD{h^{*}}(x, y)\le (2\log 2) H^{*}(x - y)\le \frac{3}{2}H^{*}(x - y).\]
\end{lemma}
\begin{proof}
  Interchanging $x$ and $y$, we obtain that
\[H(y - x)\le \BD{h}(y, x).\]
By symmetry of $H$,
\[2H(x - y) = H(x - y) + H(y - x)\le \BD{h}(x, y) + \BD{h}(y, x) = \la \nabla h(x) - \nabla h(y), x - y\ra.\]
By the Fenchel-Young inequality, for any $\alpha > 0$, 
\begin{align*}
  \la \nabla h(x) - \nabla h(y), x - y\ra&\le \alpha H(x - y) + (\alpha H)^{*}\lb \nabla h(x) - \nabla h(y)\rb\\
& = \alpha \lb H(x - y) + H^{*}\lb \frac{\nabla h(x) - \nabla h(y)}{\alpha}\rb\rb.
\end{align*}
Thus, 
\[2H(x - y)\le \alpha \lb H(x - y) + H^{*}\lb \frac{\nabla h(x) - \nabla h(y)}{\alpha}\rb\rb,\]
which implies that 
\[H(x - y)\le \frac{\alpha}{2 - \alpha}H^{*}\lb\frac{\nabla h(x) - \nabla h(y)}{\alpha}\rb.\]
Replacing $x$ and $y$ by $\nabla h^{*}(x)$ and $\nabla h^{*}(y)$, respectively, and noting that $\nabla h(\nabla h^{*})(x) = x$, we obtain that
\begin{equation}\label{eq:dual_BD1}
  H(\nabla h^{*}(x) - \nabla h^{*}(y))\le \frac{\alpha}{2 - \alpha}H^{*}\lb\frac{x - y}{\alpha}\rb.
\end{equation}
By Taylor's expansion,
\begin{align*}
  \BD{h^{*}}(x, y) &= \int_{0}^{1}\frac{d}{d\theta}\BD{h^{*}}(y + \theta(x - y), y)d\theta = \int_{0}^{1}\la \nabla h^{*}(y + \theta(x - y)) - \nabla h^{*}(y), x - y\ra d\theta\\
& \stackrel{(i)}{\le} \int_{0}^{1} (H^{*}(x - y) + H(\nabla h^{*}(y + \theta(x - y)) - \nabla h^{*}(y)))d\theta\\
& \stackrel{(ii)}{\le}\int_{0}^{1} \lb H^{*}(x - y) + \frac{\theta}{2 - \theta}H^{*}\lb \frac{\theta(x - y)}{\theta}\rb\rb d\theta\\
& = H^{*}(x - y)\int_{0}^{1}\lb 1 + \frac{\theta}{2 - \theta}\rb d\theta\\
& = H^{*}(x - y)\int_{0}^{1}\frac{2}{2 - \theta}d\theta = (2\log 2) H^{*}(x - y).
\end{align*}
where \emph{(i)} uses the Fenchel-Young inequality and \emph{(ii)} uses \eqref{eq:dual_BD1}.
\end{proof}

\subsection{Lemmas}

We prove six lemmas for the one-epoch analysis. The first lemma connects the consecutive iterates within an epoch.
\begin{lemma}\label{lem:iter}
 For any $u\in \X$, 
\[\la \nuj_{k}, \xj_{k} - u\ra + \psi(\xj_{k}) - \psi(u)\le \frac{1}{\etaj}\lb \BD{w}(u, \xj_{k}) - \BD{w}(u, \xj_{k+1})\rb + \Mj_{k}\]
where 
\begin{equation}\label{eq:Mjk}
  \Mj_{k} = \la \nuj_{k}, \xj_{k} - \xj_{k+1}\ra - \frac{1}{\etaj}\BD{w}(\xj_{k+1}, \xj_{k}) + \psi(\xj_{k}) - \psi(\xj_{k+1}).
\end{equation}
\end{lemma}
\begin{proof}
  By definition (line 9 of Algorithm \ref{algo:mirrorproximalSCSGplus}), there exists $\xi\in \partial \psi(\xj_{k+1})$ such that
\[\nuj_{k} + \frac{1}{\etaj}(\nabla w(\xj_{k+1}) - \nabla w(\xj_{k})) + \xi = 0.\]
By convexity of $\psi$, 
\begin{align*}
  &\la \nuj_{k} + \frac{1}{\etaj}(\nabla w(\xj_{k+1}) - \nabla w(\xj_{k})), u - \xj_{k+1}\ra = \la\xi, \xj_{k+1} - u\ra\ge \psi(\xj_{k + 1}) - \psi(u)\\
\Longrightarrow & \la \nuj_{k}, \xj_{k+1} - u\ra + \psi(\xj_{k+1}) - \psi(u) \le -\frac{1}{\etaj} \la\nabla w(\xj_{k}) - \nabla w(\xj_{k+1}), u - \xj_{k + 1}\ra.
\end{align*}
It is easy to verify that
\[\la\nabla w(\xj_{k}) - \nabla w(\xj_{k+1}), u - \xj_{k + 1}\ra = \BD{w}(u, \xj_{k+1}) + \BD{w}(\xj_{k+1}, \xj_{k}) - \BD{w}(u, \xj_{k}).\]
The proof is completed by rearranging the terms.
\end{proof}

The second lemma provides a bound for $\Mj_{k}$ as an analogue of Lemma \ref{lem:nuknorm} for the unregularized case, though the proof is much more involved.
\begin{lemma}\label{lem:Mk}
Assume that $\etaj L \le \frac{1}{3}$. Under assumptions \textbf{B}1 and \textbf{B}2,
  \[\E_{\sI_{k}}\Mj_{k}\le 3C_{G}g^{-1}\lb\frac{3\etaj L}{\bj}\rb\BD{f}(\xj_{0}, \xj_{k}) + F(\xj_{k}) - \E_{\sI_{k}}F(\xj_{k+1}) + \E_{\sI_{k}}\la \ej, \xj_{k} - \xj_{k + 1}\ra.\]
\end{lemma}

\begin{proof}
  We decompose $\Mj_{k}$ into three components:
\[\Mj_{k} = \Mj_{k1} + \Mj_{k2} + \Mj_{k3},\]
where 
\begin{align*}
  &\Mj_{k1} = \la \nuj_{k} - \nabla f(\xj_{k}) - \ej, \xj_{k} - \xj_{k + 1}\ra - \frac{2}{3\etaj}\BD{w}(\xj_{k+1}, \xj_{k}),\\
  & \Mj_{k2} = \la \nabla f(\xj_{k}), \xj_{k} - \xj_{k+1}\ra + \psi(\xj_{k}) - \psi(\xj_{k+1}) - \frac{1}{3\etaj}\BD{w}(\xj_{k+1}, \xj_{k}),\\
  & \Mj_{k3} = \la \ej, \xj_{k} - \xj_{k+1}\ra.
\end{align*}
First we bound $\E_{\sI_{k}}\Mj_{k1}$. By Lemma \ref{lem:bdfy} with $u = -(\nuj_{k} - \nabla f(\xj_{k}) - \ej), z = \xj_{k+1} - \xj_{k}, y = \xj_{k}, \alpha = \frac{2}{3\etaj}$, 
\begin{align}
\Mj_{k1}&\le \frac{2}{3\etaj}\BD{w^{*}}\lb \frac{-3\etaj(\nuj_{k} - \nabla f(\xj_{k}) - \ej)}{2}+ \nabla w(\xj_{k}), \nabla w(\xj_{k})\rb\nonumber\\
& \stackrel{(i)}{\le} \frac{1}{\etaj}G^{*} \lb\frac{-3\etaj(\nuj_{k} - \nabla f(\xj_{k}) - \ej)}{2}\rb\label{eq:Mk11}
\end{align}
where \emph{(i)} uses Lemma \ref{lem:dual_BD} and assumption \textbf{B}1. By definition of $\nuj_{k}$,
\[\nuj_{k} - \nabla f(\xj_{k}) - \ej = \frac{1}{\bj}\sum_{i\in \sI_{k}}\lb \nabla f_{i}(\xj_{k}) - \nabla f_{i}(\xj_{0}) - (\nabla f(\xj_{k}) - \nabla f(\xj_{0}))\rb.\]
Let $Z_{1}, \ldots, Z_{b_{j}}$ be i.i.d. random elements on $\X_{0}$ such that 
\[Z_{j}\stackrel{d}{=} -\frac{3\etaj}{2\bj}\lb\nabla f_{\tdi}(\xj_{k}) - \nabla f_{\tdi}(\xj_{0}) - (\nabla f(\xj_{k}) - \nabla f(\xj_{0}))\rb\]
where $\tdi$ is a uniform random variable in $\{1, \ldots, n\}$ that is independent of $\xj_{k}$. Since $G^{*}$ is convex and $\E Z_{j} = 0$, by Hoeffding's lemma (Proposition \ref{prop:hoeffding}),
\begin{equation}\label{eq:Mk12}
\E_{\sI_{k}}G^{*}\lb\frac{-3\etaj(\nuj_{k} - \nabla f(\xj_{k}) - \ej)}{2}\rb\le \E G^{*}\lb\sum_{j=1}^{\bj}Z_{j}\rb.
\end{equation}
Then by the property \textbf{C}3 of a CHEF,
\begin{equation}\label{eq:Mk13}
\E G^{*}\lb\sum_{j=1}^{\bj}Z_{j}\rb\le C_{G}\sum_{j=1}^{\bj}\E G^{*}(Z_{j})\le C_{G}\bj\E G^{*}(Z_{1}).
\end{equation}
Then we have
\begin{align}
G^{*}(Z_{1}) &\stackrel{(i)}{\le} \frac{1}{2}\lb G^{*}\lb\frac{3\etaj}{\bj} (\nabla f_{\tdi}(\xj_{k}) - \nabla f_{\tdi}(\xj_{0}))\rb + G^{*}\lb\frac{3\etaj}{\bj} (\nabla f(\xj_{k}) - \nabla f(\xj_{0}))\rb\rb\nonumber\\
& \stackrel{(ii)}{\le} \frac{3\etaj L}{2\bj}g^{-1}\lb\frac{3\etaj L}{\bj}\rb\lb G^{*}\lb\frac{\nabla f_{\tdi}(\xj_{k}) - \nabla f_{\tdi}(\xj_{0})}{L}\rb + G^{*}\lb\frac{\nabla f(\xj_{k}) - \nabla f(\xj_{0})}{L}\rb\rb\nonumber\\
& \stackrel{(iii)}{\le} \frac{3\etaj}{2\bj}g^{-1}\lb\frac{3\etaj L}{\bj}\rb \lb \BD{f_{\tdi}}(\xj_{0}, \xj_{k}) + \BD{f}(\xj_{0}, \xj_{k})\rb.\nonumber
\end{align}
where \emph{(i)} uses the convexity and the symmetry of $G^{*}$, \emph{(ii)} is implied by Lemma \ref{lem:CHEF} and \emph{(iii)} is implied by Lemma \ref{lem:general_cocoercive}. Taking an expectation over $\tdi$, we obtain that
\begin{equation}
  \label{eq:Mk14}
  \E G^{*}(Z_{1})\le \frac{3\etaj}{\bj}g^{-1}\lb\frac{3\etaj L}{\bj}\rb \BD{f}(\xj_{0}, \xj_{k}).
\end{equation}
Putting \eqref{eq:Mk11}- \eqref{eq:Mk14} together, we conclude that
\begin{equation}
  \label{eq:Mk1}
  \E_{\sI_{k}}\Mj_{k1}\le 3C_{G}g^{-1}\lb\frac{3\etaj L}{\bj}\rb \BD{f}(\xj_{0}, \xj_{k}).
\end{equation}

~\\
\noindent To bound $\E_{\sI_{k}}\Mj_{k2}$, by assumption \textbf{B}2 and \textbf{B}1,
\begin{align*}
  \la \nabla f(\xj_{k}), \xj_{k} - \xj_{k+1}\ra &= f(\xj_{k}) - f(\xj_{k+1}) + \BD{f}(\xj_{k+1}, \xj_{k})\\
&\le f(\xj_{k}) - f(\xj_{k+1}) + LG(\xj_{k+1} - \xj_{k})\\
& \le f(\xj_{k}) - f(\xj_{k+1}) + L\BD{w}(\xj_{k+1}, \xj_{k}).
\end{align*}
Thus,
\begin{align}
  \Mj_{k2}\le F(\xj_{k}) - F(\xj_{k+1}) - \lb\frac{1}{3\etaj} - L\rb \BD{w}(\xj_{k+1}, \xj_{k})\le F(\xj_{k}) - F(\xj_{k+1})\label{eq:Mk2}.
\end{align}

~\\
\noindent The proof is completed by combining \eqref{eq:Mk1} and \eqref{eq:Mk2}.
\end{proof}

The third lemma controls the magnitude of certain transformations of $\ej$ that will be useful in the following analysis. Note that in the unregularized case where $G(x) = \|x\|_{2}^{2} / 2$, $G^{*}(\beta\ej) = \beta^{2}\|\ej\|_{2}^{2} / 2$. Thus Lemma \ref{lem:ej} is essentially an analogue of Lemma \ref{lem:ejL2}.
\begin{lemma}\label{lem:ej}
Under assumptions \textbf{B}1 and \textbf{B}2, for any $\beta > 0$, 
\[\frac{1}{\beta}\E G^{*}(\beta \ej)\le C_{G}g^{-1}\lb\frac{3\beta L }{\Bj}\rb \bigg\{2\E (F(\td{x}_{j-1}) - F(x^{*})) + \DH\bigg\} I(\Bj < n).\]
\end{lemma}
\begin{proof}
Note that 
\[\ej = \frac{1}{\Bj}\sum_{i\in \Ij}(\nabla f_{i}(\td{x}_{j}) - \nabla f(\td{x}_{j})).\]
Clearly $\ej = 0$ when $\Bj = n$, and then 
\[\frac{1}{\beta}G^{*}(\beta \ej) = \frac{1}{\beta}G^{*}(0) \le 0.\]
Throughout the rest of the proof we assume that $\Bj < n$. Similar to the proof of Lemma \ref{lem:Mk}, let $Z_{1}, \ldots, Z_{\Bj}$ be i.i.d. random elements in $\X_{0}$ such that
\[Z_{j}\stackrel{d}{=}\frac{\beta}{\Bj}(\nabla f_{\tdi}(\td{x}_{j}) - \nabla f(\td{x}_{j})),\]
where $\tdi$ is a uniform random variable from $\{1, \ldots, n\}$. Then $\E Z_{j} = 0$ and hence by Hoeffding's lemma (Proposition \ref{prop:hoeffding}), 
  \begin{align}
    \E G^{*}(\beta \ej) &\le \E G^{*}\lb \sum_{j=1}^{\Bj}Z_{j}\rb.\label{eq:ej1}
  \end{align}
By the property \textbf{C}3 of a CHEF,
\begin{equation*}
\E G^{*}\lb \sum_{j=1}^{\Bj}Z_{j}\rb\le C_{G}\sum_{j=1}^{\Bj}\E G^{*}(Z_{j}) = C_{G}\Bj \E G^{*}(Z_{1}).
\end{equation*}
% Note that for any $x_{1}, x_{2}, x_{3}\in \X_{0}$, Jensen's inequality implies that
% \[G^{*}(x_{1} + x_{2} + x_{3}) \le \frac{1}{3}G^{*}\lb \rb\]
By definition,
\begin{align}
  \lefteqn{\E G^{*}(Z_{1})}\nonumber\\
  &= \E G^{*}\lb \frac{\beta}{\Bj}(\nabla f_{\tdi}(\td{x}_{j-1}) - \nabla f_{\tdi}(x^{*})) - \frac{\beta}{\Bj}(\nabla f(\td{x}_{j-1}) - \nabla f(x^{*})) + \frac{\beta}{\Bj}(\nabla f_{\tdi}(x^{*}) - \nabla f(x^{*}))\rb\nonumber\\
& \stackrel{(i)}{\le} \frac{1}{3}\bigg\{ \E G^{*}\lb\frac{3\beta}{\Bj} (\nabla f_{\tdi}(x^{*}) - \nabla f_{\tdi}(\td{x}_{j-1}))\rb + \E G^{*}\lb\frac{3\beta}{\Bj} (\nabla f(x^{*}) - \nabla f(\td{x}_{j-1}))\rb\nonumber\\
& \quad + \E G^{*}\lb\frac{3\beta}{\Bj} (\nabla f_{\tdi}(x^{*}) - \nabla f(x^{*}))\rb\bigg\}\nonumber\\
& \stackrel{(ii)}{\le} \frac{\beta L}{\Bj}g^{-1}\lb\frac{3\beta L }{\Bj}\rb \bigg\{ \E G^{*}\lb\frac{\nabla f_{\tdi}(x^{*}) - \nabla f_{\tdi}(\td{x}_{j-1})}{L}\rb + \E G^{*}\lb\frac{\nabla f(x^{*}) - \nabla f(\td{x}_{j-1})}{L}\rb\nonumber\\
& \qquad\qquad\qquad\quad + \E G^{*}\lb\frac{\nabla f_{\tdi}(x^{*}) - \nabla f(x^{*})}{L}\rb\bigg\}\nonumber\\
& \stackrel{(iii)}{\le} \frac{\beta}{\Bj}g^{-1}\lb\frac{3\beta L }{\Bj}\rb \bigg\{ \E \BD{f_{\tdi}}(\td{x}_{j-1}, x^{*}) + \E \BD{f}(\td{x}_{j-1}, x^{*}) + L\E G^{*}\lb\frac{\nabla f_{\tdi}(x^{*}) - \nabla f(x^{*})}{L}\rb\bigg\}\nonumber\\
& \stackrel{(iv)}{=}  \frac{\beta}{\Bj}g^{-1}\lb\frac{3\beta L }{\Bj}\rb \bigg\{2\E \BD{f}(\td{x}_{j-1}, x^{*}) + \DH\bigg\},\nonumber
\end{align}
where \emph{(i)} uses the convexity and the symmetry of $G^{*}$, \emph{(ii)} uses Lemma \ref{lem:CHEF} and \emph{(iii)} uses Lemma \ref{lem:general_cocoercive}. Finally, note that there exists $\xi\in \partial \psi(x^{*})$ such that
\[\nabla f(x^{*}) + \xi = 0.\]
By convexity of $\psi$, $\la \xi, \td{x}_{j-1} - x^{*}\ra\le \psi(\td{x}_{j-1}) - \psi(x^{*})$. Thus, 
\begin{align}
  \lefteqn{\BD{f}(\td{x}_{j - 1}, x^{*}) = f(\td{x}_{j-1}) - f(x^{*}) - \la\nabla f(x^{*}), \td{x}_{j-1} - x^{*}\ra} \nonumber\\
 & = f(\td{x}_{j-1}) - f(x^{*}) + \la \xi, \td{x}_{j-1} - x^{*}\ra\le F(\td{x}_{j-1}) - F(x^{*}).\label{eq:ej4}
\end{align}
The proof is then completed by combining \eqref{eq:ej1}-\eqref{eq:ej4}.
\end{proof}

The fourth lemma is an analogue of Lemma \ref{lem:geom_finite_L2} which enables the geometrization in the following proofs. Again, the proof is relegated to the end of this subsection to avoid distraction.
\begin{lemma}\label{lem:geom_finite}
  Assume that
  \begin{equation}
    \label{eq:etajL_cond}
    6\etaj L\le \min\left\{1, \bj g\lb \frac{\Gamma_{j}}{3C_{G}}\rb\right\}.
  \end{equation}
Then for any $j$, 
\[\E \BD{w}(x^{*}, \td{x}_{j}) < \infty, \quad \E (F(\td{x}_{j}) - F(x^{*})) < \infty, \quad \E |\la\ej, \td{x}_{j} - x^{*}\ra| < \infty.\]
\end{lemma}

The fifth lemma connects the iterates $\td{x}_{j}$ and $\td{x}_{j-1}$ in adjacent epochs. As with Lemma \ref{lem:dual}, the proof exploits the elegant property of geometrization. 
\begin{lemma}\label{lem:iter_exp}
Assume that $\etaj$ satisfies \ref{eq:etajL_cond}. Let $u\in \R^{d}$ be any variable that is independent of $\Ij$ and subsequent random subsets $(\sI_{k})_{k\ge 1}$ within the $j$-th epoch. Then under assumptions \textbf{B}1 and \textbf{B}2,
\begin{align*}
  &\E\la \nabla f(\td{x}_{j}), \td{x}_{j} - u \ra + \E (\psi(\td{x}_{j}) - \psi(u)) \le \frac{\bj}{\etaj\mj}\E \lb \BD{w}(u, \td{x}_{j-1}) - \BD{w}(u, \td{x}_{j})\rb\\
  & + \lb 1 + \frac{\bj}{\mj}\rb\E \la \ej, \td{x}_{j-1} - \td{x}_{j}\ra + \frac{\bj}{\mj}\E \lb F(\td{x}_{j-1}) - F(\td{x}_{j})\rb + 3C_{G}g^{-1}\lb\frac{3\etaj L}{\bj}\rb\E \BD{f}(\td{x}_{j-1}, \td{x}_{j}).
\end{align*}
\end{lemma}
\begin{proof}
  By Lemma \ref{lem:iter} and Lemma \ref{lem:Mk}, 
  \begin{align*}
  &\E_{\sI_{k}}\la \nuj_{k}, \xj_{k} - u\ra + \psi(\xj_{k}) - \psi(u) \le \frac{1}{\etaj}\E_{\sI_{k}}\lb \BD{w}(u, \xj_{k}) - \BD{w}(u, \xj_{k+1})\rb \\
& + 3C_{G}g^{-1}\lb\frac{3\etaj L}{\bj}\rb\BD{f}(\xj_{0}, \xj_{k}) + F(\xj_{k}) - \E_{\sI_{k}}F(\xj_{k+1}) + \E_{\sI_{k}}\la \ej, \xj_{k} - \xj_{k + 1}\ra.
  \end{align*}
 Note that
\[\E_{\sI_{k}} \nuj_{k} = \nabla f(\xj_{k}) + \ej.\]
Since $\xj_{k}$ and $u$ are independent of $\sI_{k}$, the first term on the left-hand side is 
\[\E\la \nuj_{k}, \xj_{k} - u\ra = \E\E_{\sI_{k}}\la \nuj_{k}, \xj_{k} - u\ra = \E\la\nabla f(\xj_{k}), \xj_{k} - u\ra + \E\la \ej, \xj_{k} - u\ra.\]
Letting $\E_{j}$ denote the expectation over $\Ij$ and $(\sI_{k})_{k\ge 0}$. Then
\[\E_{j}\la \ej, \xj_{k} - u\ra = \E_{j}\la \ej, \xj_{k} - \xj_{0}\ra + \E_{j}\la \ej, \xj_{0}\ra.\]
Note that $\Ij$ is independent of $\xj_{0}$, 
\[\E_{j}\la \ej, \xj_{0}\ra = \la \E_{j}\ej, \xj_{0}\ra = 0.\]
Thus, 
\begin{equation*}
  \E\la \ej, \xj_{k} - u\ra = \E \E_{j}\la \ej, \xj_{k} - u\ra = \E \la \ej, \xj_{k} - \xj_{0}\ra.
\end{equation*}
Letting $k = \Nj$, we obtain 
\begin{equation}
  \label{eq:iter_exp1}
  \E\la \nuj_{k}, \xj_{k} - u\ra = \E \la \nabla f(\td{x}_{j}), \td{x}_{j} - u\ra + \E \la\ej, \td{x}_{j} - \td{x}_{j-1}\ra.
\end{equation}

On the other hand, letting $k = \Nj$, by Lemma \ref{lem:geom} with $D_{k} = \BD{w}(u, \xj_{k})$ and Lemma \ref{lem:geom_finite},
\begin{align*}
  \E\lb \BD{w}(u, \xj_{\Nj}) - \BD{w}(u, \xj_{\Nj+1})\rb &= \frac{\bj}{\mj}\E \lb\BD{w}(u, \xj_{0}) - \BD{w}(u, \xj_{\Nj})\rb.
\end{align*}
Note that $\xj_{0} = \td{x}_{j-1}$, $\xj_{\Nj} = \td{x}_{j}$, we have
\begin{equation*}
\E\lb \BD{w}(u, \xj_{\Nj}) - \BD{w}(u, \xj_{\Nj+1})\rb = \frac{\bj}{\mj}\E \lb\BD{w}(u, \td{x}_{j-1}) - \BD{w}(u, \td{x}_{j})\rb.
\end{equation*}
Similarly, by Lemma \ref{lem:geom} and Lemma \ref{lem:geom_finite},
\begin{equation*}
\E \lb F(\xj_{\Nj}) - F(\xj_{\Nj + 1})\rb = \frac{\bj}{\mj}\E \lb F(\td{x}_{j-1}) - F(\td{x}_{j})\rb,
\end{equation*}
and 
\begin{equation}\label{eq:iter_exp4}
\E \la \ej, \xj_{\Nj} - \xj_{\Nj+1}\ra = \E \lb\la \ej, \xj_{\Nj}\ra - \la \ej, \xj_{\Nj + 1}\ra\rb = \frac{\bj}{\mj}\E \la \ej, \td{x}_{j-1} - \td{x}_{j}\ra.
\end{equation}
The proof is completed by putting \eqref{eq:iter_exp1}-\eqref{eq:iter_exp4} together. 
\end{proof}

The last lemma bounds the non-standard term $\E \la \ej, \td{x}_{j-1} - \td{x}_{j}\ra$ as in Lemma \ref{lem:xjxj-1L2}. Similarly, this lemma is a direct consequence of Lemma \ref{lem:iter_exp}.
\begin{lemma}\label{lem:xjxj-1}
Assume that $\etaj$ satisfies \ref{eq:etajL_cond}. Under assumptions \textbf{B}1 and \textbf{B}2, for any $\gamma_{j} > 0$,
\begin{align*}
  \E \la \ej, \td{x}_{j-1} - \td{x}_{j}\ra &\le \gamma_{j}\E \lb F(\td{x}_{j-1}) - F(\td{x}_{j})\rb + \frac{3\bj\gamma_{j}}{2\etaj (\mj + \bj)}G^{*}\lb \frac{(1 + \gamma_{j})\etaj(\mj + \bj)\ej}{\gamma_{j}\bj}\rb\\
& - \frac{\mj \gamma_{j}}{\mj + \bj}\lb 1 - 3C_{G}g^{-1}\lb\frac{3\etaj L}{\bj}\rb\rb \E \BD{f}(\td{x}_{j-1}, \td{x}_{j}).
\end{align*}
\end{lemma}
\begin{proof}
  By Lemma \ref{lem:iter_exp} with $u = \td{x}_{j-1}$,
  \begin{align*}
    &\frac{\bj}{\etaj\mj}\E \BD{w}(\td{x}_{j-1}, \td{x}_{j})\le \lb 1 + \frac{\bj}{\mj}\rb \E \la\ej, \td{x}_{j-1} - \td{x}_{j}\ra + \frac{\bj}{\mj}\E \lb F(\td{x}_{j-1}) - F(\td{x}_{j})\rb \\
& \quad + \E \lb \psi(\td{x}_{j-1}) - \psi(\td{x}_{j})\rb + 3C_{G}g^{-1}\lb\frac{3\etaj L}{\bj}\rb\E \BD{f}(\td{x}_{j-1}, \td{x}_{j}) + \E \la\nabla f(\td{x}_{j}), \td{x}_{j-1} - \td{x}_{j}\ra\\
& = \lb 1 + \frac{\bj}{\mj}\rb \E \la\ej, \td{x}_{j-1} - \td{x}_{j}\ra + \frac{\bj}{\mj}\E \lb F(\td{x}_{j-1}) - F(\td{x}_{j})\rb  + \E \lb \psi(\td{x}_{j-1}) - \psi(\td{x}_{j})\rb\\
& \quad + 3C_{G}g^{-1}\lb\frac{3\etaj L}{\bj}\rb\E \BD{f}(\td{x}_{j-1}, \td{x}_{j}) + \E (f(\td{x}_{j-1}) - f(\td{x}_{j})) - \E \BD{f}(\td{x}_{j-1}, \td{x}_{j})\\
& = \lb 1 + \frac{\bj}{\mj}\rb\left\{ \E \la\ej, \td{x}_{j-1} - \td{x}_{j}\ra +\E \lb F(\td{x}_{j-1}) - F(\td{x}_{j})\rb\right\}\\
&\quad  - \lb 1 - 3C_{G}g^{-1}\lb\frac{3\etaj L}{\bj}\rb\rb\E \BD{f}(\td{x}_{j-1}, \td{x}_{j}).
  \end{align*}
  This simplifies into 
  \begin{align}
    &\frac{\bj}{\etaj (\mj + \bj)}\E \BD{w}(\td{x}_{j-1}, \td{x}_{j})\le \E \la\ej, \td{x}_{j-1} - \td{x}_{j}\ra +\E \lb F(\td{x}_{j-1}) - F(\td{x}_{j})\rb\nonumber\\
& \qquad - \frac{\mj}{\mj + \bj}\lb 1 - 3C_{G}g^{-1}\lb\frac{3\etaj L}{\bj}\rb\rb \E \BD{f}(\td{x}_{j-1}, \td{x}_{j}).    \label{eq:xjxj-11}
  \end{align}
Fix $\gamma_{j} > 0$. By Lemma \ref{lem:bdfy} with $u = \ej, z = \td{x}_{j-1} - \td{x}_{j}, y = \td{x}_{j}, \alpha = \frac{\bj}{\etaj (\mj + \bj)}\frac{\gamma_{j}}{1 + \gamma_{j}}$, 
\begin{align}
  \lefteqn{\E \la\ej, \td{x}_{j-1} - \td{x}_{j}\ra} \nonumber\\ 
& \le \frac{\bj}{\etaj (\mj + \bj)}\frac{\gamma_{j}}{1 + \gamma_{j}}\lb \BD{w}(\td{x}_{j-1}, \td{x}_{j}) + \BD{w^{*}}\lb \frac{(1 + \gamma_{j})\etaj(\mj + \bj)\ej}{\gamma_{j}\bj} + \nabla w(\td{x}_{j}), \nabla w(\td{x}_{j})\rb\rb. \nonumber
\end{align}
By Lemma \ref{lem:dual_BD} with $h = w$ and $H = G$, 
\begin{align}
  \lefteqn{\frac{\bj}{\etaj (\mj + \bj)} \E \BD{w}(\td{x}_{j-1}, \td{x}_{j})} \nonumber\\
& \ge \frac{1 + \gamma_{j}}{\gamma_{j}}\E \la\ej, \td{x}_{j-1} - \td{x}_{j}\ra - \frac{3\bj}{2\etaj (\mj + \bj)}G^{*}\lb \frac{(1 + \gamma_{j})\etaj(\mj + \bj)\ej}{\gamma_{j}\bj}\rb.  \label{eq:xjxj-12}
\end{align}
The proof is completed by combining \eqref{eq:xjxj-11} and \eqref{eq:xjxj-12}.
\end{proof}

\begin{proof}[\textbf{Proof of Lemma \ref{lem:geom_finite}}]
  We prove the first two claims by induction. When $j = 0$, the claim is obvious. Suppose we prove the claim for the case of $j - 1$, i.e. 
\[\E \BD{w}(x^{*}, \td{x}_{j-1}) < \infty, \quad \E (F(\td{x}_{j-1}) - F(x^{*})) < \infty.\]
Let $\Mj_{k}$ be defined in \eqref{eq:Mjk}. Similar to the proof of Lemma \ref{lem:Mk}, we decompose it into three terms:
\[\Mj_{k} = \tdMj_{k1} + \tdMj_{k2} + \tdMj_{k3},\]
where
\begin{align*}
  &\tdMj_{k1} = \Mj_{k1} = \la \nuj_{k} - \nabla f(\xj_{k}) - \ej, \xj_{k} - \xj_{k + 1}\ra - \frac{2}{3\etaj}\BD{w}(\xj_{k+1}, \xj_{k}),\\
  & \tdMj_{k2} = \la \nabla f(\xj_{k}), \xj_{k} - \xj_{k+1}\ra + \psi(\xj_{k}) - \psi(\xj_{k+1}) - \frac{1}{6\etaj}\BD{w}(\xj_{k+1}, \xj_{k}),\\
  & \tdMj_{k3} = \la \ej, \xj_{k} - \xj_{k+1}\ra - \frac{1}{6\etaj}\BD{w}(\xj_{k+1}, \xj_{k}).
\end{align*}
By \eqref{eq:Mk11}, we have
\begin{align*}
  \Mj_{k1} & \le \frac{1}{\etaj}G^{*} \lb\frac{-3\etaj(\nuj_{k} - \nabla f(\xj_{k}) - \ej)}{2}\rb\\
& = \frac{1}{\etaj}G^{*} \lb-\frac{3\etaj}{2\bj}\sum_{i\in \sI_{k}}\lb \nabla f_{i}(\xj_{k}) - \nabla f_{i}(\xj_{0}) - (\nabla f(\xj_{k}) - \nabla f(\xj_{0}))\rb\rb.
\end{align*}
Using the convexity and the symmetry of $G^{*}$, we have
\begin{align*}
  \Mj_{k1} & \le \frac{1}{2\etaj}\bigg\{G^{*} \lb\frac{3\etaj}{\bj}\sum_{i\in \sI_{k}}\lb \nabla f_{i}(x^{*}) - \nabla f_{i}(\xj_{k}) - (\nabla f(x^{*}) - \nabla f(\xj_{k}))\rb \rb\\
& \quad + G^{*} \lb\frac{3\etaj}{\bj}\sum_{i\in \sI_{k}}\lb \nabla f_{i}(x^{*}) - \nabla f_{i}(\xj_{0}) -  (\nabla f(x^{*}) - \nabla f(\xj_{0}))\rb\rb\bigg\}.
\end{align*}
Using the same arguments as \eqref{eq:Mk12} and \eqref{eq:Mk13},
\begin{align*}
  \lefteqn{\E_{\sI_{k}}G^{*} \lb\frac{3\etaj}{\bj}\sum_{i\in \sI_{k}}\lb \nabla f_{i}(x^{*}) - f_{i}(\xj_{k}) - (\nabla f(x^{*}) - \nabla f(\xj_{k}) )\rb \rb}\\
& \le C_{G}\bj \E_{\tdi} G^{*}\lb\frac{3\etaj}{\bj}\lb \nabla f_{\tdi}(x^{*}) - \nabla f_{\tdi}(\xj_{k}) - (\nabla f(x^{*}) - \nabla f(\xj_{k}))\rb\rb\\
& \le \frac{C_{G}\bj}{2}\E_{\tdi}\bigg\{G^{*}\lb\frac{6\etaj}{\bj}\lb \nabla f_{\tdi}(x^{*}) - \nabla f_{\tdi}(\xj_{k})\rb\rb + G^{*}\lb \frac{6\etaj}{\bj}\lb \nabla f(x^{*}) - \nabla f(\xj_{k})\rb\rb\bigg\}\\
& \stackrel{(i)}{\le} 3C_{G}\etaj L g^{-1}\lb\frac{6\etaj L}{\bj}\rb\E_{\tdi}\bigg\{G^{*}\lb\frac{\nabla f_{\tdi}(x^{*}) - \nabla f_{\tdi}(\xj_{k})}{L}\rb + G^{*}\lb \frac{\nabla f(x^{*}) - \nabla f(\xj_{k})}{L}\rb\bigg\}\\
& \stackrel{(ii)}{\le}3C_{G}\etaj g^{-1}\lb\frac{6\etaj L}{\bj}\rb\E_{\tdi}\{\BD{f_{\tdi}}(\xj_{k}, x^{*}) + \BD{f}(\xj_{k}, x^{*})\}\\
& = 6C_{G}\etaj g^{-1}\lb\frac{6\etaj L}{\bj}\rb\BD{f}(\xj_{k},x^{*}),
\end{align*}
where \emph{(i)} is implied by Lemma \ref{lem:CHEF} and \emph{(ii)} is implied by Lemma \ref{lem:general_cocoercive}. Analogously, 
\begin{align*}
  &\E_{\sI_{k}}G^{*} \lb\frac{3\etaj}{\bj}\sum_{i\in \sI_{k}}\lb \nabla f_{i}(x^{*}) - f_{i}(\xj_{0}) - (\nabla f(x^{*}) - \nabla f(\xj_{0}) )\rb \rb\\
& \,\,\le 6C_{G}\etaj g^{-1}\lb\frac{6\etaj L}{\bj}\rb\BD{f}(\xj_{0}, x^{*}).
\end{align*}
Thus,
\begin{align}
  \E_{\sI_{k}}\tdMj_{k1} &\le 3C_{G} g^{-1}\lb\frac{6\etaj L}{\bj}\rb\lb \BD{f}(\xj_{k}, x^{*}) + \BD{f}(\xj_{0}, x^{*})\rb.  \label{eq:tdMk1}
\end{align}
 Using the same argument as \eqref{eq:Mk2} and the assumption that $\etaj L \le 1 / 6$, we have
\begin{align}
  \tdMj_{k2}\le F(\xj_{k}) - F(\xj_{k+1}) - \lb\frac{1}{6\etaj} - L\rb \BD{w}(\xj_{k+1}, \xj_{k})\le F(\xj_{k}) - F(\xj_{k+1})\label{eq:tdMk2}
\end{align}
By Fenchel-Young inequality, for any convex function $H$,
\[\la\ej, \xj_{k} - \xj_{k+1}\ra\le H^{*}(\ej) + H(\xj_{k} - \xj_{k+1}).\]
Let $H(z) = G(z) / 6\etaj$, then $H^{*}(z) = G^{*}(6\etaj z) / 6\etaj$ and thus,
\[\la\ej, \xj_{k} - \xj_{k+1}\ra\le \frac{1}{6\etaj}G^{*}(6\etaj\ej) + \frac{1}{6\etaj}G(\xj_{k} - \xj_{k+1})\le \frac{1}{6\etaj}G^{*}(6\etaj\ej) + \frac{1}{6\etaj}\BD{w}(\xj_{k},  \xj_{k+1}).\]
By Lemma \ref{lem:ej},
\begin{equation}
  \label{eq:tdMk3}
  \E\tdMj_{k3}\le \frac{1}{6\etaj}\E G^{*}(6\etaj\ej)\le C_{G}g^{-1}\lb\frac{18\etaj L}{\Bj}\rb \bigg\{2\E (F(\xj_{0}) - F(x^{*})) + \DH\bigg\}.
\end{equation}
Putting \eqref{eq:tdMk1} - \eqref{eq:tdMk3} together and using the monotonicity of $g^{-1}$, we conclude that
\begin{align}
  \E\tdMj_{k} & \le \E F(\xj_{k}) - \E F(\xj_{k+1}) + 3C_{G} g^{-1}\lb\frac{6\etaj L}{\bj}\rb\lb \E\BD{f}(\xj_{k}, x^{*}) + \E\BD{f}(\xj_{0}, x^{*})\rb\nonumber\\
& \quad + C_{G}g^{-1}\lb\frac{18\etaj L}{\Bj}\rb \bigg\{2\E (F(\xj_{0}) - F(x^{*})) + \DH\bigg\}\nonumber\\
& \le \E F(\xj_{k}) - \E F(\xj_{k+1}) + \E\BD{f}(\xj_{k}, x^{*}) \nonumber\\
& \quad + C_{G}g^{-1}(18\etaj L)\bigg\{6\E\BD{f}(\xj_{0}, x^{*}) + 2\E (F(\xj_{0}) - F(x^{*})) + \DH \bigg\},\nonumber
\end{align}
where the last inequality uses the condition on $\etaj L$ that
\[3C_{G} g^{-1}\lb\frac{6\etaj L}{\bj}\rb\le \Gamma \le 1.\]
Similar to \eqref{eq:ej4}, we have
\[\E\BD{f}(\xj_{k}, x^{*})\le \E (F(\xj_{k}) - F(x^{*})), \quad \E\BD{f}(\xj_{0}, x^{*})\le \E (F(\xj_{0}) - F(x^{*})).\]
Thus, 
\begin{equation}
  \label{eq:tdMk}
  \E\tdMj_{k} \le \E F(\xj_{k}) - \E F(\xj_{k+1}) + \E (F(\xj_{k}) - F(x^{*})) + \td{M}_{0},
\end{equation}
where
\begin{equation*}
  \td{M}_{0} = C_{G}g^{-1}(18\etaj L)\bigg\{8\E (F(\xj_{0}) - F(x^{*})) + \DH \bigg\}
\end{equation*}

~\\
\noindent Next, by Lemma \ref{lem:iter} with $u = x^{*}$ and by \eqref{eq:tdMk}
\begin{align*}
  &\E \la \nuj_{k}, \xj_{k} - x^{*}\ra + \E\psi(\xj_{k}) - \E\psi(x^{*})\le \frac{1}{\etaj}\lb \E\BD{w}(x^{*}, \xj_{k}) - \E\BD{w}(x^{*}, \xj_{k+1})\rb\\
& \quad + \E F(\xj_{k}) - \E F(\xj_{k+1}) + \E (F(\xj_{k}) - F(x^{*})) + \td{M}_{0}.
\end{align*}
In addition,
\begin{align*}
  &\E \la \nuj_{k}, \xj_{k} - x^{*}\ra = \E \E_{\sI_{k}}\la \nuj_{k}, \xj_{k} - x^{*}\ra\\
&  = \E \la\nabla f(\xj_{k}), \xj_{k} - x^{*}\ra + \E \la \ej, \xj_{k} - x^{*}\ra\\
& \ge \E f(\xj_{k}) - f(x^{*}) + \E \la \ej, \xj_{k} - x^{*}\ra.
\end{align*}
The above two inequalities imply that 
\begin{align}
  &\frac{1}{\etaj}\E \BD{w}(x^{*}, \xj_{k+1}) + \E (F(\xj_{k+1}) - F(x^{*}))\nonumber\\
\le & \frac{1}{\etaj}\E \BD{w}(x^{*}, \xj_{k}) + \E (F(\xj_{k}) - F(x^{*})) + \E \la \ej, x^{*} - \xj_{k}\ra + \td{M}_{0}. \label{eq:geom_finite_1}
\end{align}
By Fenchel-Young inequality, for any constant $c_{j} > 0$,
\[\la \ej, x^{*} - \xj_{k}\ra \le \frac{c_{j}}{\etaj}G^{*}\lb\frac{\etaj}{c_{j}}\ej\rb + \frac{c_{j}}{\etaj}G(x^{*} - \xj_{k})\le \frac{c_{j}}{\etaj}G^{*}\lb\frac{\etaj}{c_{j}}\ej\rb + \frac{c_{j}}{\etaj}\BD{w}(x^{*}, \xj_{k}).\]
By \eqref{eq:geom_finite_1} and Lemma \ref{lem:ej},
\begin{align}
  \lefteqn{\frac{1}{\etaj}\E \BD{w}(x^{*}, \xj_{k+1}) + \E (F(\xj_{k+1}) - F(x^{*}))}\nonumber\\
&\le  \frac{1 + c_{j}}{\etaj}\E \BD{w}(x^{*}, \xj_{k}) + \E (F(\xj_{k}) - F(x^{*})) + \td{M}_{0}' + \td{M}_{0}\label{eq:geom_finite_2}, 
\end{align}
where
\[\td{M}_{0}' = C_{G}g^{-1}\lb \frac{3\etaj L}{c_{j}B_{j}}\rb\bigg\{2\E (F(\td{x}_{j-1}) - F(x^{*})) + \DH\bigg\}.\]
Let $c_{j} = 1 / 2\mj$. Then it is easy to see that $\td{M}_{0}'\le \td{M}_{0}$. Further let 
\[\Qj_{k} = \frac{1}{\etaj}\E \BD{w}(x^{*}, \xj_{k}) + \E (F(\xj_{k}) - F(x^{*})).\]
Then \eqref{eq:geom_finite_2} implies that
\begin{align*}
  \Qj_{k+1}\le \lb 1 + \frac{1}{2\mj}\rb\Qj_{k} + 2\td{M}_{0}&\Longrightarrow \Qj_{k+1} + 4\mj \td{M}_{0}\le \lb 1 + \frac{1}{2\mj}\rb \lb \Qj_{k}+ 4\mj\td{M}_{0}\rb\\
& \Longrightarrow \Qj_{k}\le \lb 1 + \frac{1}{2\mj}\rb^{k}\lb \Qj_{0}+ 4\mj\td{M}_{0}\rb.
\end{align*}
Therefore, 
\[\E \Qj_{\Nj} \le \E \lb 1 + \frac{1}{2\mj}\rb^{\Nj}\lb \Qj_{0}+ 4\mj\td{M}_{0}\rb.\]
Recalling that 
\[P(\Nj = k) = \frac{\bj}{\mj + \bj}\lb \frac{\mj}{\mj + \bj}\rb^{k}\le \lb \frac{\mj}{\mj + 1}\rb^{k}.\]
Thus, 
\[\E \lb 1 + \frac{1}{2\mj}\rb^{\Nj} \le \sum_{k\ge 0}\lb\frac{2\mj + 1}{2\mj}\frac{\mj}{\mj + 1}\rb^{k} = \sum_{k\ge 0}\lb\frac{2\mj + 1}{2\mj + 2}\rb^{k} \le 2\mj + 2.\]
As a result,
\[\E \Qj_{\Nj} \le (2\mj + 2)\lb \Qj_{0}+ 4\mj\td{M}_{0}\rb.\]
By the induction hypothesis, we know that 
\[\Qj_{0} + 4\mj\td{M}_{0} < \infty\]
and thus,
\[\E \Qj_{\Nj} = \frac{1}{\etaj}\E \BD{w}(x^{*}, \td{x}_{j}) + \E (F(\td{x}_{j}) - F(x^{*})) < \infty.\]
To prove the last claim, by Fenchel-Young inequality, we have
\[|\la\ej, \td{x}_{j} - x^{*}\ra|\le G^{*}(\ej) + G(\td{x}_{j} - x^{*})\le G^{*}(\ej) + \BD{w}(x^{*}, \td{x}_{j}).\]
The finiteness of the expectation is then followed by the first two claims and Lemma \ref{lem:ej}.
\end{proof}

\subsection{One-epoch analysis}
First we prove a more general version of Theorem \ref{thm:one_epoch_simplify}. 
\begin{theorem}\label{thm:one_epoch}
Fix any $\Gamma_{j}\in (0, 1)$. Assume that
\[6\etaj L\le \min \left\{1, \bj g\lb \frac{\Gamma_{j}}{3C_{G}}\rb\right\}.\]
Then under assumptions \textbf{B}1 and \textbf{B}2,
  \begin{align*}
    \lefteqn{\lb 1 + \frac{\bj}{\mj}\rb \E (F(\td{x}_{j}) - F(x^{*})) + \frac{\bj (1 - \Gamma_{j})}{\etaj \mj}\E \BD{w}(x^{*}, \td{x}_{j})} \\
& \le \lb \Gamma_{j} + \frac{\bj}{\mj} + 3a_{j}\lb 1 + \frac{\bj}{\mj}\rb\rb \E (F(\td{x}_{j-1}) - F(x^{*})) + \frac{\bj (1 - \Gamma_{j})}{\etaj \mj}\E \BD{w}(x^{*}, \td{x}_{j-1})\\
& \quad  + \frac{3a_{j}}{2}\lb 1 + \frac{\bj}{\mj}\rb \DH I(\Bj < n),
  \end{align*}
where 
\[a_{j} = C_{G}g^{-1}\lb\frac{3\etaj L (\mj + \bj)}{\Bj \bj \Gamma_{j}}\rb. \]
\end{theorem}
\begin{proof}
Let $\gamma_{j} = \Gamma_{j} / (1 - \Gamma_{j})$ (and hence $\Gamma_{j} = \gamma_{j} / (1 + \gamma_{j})$). For convenience, let 
\[\beta_{j} = \frac{\etaj (\mj + \bj)}{\bj \Gamma_{j}} = \frac{(1 + \gamma_{j})\etaj (\mj + \bj)}{\gamma_{j}\bj}.\]
By Lemma \ref{lem:ej},
\[\frac{1}{\beta_{j}}\E G^{*}(\beta_{j}\ej)\le 2a_{j}\E (F(\td{x}_{j-1}) - F(x^{*})) + a_{j}\DH I(\Bj < n).\]  
By Lemma \ref{lem:xjxj-1},
\begin{align}
& \E \la \ej, \td{x}_{j-1} - \td{x}_{j}\ra \le \gamma_{j}\E \lb F(\td{x}_{j-1}) - F(\td{x}_{j})\rb + 3a_{j}(1 + \gamma_{j})\E (F(\td{x}_{j-1}) - F(x^{*}))\nonumber\\
&\qquad \qquad + \frac{3}{2}a_{j}(1 + \gamma_{j})\DH I(\Bj < n) - \frac{\mj \gamma_{j}}{\mj + \bj}\lb 1 - 3C_{G}g^{-1}\lb\frac{3\etaj L}{\bj}\rb\rb \E \BD{f}(\td{x}_{j-1}, \td{x}_{j}).\label{eq:xjxj-1main}
\end{align}
By Lemma \ref{lem:iter_exp} with $u = x^{*}$,
\begin{align*}
  &\E\la \nabla f(\td{x}_{j}), \td{x}_{j} - x^{*} \ra + \E (\psi(\td{x}_{j}) - \psi(x^{*})) \le \frac{\bj}{\etaj\mj}\E \lb \BD{w}(x^{*}, \td{x}_{j-1}) - \BD{w}(x^{*}, \td{x}_{j})\rb\\
  & + \lb 1 + \frac{\bj}{\mj}\rb\E \la \ej, \td{x}_{j-1} - \td{x}_{j}\ra + \frac{\bj}{\mj}\E \lb F(\td{x}_{j-1}) - F(\td{x}_{j})\rb + 3C_{G}g^{-1}\lb\frac{3\etaj L}{\bj}\rb\E \BD{f}(\td{x}_{j-1}, \td{x}_{j}).\\
\stackrel{(i)}{\le} & \frac{\bj}{\etaj\mj}\E \lb \BD{w}(x^{*}, \td{x}_{j-1}) - \BD{w}(x^{*}, \td{x}_{j})\rb + \lb\gamma_{j}\lb 1 + \frac{\bj}{\mj}\rb + \frac{\bj}{\mj}\rb\E \lb F(\td{x}_{j-1}) - F(\td{x}_{j})\rb\\
&  + 3a_{j}(1 + \gamma_{j})\lb 1 + \frac{\bj}{\mj}\rb\E (F(\td{x}_{j-1}) - F(x^{*})) + \frac{3}{2}a_{j}(1 + \gamma_{j})\lb 1 + \frac{\bj}{\mj}\rb \DH I(\Bj < n)\\
& - \lb \gamma_{j} - (1 + \gamma_{j})3C_{G}g^{-1}\lb\frac{3\etaj L}{\bj}\rb\rb\E \BD{f}(\td{x}_{j-1}, \td{x}_{j})\\
\stackrel{(ii)}{\le} & \frac{\bj}{\etaj\mj}\E \lb \BD{w}(x^{*}, \td{x}_{j-1}) - \BD{w}(x^{*}, \td{x}_{j})\rb + \lb\gamma_{j}\lb 1 + \frac{\bj}{\mj}\rb + \frac{\bj}{\mj}\rb\E \lb F(\td{x}_{j-1}) - F(\td{x}_{j})\rb\\
&  + 3a_{j}(1 + \gamma_{j})\lb 1 + \frac{\bj}{\mj}\rb\E (F(\td{x}_{j-1}) - F(x^{*})) + \frac{3}{2}a_{j}(1 + \gamma_{j})\lb 1 + \frac{\bj}{\mj}\rb \DH I(\Bj < n),
\end{align*}
where \emph{(i)} uses \eqref{eq:xjxj-1main} and \emph{(ii)} uses the condition that
\[\gamma_{j} - (1 + \gamma_{j})3C_{G}g^{-1}\lb\frac{3\etaj L}{\bj}\rb = 3C_{G}(1 + \gamma_{j})\lb\frac{\Gamma_{j}}{3C_{G}} - g^{-1}\lb\frac{3\etaj L}{\bj}\rb\rb\le 0.\]
By convexity of $f$, we have
\[\la\nabla f(\td{x}_{j}), \td{x}_{j} - x^{*} \ra + \psi(\td{x}_{j}) - \psi(x^{*}) \ge f(\td{x}_{j}) - f(x^{*}) + \psi(\td{x}_{j}) - \psi(x^{*}) = F(\td{x}_{j}) - F(x^{*}).\]
Therefore, we have
\begin{align*}
\lefteqn{\E (F(\td{x}_{j}) - F(x^{*}))} \\
& \le \frac{\bj}{\etaj\mj}\E \lb \BD{w}(x^{*}, \td{x}_{j-1}) - \BD{w}(x^{*}, \td{x}_{j})\rb + \lb\gamma_{j}\lb 1 + \frac{\bj}{\mj}\rb + \frac{\bj}{\mj}\rb\E \lb F(\td{x}_{j-1}) - F(\td{x}_{j})\rb\\
& \quad + 3a_{j}(1 + \gamma_{j})\lb 1 + \frac{\bj}{\mj}\rb\E (F(\td{x}_{j-1}) - F(x^{*})) + \frac{3}{2}a_{j}(1 + \gamma_{j})\lb 1 + \frac{\bj}{\mj}\rb \DH I(\Bj < n).  
\end{align*}
Rearanging the terms yields
\begin{align*}
  \lefteqn{\lb 1 + \gamma_{j}\rb\lb 1 + \frac{\bj}{\mj}\rb\E (F(\td{x}_{j}) - F(x^{*})) + \frac{\bj}{\etaj\mj}\E \BD{w}(x^{*}, \td{x}_{j})} \\
& \le \lb\gamma_{j} + (1 + \gamma_{j})\frac{\bj}{\mj} + 3a_{j}(1 + \gamma_{j})\lb 1 + \frac{\bj}{\mj}\rb\rb\E (F(\td{x}_{j-1}) - F(x^{*})) + \frac{\bj}{\etaj\mj}\E \BD{w}(x^{*}, \td{x}_{j - 1})\\
& \quad + \frac{3}{2}a_{j}(1 + \gamma_{j})\lb 1 + \frac{\bj}{\mj}\rb \DH I(\Bj < n).
\end{align*}
The proof is completed by dividing both sides by $1 + \gamma_{j} = \frac{1}{1 - \Gamma_{j}}$.
\end{proof}

\begin{proof}[\textbf{Proof of Theorem \ref{thm:one_epoch_simplify}}]
Let 
\[\Gamma_{j}\equiv \Gamma = \frac{1}{4\alpha^{1/\xi}}.\]
Since $6\etaj L\le \bj g\lb \frac{\Gamma_{j}}{3C_{G}}\rb$, $\mj\ge 4b$ and $g^{-1}$ is increasing,
\[a_{j} \le C_{G}g^{-1}\lb g\lb \frac{\Gamma}{3C_{G}}\rb \frac{\mj + b}{2\Gamma\Bj}\rb\le C_{G}g^{-1}\lb g\lb \frac{\Gamma}{3C_{G}}\rb \frac{5\mj}{8\Gamma\Bj}\rb.\]
By definition,
\[\Bj \ge \frac{5}{8\Gamma}g\lb\frac{\Gamma}{3C_{G}}\rb \frac{\mj}{g(1 / \mj)}.\]
Thus, 
\begin{equation*}
  a_{j} \le \frac{C_{G}}{\mj}.
\end{equation*}
By assumption \textbf{B}3, 
\[\frac{1}{4}\E (F(\td{x}_{j}) - F(x^{*}))\ge \frac{\mu}{4}\E \BD{w}(x^{*}, \td{x}_{j}).\]  
By Theorem \ref{thm:one_epoch},
\begin{align}
      \lefteqn{\lb \frac{3}{4} + \frac{b}{\mj}\rb \E (F(\td{x}_{j}) - F(x^{*})) + \lb\frac{\mu}{4} + \frac{b (1 - \Gamma)}{\eta \mj}\rb\E \BD{w}(x^{*}, \td{x}_{j})} \nonumber\\
& \le \lb \Gamma + \frac{b}{\mj} + \frac{3C_{G}}{\mj}\lb 1 + \frac{b}{\mj}\rb\rb \E (F(\td{x}_{j-1}) - F(x^{*})) + \frac{b (1 - \Gamma)}{\eta\mj }\E \BD{w}(x^{*}, \td{x}_{j-1})\nonumber\\
& \quad  + \frac{3C_{G}}{2 \mj}\lb 1 + \frac{b}{\mj}\rb \DH I(\Bj < n).\label{eq:main1}
\end{align}
Since $\mj \ge 4\alpha^{1/\xi}\max\{b, 4C_{G}\}$,
\[\frac{b}{\mj}\le \frac{1}{4\alpha^{1/\xi}} = \Gamma,\]
and
\[\frac{3C_{G}}{\mj}\lb 1 + \frac{b}{\mj}\rb = \frac{3C_{G}}{m_{0}\alpha^{j}}\lb 1 + \frac{b}{\mj}\rb\le \frac{3\Gamma}{4\alpha^{j}}\lb 1  + \frac{1}{4}\rb\le \frac{\Gamma}{\alpha^{j}}\le \Gamma.\]
Therefore, \eqref{eq:main1} implies that
\begin{align}
      \lefteqn{\frac{3}{4}\E (F(\td{x}_{j}) - F(x^{*})) + \lb\frac{\mu}{4} + \frac{b (1 - \Gamma)}{\eta \mj}\rb\E \BD{w}(x^{*}, \td{x}_{j})} \nonumber\\
& \le 3\Gamma \E (F(\td{x}_{j-1}) - F(x^{*})) + \frac{b (1 - \Gamma)}{\eta\mj}\E \BD{w}(x^{*}, \td{x}_{j-1}) + \frac{\Gamma}{2\alpha^{j}}\DH I(\Bj < n)\nonumber\\
& \le 3\Gamma \E (F(\td{x}_{j-1}) - F(x^{*})) + \frac{b (1 - \Gamma)}{\eta\mj}\E \BD{w}(x^{*}, \td{x}_{j-1}) + \frac{\DH}{8\alpha^{j}} I(\Bj < n), \nonumber
\end{align}
where the last line uses the fact that $\Gamma \le \frac{1}{4}$. The proof is completed by multiplying both sides by $\frac{4}{3}$.  
\end{proof}

\subsection{Multi-epoch analysis}
\begin{proof}[\textbf{Proof of Theorem \ref{thm:main}}]
Let 
\[\Lj = \E (F(\td{x}_{j}) - F(x^{*})) + \lb \frac{4b (1 - \Gamma)}{3\eta}\frac{1}{\mj} + \frac{\mu}{3}\rb\E \BD{w}(x^{*}, \td{x}_{j}).\]
Then Theorem \eqref{thm:one_epoch_simplify} implies that
\begin{align}
  \Lj &\le 4\Gamma \E (F(\td{x}_{j-1}) - F(x^{*})) + \frac{4b (1 - \Gamma)}{3\eta}\frac{1}{\mj}\E \BD{w}(x^{*}, \td{x}_{j-1}) + \frac{\DH}{6\alpha^{j}} I(\Bj < n)\nonumber\\
& \le \max\left\{4\Gamma, \frac{\frac{4b (1 - \Gamma)}{3\eta}\frac{1}{\mj}}{\frac{4b (1 - \Gamma)}{3\eta}\frac{1}{\mjj} + \frac{\mu}{3}}\right\}\Ljj + \frac{\DH}{6\alpha^{j}} I(\Bj < n)\nonumber\\
& = \max\left\{\frac{1}{\alpha^{1/\xi}}, \frac{1}{\alpha + \frac{m_{0}\eta\mu }{4b(1 - \Gamma)}\alpha^{j}}\right\}\Ljj + \frac{\DH}{6\alpha^{j}} I(\Bj < n)\nonumber\\
% & = \max\left\{\frac{1}{\alpha^{1/\xi}}, \frac{1}{\alpha + \frac{m_{0}\eta\mu}{b(1 - \Gamma)}\alpha^{j + 1/\xi}}\right\}\Ljj + \frac{\DH}{6\alpha^{j}} I(\Bj < n)\nonumber\\
% & \stackrel{(i)}{\le} \max\left\{\frac{1}{\alpha^{1/\xi}}, \frac{1}{\alpha + \frac{\eta\mu (1 - \Gamma)}{\Gamma}\alpha^{j + 1/\xi}}\right\}\Ljj + \frac{\DH}{6\alpha^{j}} I(\Bj < n)\nonumber\\
& \stackrel{(i)}{\le} \max\left\{\frac{1}{\alpha^{1/\xi}}, \frac{1}{\alpha + \eta\mu \alpha^{j + 1/\xi}}\right\}\Ljj + \frac{\DH}{6\alpha^{j}} I(\Bj < n),\label{eq:main3}
\end{align}
where \emph{(i)} uses the fact that 
\[\frac{m_{0}}{4b(1 - \Gamma)}\ge \frac{m_{0}}{4b}\ge \alpha^{1/\xi}.\]
% $1 / 4\Gamma = \alpha^{1/\xi}$ and (ii) uses the fact that $\Gamma \le 1 / 4 < 1 / 2$. 
For any $j\ge 0$,
\[\max\left\{\frac{1}{\alpha^{1/\xi}}, \frac{1}{\alpha + \eta\mu \alpha^{j + 1/\xi}}\right\}\le \max\left\{\frac{1}{\alpha^{1/\xi}}, \frac{1}{\alpha}\right\}\le \frac{1}{\alpha}.\]
When $j\ge T_{\kappa}^{*}$, $\alpha^{j}\ge \kappa / \eta L = 1 / \eta \mu$ and thus
\[\max\left\{\frac{1}{\alpha^{1/\xi}}, \frac{1}{\alpha + \eta\mu \alpha^{j + 1/\xi}}\right\}\le \frac{1}{\alpha^{1/\xi}}.\]
In summary, 
\begin{equation*}
  \max\left\{\frac{1}{\alpha^{1/\xi}}, \frac{1}{\alpha + \eta\mu \alpha^{j + 1/\xi}}\right\}\le\lambda_{j}^{-1}.
\end{equation*}
Plugging this into \eqref{eq:main3}, we conclude that
\begin{equation*}
  \Lj\le \lambda_{j}^{-1}\Ljj + \frac{\DH}{6\alpha^{j}}I(\Bj < n).
\end{equation*}
Using the same argument as in the proof of Theorem \ref{thm:mainL2}, we can prove that
\begin{equation*}
\L_{T} \le \Lambda_{T}^{-1} \L_{0} + \td{\Lambda}_{T}^{-1}\frac{\DH (T\wedge T_{n}^{*})}{6}.
\end{equation*}

% ~\\
% \noindent Finally we prove the following statement by induction.
% \begin{equation}\label{eq:LT}
% \L_{T} \le \Lambda_{T}^{-1} \L_{0} + \td{\Lambda}_{T}^{-1}\frac{\DH (T\wedge T_{n}^{*})}{6}.
% \end{equation}
% It is obvious that \eqref{eq:induction} holds for $T = 0$. Suppose it holds for $T - 1$, then by \eqref{eq:main5},
% \begin{align*}
%   &\L_{T}\le \lambda_{T}^{-1} \L_{T-1} + \frac{\DH}{6\alpha^{T}}I(B_{T} < n)\\
% \le &\lambda_{T}^{-1}\lb \Lambda_{T-1}^{-1} \L_{0} + \td{\Lambda}_{T-1}^{-1}\frac{\DH ((T - 1)\wedge T_{n}^{*})}{6}\rb + \frac{\DH}{6\alpha^{T}}I(B_{T} < n)\\
%   = & \Lambda_{T}^{-1}\L_{0} + \frac{\DH}{6}\lb \lambda_{T}^{-1}\td{\Lambda}_{T-1}^{-1} ((T - 1)\wedge T_{n}^{*}) + \alpha^{-T}I(B_{T} < n)\rb.
% \end{align*}
% If $T < T_{n}^{*}$, then $\td{\Lambda}_{T} = \alpha^{-T}, \td{\Lambda}_{T-1} = \alpha^{-(T - 1)}$. Since $\lambda_{T}\ge \alpha$, we have
% \[\lambda_{T}^{-1}\td{\Lambda}_{T-1}^{-1} ((T - 1)\wedge T_{n}^{*}) + \alpha^{-T}I(B_{T} < n) \le \alpha^{-T}(T - 1) + \alpha^{-T} = \alpha^{-T}T = \td{\Lambda}_{T}^{-1}(T \wedge T_{n}^{*}).\]
% If $T > T_{n}^{*}$, then $\lambda_{T} = \td{\lambda}_{T}$ and thus
% \[\lambda_{T}^{-1}\td{\Lambda}_{T-1}^{-1} ((T - 1)\wedge T_{n}^{*}) + \alpha^{-T}I(B_{T} < n) = \td{\lambda}_{T}^{-1}\td{\Lambda}_{T-1}^{-1} ((T - 1)\wedge T_{n}^{*}) \le \td{\Lambda}_{T}^{-1}(T \wedge T_{n}^{*}).\]
% Therefore, \eqref{eq:LT} is proved. 

\noindent The proof is then completed by noting that
\[\L_{T}\ge \E (F(\td{x}_{T}) - F(x^{*}))\]
and 
\begin{align*}
  \L_{0} & = \DF + \lb\frac{\mu}{3} + \frac{4b (1 - \Gamma)}{3\eta}\frac{1}{m_{0}}\rb\E \BD{w}(x^{*}, \td{x}_{0})\\
& \stackrel{(i)}{\le} \DF + \lb\frac{1}{3} + \frac{4(1 - \Gamma)\Gamma}{3\eta L}\rb \Dx\\
& \stackrel{(ii)}{\le} \DF + \lb\frac{1}{18\eta L} + \frac{1}{4\eta L}\rb\Dx = \DF + \frac{\Dx}{3\eta L},
\end{align*}
where \emph{(i)} uses the fact that $\mu \le L$ and the condition $b / m_{0}\le \Gamma$ and \emph{(ii)} uses the fact that $\eta L \le 1 / 6$ and $\Gamma \le 1 / 4$.
\end{proof}

\subsection{Complexity analysis}
\begin{proof}[\textbf{Proof of Theorem \ref{thm:complexity}}]
Let
\begin{equation*}
  T^{(1)}(\eps) = \min\left\{T: \Lambda_{T}\ge \frac{\Dx + \DF}{\eps \eta L}\right\}, \quad T^{(2)}(\eps) = \min\left\{T: \td{\Lambda}_{T}\ge \frac{\DH T_{n}^{*}}{\eps}\right\}.
\end{equation*}
Then for any $T \ge \max\{T^{(1)}(\eps), T^{(2)}(\eps)\}$,
\[\E (f(\td{x}_{T}) - f(x^{*}))\le \eps \eta L \frac{\DF + \frac{1}{3\eta L}\Dx}{\DF + \Dx} + \frac{\eps}{6}\le \frac{1}{3}\eps + \frac{1}{6}\eps \le \eps.\]
where we use the fact that $\eta L \le 1 / 6\le 1 / 3$. This entails that
\begin{equation*}
  T(\eps)\le \max\{T^{(1)}(\eps), T^{(2)}(\eps)\}.
\end{equation*}

The same arguments as in the proof of Theorem \ref{thm:complexityL2} imply that
\[  \alpha^{T(\eps)} = \td{O}\lb\min\left\{\frac{\D}{\eps}, \kappa\lb\frac{\Dx + \DF}{\eps \kappa}\rb^{\xi}_{*} + \frac{\DH}{\eps}, \td{\kappa}\lb\frac{\D}{\eps \td{\kappa}}\rb^{\xi}_{*}\right\}\rb.\]
On the other hand, since $g$ is increasing, whenever $j \le T_{n}^{*}$,
\begin{equation*}
\frac{\Bj}{B_{j-1}} = \frac{\mj}{\mjj}\frac{g(1/\mjj)}{g(1 / \mj)} \ge \frac{\mj}{mjj} = \alpha.
\end{equation*}
Thus, if $T(\eps) \le T_{n}^{*}$, 
\[\sum_{j=1}^{T(\eps)}\Bj \le B_{T(\eps)}\sum_{j=1}^{T(\eps)}\alpha^{-j} = O(B_{T(\eps)}) = O\lb\frac{\alpha^{T(\eps)}}{g(\alpha^{-T(\eps)})}\rb. \]
If $T(\eps) > T_{n}^{*}$, since $\Bj \le n$,
\[\sum_{j=1}^{T(\eps)}\Bj \le n T(\eps).\]
Putting the two pieces together, we obtain that
\begin{equation*}
  \sum_{j=1}^{T(\eps)}\Bj = O\lb \frac{\alpha^{T(\eps)}}{g(\alpha^{-T(\eps)})} \wedge nT(\eps)\rb.
\end{equation*}
Finally, by \eqref{eq:generic_complexity}, we conclude that
\begin{equation*}
  \E\comp(\eps) = O\lb \alpha^{T(\eps)} + \frac{\alpha^{T(\eps)}}{g(\alpha^{-T(\eps)})} \wedge nT(\eps)\rb = O\lb \frac{\alpha^{T(\eps)}}{g(\alpha^{-T(\eps)})} \wedge (\alpha^{T(\eps)} + nT(\eps))\rb.
\end{equation*}
The proof is then completed by replacing $\alpha^{T(\eps)}$ ith $A(\eps)$. 
\end{proof}

\section{Miscellaneous}\label{app:miscellaneous}
\begin{lemma}[Theorem 2.1.5 of \cite{nesterov04}.]\label{lem:cocoercive}
Let $g$ be a convex function that satisfies the assumption $\textbf{A}1$,
\begin{align*}
\|\nabla g(x) - \nabla g(y)\|_{2}^{2}&\le 2L(\breg{g}{x}{y}).
\end{align*}
\end{lemma}

\begin{lemma}[Lemma A.1 of \cite{lei17}]\label{lem:var_sampling}
Let $z_{1}, \ldots, z_{M}\in \R^{d}$ be an arbitrary population and $\mathcal{J}$ be a uniform random subset of $[M]$ with size $m$. Then 
\[\E \left\|\frac{1}{m}\sum_{j\in \mathcal{J}}z_{j} - \frac{1}{M}\sum_{j=1}^{M}z_{j}\right\|_{2}^{2}\le \frac{I(m < M)}{m}\cdot \frac{1}{M}\sum_{j=1}^{M}\|z_{j}\|_{2}^{2}.\]
\end{lemma}

\begin{lemma}\label{lem:dual_Gq}
Let $\|\cdot\|$ be an arbitrary norm on a Hilbert space and let $\|\cdot\|_{*}$ be its dual norm. Letting $G_{q}(x) = \|x\|^{q} / q$, we have
\[G_{q}^{*}(x) = \frac{1}{p}\|x\|_{*}^{p},\]
where $p = q / (q - 1)$.  
\end{lemma}
\begin{proof}
By definition of $\|\cdot\|_{*}$,
\[\|x\|_{*} = \sup_{y: \|y\| = 1}\la x, y\ra.\]
Thus, 
\[\la x, y\ra\le \|x\|_{*}\|y\|.\]
By the Fenchel-Young inequality, for any nonnegative numbers $a, b$
\[ab \le \frac{a^{p}}{p} + \frac{b^{q}}{q}.\]
Let $a = \|x\|_{*}, b = \|y\|$, 
\[\|x\|_{*}\|y\|\le \frac{1}{p}\|x\|_{*}^{p} + \frac{1}{q}\|y\|^{q} = \frac{1}{p}\|x\|_{*}^{p} + G_{q}(y).\]
By definition,
\[G_{q}^{*}(x) = \sup_{y\in \X_{0}}\la x, y\ra - G_{q}(y).\]
Thus, we have proved that
\[G_{q}^{*}(x)\le \frac{1}{p}\|x\|_{*}^{p}.\]
To show the equality, it remains to find a $y(x)$ for any $x$ that achieves the equality. Let 
\[y(x) = \|x\|_{*}^{p - 1}\argmax_{y: \|y\| = 1}\la x, y\ra.\]
Then $\|y(x)\| = \|x\|_{*}^{p-1}$ and 
\[\la x, y(x)\ra = \|x\|_{*}^{p - 1}\sup_{y: \|y\| = 1}\la x, y\ra = \|x\|_{*}^{p}.\]
Therefore, 
\[\la x, y(x)\ra - G_{q}(y(x)) = \|x\|_{*}^{p} - \frac{1}{q}\|x\|_{*}^{(p - 1)q} = \frac{1}{p}\|x\|_{*}^{p}.\]

The proof is completed by putting the two pieces together.
\end{proof}

\begin{proposition}[Theorem 4, \cite{hoeffding63}]\label{prop:hoeffding}
Let $C = (c_{1}, \ldots, c_{n})$ be a finite population of random elements. Let $Z_{1}, \ldots, Z_{m}$ be a random sample with replacement from $C$ and $W_{1}, \ldots, W_{m}$ be a random sample without replacement from $C$. If the function $f(x)$ is continuous and convex, then
\[\E f\lb\sum_{i=1}^{m}Z_{i}\rb\ge \E f\lb\sum_{i=1}^{m}W_{i}\rb.\]
\end{proposition}

\begin{proposition}\label{prop:Hanner}
Let $\X = \R^{d}$ and $G(x) = \|x\|_{r}^{q} / q$ with $r\le q$. Then property \textbf{C}3 holds with a constant that only depends on $(d, r, q)$.
\end{proposition}
\begin{proof}
By Lemma \ref{lem:dual_Gq}, 
\[G^{*}(x) = \|x\|_{r'}^{p} / p,\]
where $r' = r / (r - 1)$ and $p = q / (q - 1)$. Let $\rho_{r}(t)$ denote the modulus of smoothness:
\[\rho_{r}(t) = \sup_{\|x\|_{r} = \|y\|_{r} = 1}\left\{\frac{\|x + ty\|_{r} + \|x - ty\|_{r}}{2} - 1\right\}.\]
Let $\mathcal{F}_{n}$ denote the sigma field generated by $\{Z_{1}, \ldots, Z_{n}\}$ for $n\le m$ and $\mathcal{F}_{n} = \mathcal{F}_{m}$ for all $n > m$. Further let
\[X_{0} = 0, \quad X_{n} = \left\{
    \begin{array}{ll}
      X_{n - 1} + Z_{n} & (n \le m)\\
      X_{n - 1} & (n > m)
    \end{array}\right..\]
Clearly $\{X_{n}\}_{n\ge 0}$ is a martingale with respect to the filtration $(\mathcal{F}_{n})_{n\ge 0}$. By Proposition 2.4 (b) of \cite{pisier75}, it is left to show that 
\begin{equation}\label{eq:p_smooth}
\rho_{r'}(t)\le At^{p},
\end{equation}
for some constant $A > 0$. By Hanner's inequality \citep{hanner56},
\[\lb\|x + ty\|_{r'} + \|x - ty\|_{r'}\rb^{r'}\le \|2x\|_{r'}^{r'} + \|2ty\|_{r'}^{r'} = 2^{r'}(1 + t^{r'}).\]
This implies that
\[\rho_{r'}(t) \le (1 + t^{r'})^{1/r'} - 1.\]
Note that $r\le q$ implies that $r' \ge p$. First we notice that
\[\lim_{t\rightarrow 0}\frac{(1 + t^{r'})^{1/r'} - 1}{t^{p}} = \lim_{t\rightarrow 0}\frac{t^{r'}}{r' t^{p}} = \left\{\begin{array}{ll}
0 & (r' > p)\\
1 / r' & (r' = p)
\end{array}\right..\]
Thus there exists a constant $c$ such that for any $t \le c$, 
\[(1 + t^{r'})^{1/r'} \le 1 + t^{p}.\]
When $t > c$, 
\[(1 + t^{r'})^{1/r'}\le \lb \frac{t^{r'}}{c^{r'}} + t^{r'}\rb^{1/r'} = t\lb c^{-r'} + 1\rb^{1/r'}\le t^{p} c^{-(p - 1)}\lb c^{-r'} + 1\rb^{1/r'},\]
where the last inequality uses the fact that $p\ge 1$. Putting the pieces together, we prove \eqref{eq:p_smooth} with
\[A = 1 + c^{-(p - 1)}\lb c^{-r'} + 1\rb^{1/r'}.\]
\end{proof}

\section{Empirical Demonstration}\label{app:empirical}

Although our paper is focusing on the provable adaptivity of SCSG which is not easy to be illustrated numerically, it is worth demonstrating the empirical performance of SCSG as a sanity check. We consider multi-class classification problems on five datasets from UCI machine learning repository (\url{https://archive.ics.uci.edu/ml/index.php}, \cite{UCI})  -- \texttt{adult}\footnote{Data source: https://archive.ics.uci.edu/ml/datasets/Adult ; contributed by \cite{adult}.}, \texttt{connect}\footnote{Data source: https://archive.ics.uci.edu/ml/datasets/Connect-4; contributed by John Tromp.}, \texttt{credit}\footnote{Data source: https://archive.ics.uci.edu/ml/datasets/default+of+credit+card+clients ; contributed by \cite{credit}.}, \texttt{crowdsource}\footnote{Data source: https://archive.ics.uci.edu/ml/datasets/Crowdsourced+Mapping ; contributed by \cite{crowdsource}.} and \texttt{sensor}\footnote{Data source: https://archive.ics.uci.edu/ml/datasets/Gas+Sensor+Array+Drift+Dataset+at\\+Different+Concentrations ; contributed by \cite{sensor1} and \cite{sensor2}.}. We also include the well-known MNIST dataset\footnote{Data source: http://yann.lecun.com/exdb/mnist/ ; contributed by \cite{lecun1998gradient}.} For each dataset, we fit a multi-class logistic regression with individual loss
\[f_{i}(x) = \log\lb 1 + \sum_{k=1}^{K - 1}e^{a_{i}^{T}x_{k}}\rb - \sum_{k=1}^{K - 1}I(y_{i} = k)a_{i}^{T}x_{k}\]
where $K$ denotes the number of classes and $x\in \R^{p\times (K-1)}$ denotes the concatenation of $x_{1}, x_{2}, \ldots, x_{K-1}$. By convention, we add an $L_{2}$ regularizer $(1/n)\|x\|_{2}^{2}$ on the loss function to guarantee the existence of (finite) optimizers, i.e.
\[F(x) = \frac{1}{n}\sum_{i=1}^{n}f_{i}(x) + \frac{1}{n}\|x\|_{2}^{2}.\]
Although $F$ is guaranteed to be $1/n$ strongly convex, the parameter is too small to be useful for algorithms that takes it as an input. 

For each dataset, we optimize the objective function by the following six methods: 
\begin{enumerate}[Method 1.]
\item SCSG with $\alpha = 1.25, B_{0} = 0.001n, m_{0} = 0.005n$, $b = 0.0001n$, as suggested in Remark \ref{rem:practical}, and a constant step size $\eta$ (to be discussed below).
\item SVRG with inner loop size $m = 2n$ as suggested in \cite{SVRG}, mini-batch size $b = 0.0001n$ and a constant step size $\eta$;
\item SARAH \citep{nguyen2019finite} with inner loop size $m = 2n$, mini-batch size $b = 0.0001n$ and a constant step size $\eta$;
\item Katyusha$^{\mathrm{ns}}$, which works for non-strongly convex functions, with inner loop size $m = 2n$ (Algorithm 2 of \cite{Katyusha}). We choose option II to update $y_{k+1}$ (line 12 of Algorithm 2) because this yields better results than optionI. In addition, we find that the default option $\alpha_{s} = 1 / 3\tau_{1,s}L$ may be too conservative. To explore the best performance of Katyusha, we set it to be $\eta / \eta_{1,s}$ and find the best-tuned $\eta$. Indeed, the best-tuned $\eta$ is at least $1/L$ on all six datasets and is as large as $8/L$ for \texttt{mnist} and \texttt{sensor}. % \footnote{We did not consider Katyusha (Algorithm 1 of \cite{Katyusha}) because the default setting sets $\alpha = 1 / 3\tau_{1} L$ while requires $\tau_{1}\le 1/2$. This restricts $\alpha \le 2 / 3L$ which prevents effective tuning and indeed loses to Katyusha$^{\mathrm{ns}}$ with best-tuned $\eta$. Although we could reset $\alpha = \eta / \tau_{1}$ as what we did for Katyusha$^{\mathrm{ns}}$, it would require tuning two parameters so as to render it less attractive.}
\item SGD with a constant stepsize $\eta$ and mini-batch size $b = 0.0001n$;
\item SGD with a linearly decayed step size $\eta_{t} = \eta / (1 + t)$ and mini-batch size $b = 0.0001n$;
\item GD with a constant step size $\eta$.
\end{enumerate}

We run each algorithm and each parameter $\eta$ for $50$ effective passes of data and choose the step size that yields the smallest function value. In particular, we set 
\[\eta = \frac{c}{L}, \quad \mbox{where }c \in \{2^{k}: k = -10, -9, \ldots, 9, 10\}\]
where $L$ is estimated by 
\[L = \frac{1}{n}\sum_{i=1}^{n}L_{i}, \quad L_{i} = 2\|a_{i}\|^{2}\]
where $L_{i}$ is obtained from Proposition E.1 of \cite{SCSG}. Note that we compute the average smoothness as opposed to the worst-case smoothness in our theory because the former yields a better scaling for $c$ empirically. In principle we could set $L = \max_{i}L_{i}$ and change the scale of $c$ correspondingly. In addition, we remove 5\% outliers with largest $L_{i}$'s from each dataset to avoid unstable performance. We emphasize that outlier removal is helpful to stabilize all algorithms considered here. 

Finally, we estimate the true optimum $f(x^{*})$ by running SCSG with the best-tuned $\eta$ with 5000 effective passes of data. As a sanity check, we also estimate $f(x^{*})$ via SVRG
 with the best-tuned $\eta$ with 5000 effective passes of data and find that SCSG consistently yields better solution than SVRG for all datasets.

The programs to replicate all results are available at \url{https://github.com/lihualei71/ScsgAdaptivity}. The best-tuned results for each algorithm are displayed in Figure \ref{fig:experiment}. We find that the best-tuned step size does not lie in the boundary of the tuning set, namely $2^{-10}/L$ or $2^{10}/ L$, for all methods and all datasets. This indicates that we tune each method sufficiently. For each plot, the y-axis gives the ratio $(f(x) - f(x^{*})) / (f(x_0)-  f(x^{*}))$ in log-scales so that the initial value is always $10^{0}$. It is clear that SCSG performs well on all these datasets.

\begin{figure}
  \centering
  \includegraphics[width = 0.48\textwidth]{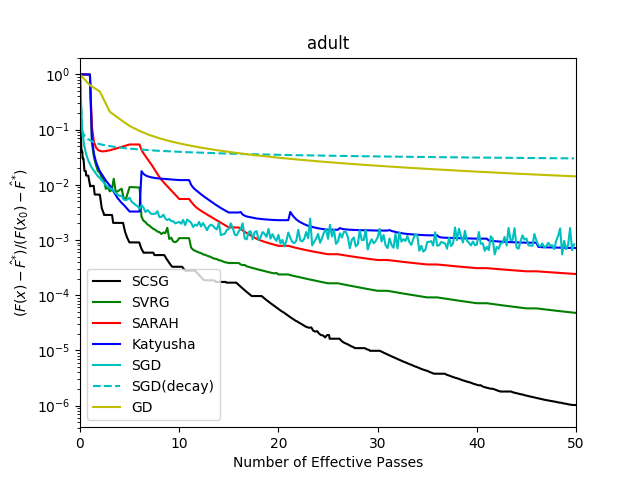}
  \includegraphics[width = 0.48\textwidth]{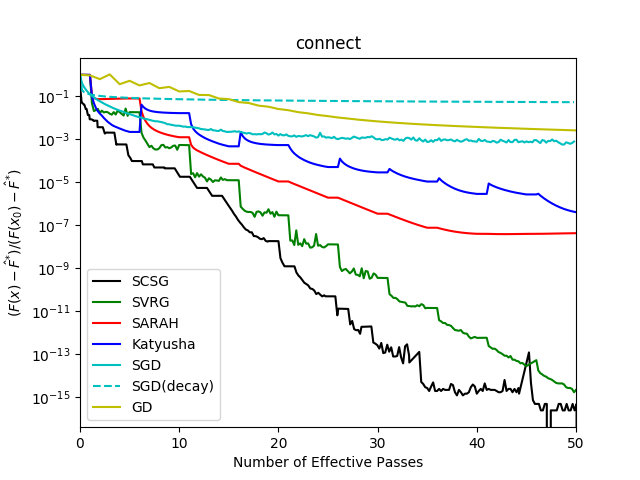}
  \includegraphics[width = 0.48\textwidth]{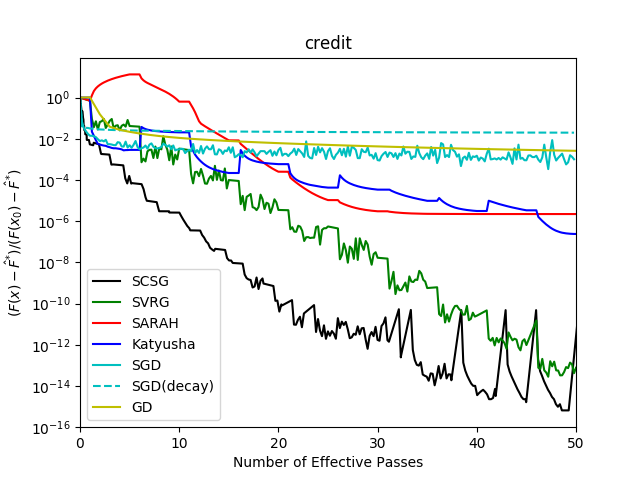}
  \includegraphics[width = 0.48\textwidth]{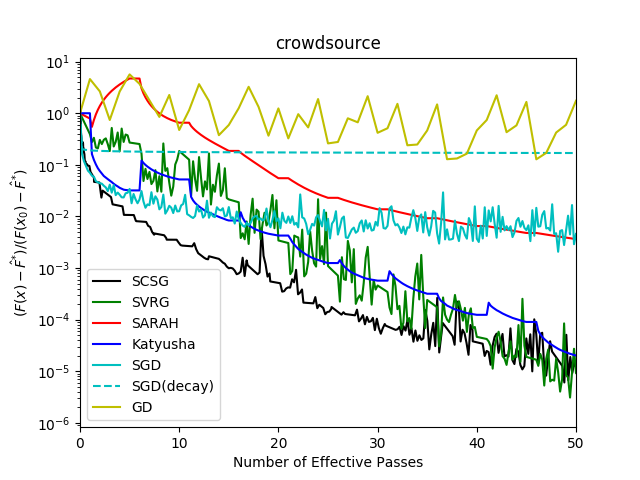}
  \includegraphics[width = 0.48\textwidth]{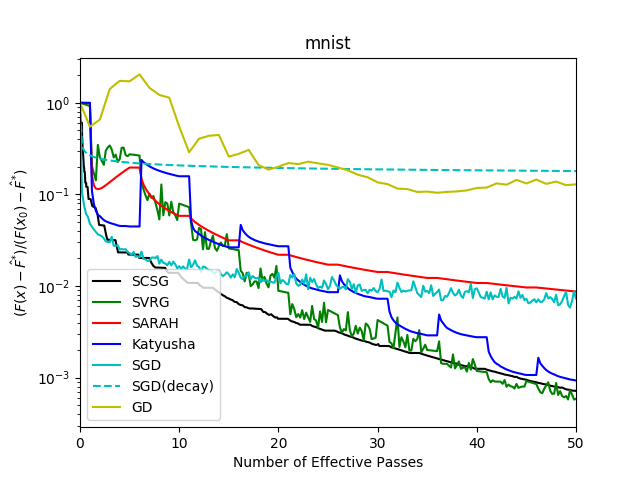}
  \includegraphics[width = 0.48\textwidth]{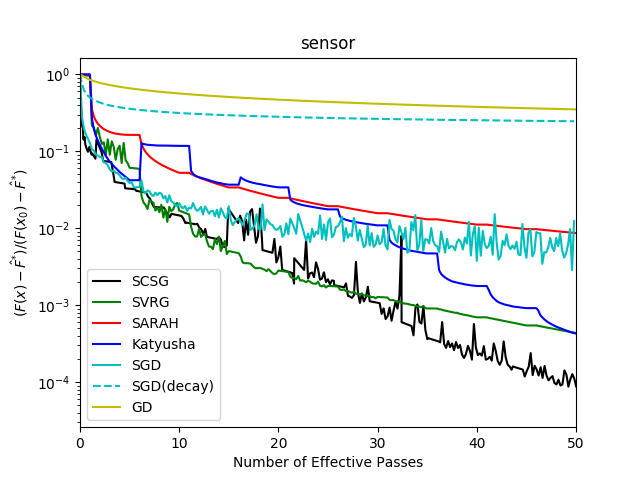}
  \caption{Experimental results on six datasets}\label{fig:experiment}
\end{figure}

\end{document}